\documentclass[reqno,10pt]{amsart}

\usepackage{amsmath,amssymb,mathrsfs,graphicx,enumitem,verbatim}

\usepackage[pdftex=true,
	bookmarks=true,
	bookmarksopen=true,
	bookmarksnumbered=true,
	pdftoolbar=true,
	pdfmenubar=true,
	pdffitwindow=false,
	pdfstartview={FitH},
	colorlinks=true,
	linkcolor=red,
	citecolor=blue,
	breaklinks=true]{hyperref}

\allowdisplaybreaks

\newlength{\enummargin}
\newcounter{CounterEnumi}

\numberwithin{equation}{section}

\newtheorem{thm}{Theorem}[section]
\newtheorem{thmi}{Theorem}
\newtheorem{prop}[thm]{Proposition}
\newtheorem{lemma}[thm]{Lemma}
\newtheorem{cor}[thm]{Corollary}
\newtheorem*{thm*}{Theorem}
\newtheorem*{prop*}{Proposition}
\newtheorem*{cor*}{Corollary}
\newtheorem*{conj*}{Conjecture}

\theoremstyle{definition}

\theoremstyle{remark}
\newtheorem{rmk}[thm]{Remark}

\newcommand{\itref}[1]{(\ref{#1})}

\newcommand{\mc}{\mathcal}

\newcommand{\cH}{\mc H}

\newcommand{\cL}{\mc L}

\newcommand{\cV}{\mc V}
\newcommand{\cX}{\mc X}
\newcommand{\cY}{\mc Y}
\newcommand{\cM}{\mc M}

\newcommand{\sfe}{\mathsf{e}}

\newcommand{\bB}{\mathbf{B}}

\newcommand{\ft}{\mathfrak{t}}

\newcommand{\R}{\mathbb R}
\newcommand{\RR}{\mathbb R}
\newcommand{\C}{\mathbb C}
\newcommand{\Cx}{\mathbb C}
\newcommand{\N}{\mathbb N}

\newcommand{\bhm}{M_\bullet}

\newcommand{\la}{\langle}
\newcommand{\ra}{\rangle}
\newcommand{\pa}{\partial}

\newcommand{\tn}{\textnormal}

\newcommand{\Sphere}{\mathbb S}
\newcommand{\ep}{\epsilon}
\newcommand{\eps}{\epsilon}

\newcommand{\Span}{\operatorname{Span}}
\newcommand{\supp}{\operatorname{supp}}

\newcommand{\wt}{\widetilde}
\newcommand{\wh}{\widehat}
\renewcommand{\tilde}{\wt}
\renewcommand{\hat}{\wh}

\renewcommand{\Im}{\operatorname{Im}}
\renewcommand{\Re}{\operatorname{Re}}

\newcommand{\Sym}{\operatorname{S}}
\newcommand{\thr}{\mathrm{th}}

\newcommand{\bl}{{\mathrm{b}}}
\newcommand{\cl}{{\mathrm{c}}}

\newcommand{\CI}{\mc C^\infty}

\newcommand{\CIc}{\mc C^\infty_\cl}
\newcommand{\CIdot}{\dot{\mc C}^\infty}
\newcommand{\CIdotc}{\CIdot_\cl}

\newcommand{\ftrans}{\;\!\widehat{\ }\;\!}

\newcommand{\Diff}{\mathrm{Diff}}
\newcommand{\WF}{\mathrm{WF}}
\newcommand{\wozero}{\setminus o}
\newcommand{\numin}{\nu_{\min}}

\newcommand{\semi}{\hbar}
\newcommand{\Psih}{\Psi_{\semi}}
\newcommand{\WFh}{\WF_{\semi}}

\DeclareMathOperator{\Op}{Op}
\newcommand{\CxTb}{{}^{\C\bl}T}
\newcommand{\Tb}{{}^{\bl}T}
\newcommand{\Lambdab}{{}^{\bl}\Lambda}
\newcommand{\rcTb}{{}^{\bl}\overline{T}}

\newcommand{\Sb}{{}^{\bl}S}
\newcommand{\bdiff}{{}^\bl\!d}
\newcommand{\Diffb}{\Diff_\bl}
\newcommand{\Hb}{H_{\bl}}

\newcommand{\WFb}{\WF_{\bl}}
\newcommand{\Psib}{\Psi_{\bl}}
\newcommand{\Vb}{\cV_\bl}
\newcommand{\Db}{{}^{\bl}\!D}

\newcommand{\fpsi}{\Psi}
\newcommand{\fpsib}{\Psi_\bl}

\newcommand{\Rnhalf}{{\overline{\R^n_+}}}

\newcommand{\derivsU}{10}
\newcommand{\derivsUdU}{12}

\begin{document}

\title[Quasilinear wave equations on Kerr-de Sitter]{Global analysis of quasilinear wave equations on asymptotically Kerr-de Sitter spaces}

\author{Peter Hintz and Andras Vasy}
\address{Department of Mathematics, Stanford University, CA 94305-2125, USA}
\email{phintz@math.stanford.edu}
\email{andras@math.stanford.edu}
\thanks{The authors were supported in part by A.V.'s National Science
  Foundation grants DMS-0801226 and DMS-1068742 and P.H.\ was
  supported in part by a Gerhard Casper Stanford Graduate Fellowship.}
\date{April 4, 2014. Final revision: July 12, 2015.}
\subjclass{35L72, 35L05, 35P25}
\keywords{Quasilinear waves,
  Kerr-de Sitter space,
  b-pseudodifferential operators, Nash-Moser iteration, resonances, asymptotic expansion}

\begin{abstract}
  We consider quasilinear wave equations on manifolds for which infinity has a structure generalizing that of Kerr-de Sitter space; in particular the trapped geodesics form a normally hyperbolic invariant manifold. We prove the global existence and decay, to constants for the actual wave equation, of solutions. The key new ingredient compared to earlier work by the authors in the semilinear case \cite{HintzVasySemilinear} and by the first author in the non-trapping quasilinear case \cite{HintzQuasilinearDS} is the use of the Nash-Moser iteration in our framework.
\end{abstract}

\maketitle

\section{Introduction}

We consider quasilinear wave equations on manifolds for which infinity has a structure generalizing that of Kerr-de Sitter space. An important feature is that, as in perturbations of Kerr-de Sitter space, the trapped geodesics form a normally hyperbolic invariant manifold. We prove the global existence and decay of solutions; this means decay to constants for the actual wave equation. This result is part of a new framework for solving quasilinear wave equations with normally hyperbolic trapping, which extends the semilinear framework developed by the two authors \cite{HintzVasySemilinear} and the {\em non-trapping} quasilinear theory developed by the first author \cite{HintzQuasilinearDS}. The main new tool introduced here is a Nash-Moser iteration necessitated by the loss of derivatives in the linear estimates at the normally hyperbolic trapping. To our knowledge, this is the first global result for the forward problem for a quasilinear wave equation on either a Kerr or a Kerr-de Sitter background. We remark, however, that Dafermos, Holzegel and Rodnianski \cite{Dafermos-Holzegel-Rodnianski:Scattering} have constructed backward solutions for Einstein's equations on the Kerr background; for backward constructions the trapping does not cause difficulties. For concreteness, we state our results first in the special case of Kerr-de Sitter space, but it is important to keep in mind that the setting is more general.

The region of Kerr-de Sitter space we are interested in is a (non-compact) 4-dimensional manifold
\[
  M^\circ = \R_{t_*} \times (r_--2\delta,r_++2\delta)_r \times \Sphere^2,
\]
which extends past the event horizon (at $r=r_-$) and past the cosmological horizon (at $r=r_+$). The manifold $M^\circ$ is equipped with a stationary Lorentzian metric $g_0$ (i.e.\ $\pa_{t_*}$ is a Killing vector field) which depends on three parameters, $\Lambda>0$ (the cosmological constant), $\bhm>0$ (the black hole mass) and $a$ (the angular momentum), though we usually drop this in the notation. We will always assume that $\Lambda$, $\bhm$ and $a$ are such that the non-degeneracy condition \cite[(6.2)]{VasyMicroKerrdS} holds, which in particular ensures that the cosmological horizon lies outside the black hole event horizon, i.e.\ $r_+>r_-$. Concretely, in Boyer-Lindquist coordinates $(t,r,\theta,\phi)$ on $\R_t\times(r_-,r_+)_r\times\Sphere^2$, the Kerr-de Sitter metric, see e.g.\ \cite[Equation~(6.1)]{VasyMicroKerrdS},\footnote{Our $t,\phi,t_*,\phi_*$ are denoted $\tilde t,\tilde\phi,t,\phi$ in \cite{VasyMicroKerrdS}.} does not extend to $r=r_\pm$ due to a coordinate singularity, but introducing $t_*=t+h(r)$, $\phi_*=\phi+P(r)$ with suitable functions $h,P$ as in \cite[Equation~(6.5)]{VasyMicroKerrdS}, the metric does extend smoothly to $r=r_\pm$ and beyond.

In order to set up our problem, see Figure~\ref{FigKDS} for an illustration, we consider the domain
\[
  \Omega^\circ = [0,\infty)_{t_*}\times[r_--\delta,r_++\delta]_r\times\Sphere^2 \subset M^\circ,
\]
which is a submanifold with corners with two boundary hypersurfaces, which are the intersections of
\[
  H_1 = \{t_*=0\}, \quad H_2 = \{r=r_--\delta\}\cup\{r=r_++\delta\}
\]
with $\Omega^\circ$. Thus, $H_1$ is a Cauchy hypersurface, and $H_2$ has two connected components, both spacelike hypersurfaces. We are interested in solving the forward problem for wave-like equations in $\Omega^\circ$, i.e.\ imposing vanishing Cauchy data at $H_1$; initial value problems with general Cauchy data can always be converted into an equation of this type.

\begin{figure}[!ht]
  \centering 
  \includegraphics{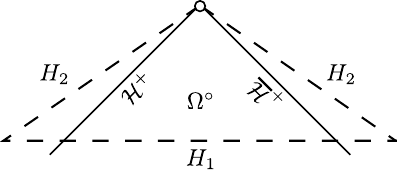}
  \caption{Penrose diagram of the domain $\Omega^\circ$, bounded by the dashed lines, on which we solve quasilinear wave equations. $\cH^+$ is the event horizon, $\overline{\cH}^+$ the cosmological horizon, further $H_1$ is a Cauchy hypersurface, where we impose vanishing Cauchy data, and $H_2$ has two spacelike connected components.}
\label{FigKDS}
\end{figure}

The simplest (albeit not the most natural from a geometric perspective, see below) wave equations we consider are of the form
\[
  \Box_{g(u,du)}u = f+q(u,du),
\]
$u$ real-valued, where $g(0,0)=g_0$ is the Kerr-de Sitter metric, and at each $p=(t_*,r,\omega)\in M^\circ$, the metric $g(u,du)$ is $g_p(u(p),du(p))$, where
\[
  g_p\colon\R\oplus T^*_p M \to S^2 T^*_p M
\]
depends smoothly on $p$, and in fact only depends on $(r,\omega)$, \emph{but not on $t_*$}; further
\[
  q(u,du) = \sum_{j=1}^{N'} a_j u^{e_j}\prod_{k=1}^{N_j} X_{jk}u, \quad e_j,N_j\in\N_0,\ N_j+e_j\geq 2,
\]
with $a_j\in\CI(M^\circ)$ and $X_{jk}\in\cV(M^\circ)$ real and independent of $t_*$. (Here $a_j$ is only relevant if $N_j=0$, since it can otherwise be absorbed into $X_{j1}$.)

Using the stationary nature of the spacetime further, we will use the Sobolev spaces $H^s$ on $M^\circ=\R_{t_*}\times X$, with $X=(r_--2\delta,r_++2\delta)_r\times\Sphere^2$ considered as an open subset of $\R^3_y$; thus, the norm of $u\in H^s$ is defined by
\begin{equation}
\label{EqKdSSobolev}
  \|u\|_{H^s}^2 = \sum_{j+|\beta|\leq s} \|\pa_{t_*}^j \pa_y^\beta u\|_{L^2(M^\circ,|dg_0|)}^2
\end{equation}

Our central result in the form which is easiest to state is:

\begin{thmi}
\label{ThmIntroWave-simple}
  For $g(0,0)=g_0$ a Kerr-de Sitter metric with angular momentum $|a|\ll\bhm$ as above, and for $\alpha>0$ sufficiently small and $f\in\CIc(M^\circ)$ real-valued with sufficiently small $H^{\derivsUdU}$-norm, the wave equation $\Box_{g(u,du)} u=f+q(u,du)$ in $\Omega^\circ$, with vanishing Cauchy data, and with $q$ as above with $N_j\geq 1$ for all $j$, has a unique real smooth (in $\Omega^\circ$) global forward solution of the form $u=u_0+\tilde u$, $e^{\alpha t_*}\tilde u$ bounded, $u_0=c\chi$, $c\in\R$, where $\chi=\chi(t_*)$ is identically $1$ for large $t_*$. More precisely, for any fixed $s\in\R$, we will have $e^{\alpha t_*}\tilde u\in H^s$ if $f$ is sufficiently small (depending on $s$).
\end{thmi}

Further, the analogous conclusion holds for the Klein-Gordon operator $\Box-m^2$ with $m>0$ sufficiently small, without the presence of the $u_0$ term, i.e.\ for $\alpha>0,\ m>0$ sufficiently small, if $f\in\CIc(\Omega^\circ)$ has sufficiently small $H^{\derivsUdU}$-norm, $(\Box_{g(u,du)}-m^2)u=f+q(u,du)$ has a unique smooth global forward solution $u\in e^{-\alpha t_*}L^\infty(\Omega^\circ)$.  In fact, for Klein-Gordon equations one can also obtain a leading term, analogously to $u_0$, which now has the form $c e^{-i\sigma_1 t_*}\chi$, $\sigma_1$ the resonance of $\Box_{g(0)}-m^2$ with the largest imaginary part; thus $\Im\sigma_1<0$, so this is a decaying solution.

The {\em only} reason the assumption $|a|\ll\bhm$ is made is due to the possible presence (to the extent that we do not disprove it here) of resonances in $\Im\sigma\geq 0$, apart from the $0$-resonance with constants as the resonant state, for larger $a$. Below, in \S\ref{sec:general}, we give a general result in a form that makes it clear that this is the only remaining item to check --- indeed, this even holds in natural vector bundle settings.

We proceed to describe the natural geometric setup in which our analysis of quasilinear wave equations will take place. To begin with, one can conveniently encode the uniform structure of $(M^\circ,g_0)$ as $t_*\to\infty$ by adding an `ideal boundary' at infinity, that is, by introducing
\[
  x:=e^{-t_*}
\]
and adding $x=0$ to the spacetime, obtaining the 4-dimensional manifold
\[
  M = [0,\infty)_x \times (r_--2\delta,r_++2\delta)_r \times \Sphere^2
\]
with boundary $X=\{x=0\}$. The Lorentzian metric $g_0$ then has a specific structure at $\pa M$, i.e.\ `infinity', called a totally characteristic, or b-, structure, which we recall below. The closure of $\Omega^\circ$ in $M$ is the compact domain $\Omega=[0,1]_x\times[r_--\delta,r_++\delta]_r\times\Sphere^2$, which is a submanifold with corners with three boundary hypersurfaces which are the intersections of $H_1$, $H_2$ (defined above) and $\pa M$ (the boundary at future infinity) with $\Omega$; see Figure~\ref{FigKDSFull}. We are interested in solving forward problems for wave-like equations in $\Omega$ with vanishing Cauchy data at $H_1$ as above.

\begin{figure}[!ht]
  \centering 
  \includegraphics{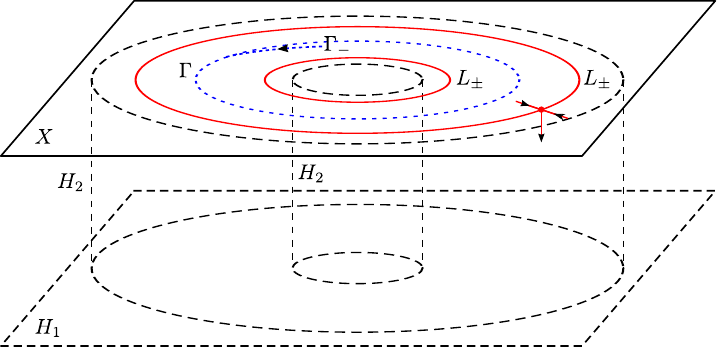}
  \caption{Setup for the discussion of the forward problem on Kerr-de Sitter space from a compactified perspective. Indicated are the ideal boundary $X$, the Cauchy hypersurface $H_1$ and the hypersurface $H_2$, which has two connected components which lie beyond the cosmological horizon and beyond the black hole event horizon, respectively. The domain $\Omega$ is bounded by $H_1$ in the past and by $X$ and $H_2$ in the future. The horizons at $X$ themselves are the projections to the base of the (generalized) radial sets $L_\pm$, discussed below, each of which has two components, corresponding to the two horizons. The projection to the base of the bicharacteristic flow is indicated near a point on $L_+$; near $L_-$, the directions of the flowlines are reversed. Lastly, $\Gamma$ is the trapped set, and the projection of a trapped trajectory approaching $\Gamma$ within $\Gamma_-=\Gamma_-^+\cup\Gamma_-^-$, discussed below, is indicated.}
\label{FigKDSFull}
\end{figure}

Now, recall that on any $n$-dimensional manifold with boundary $M$, the Lie algebra of smooth vector fields tangent to the boundary is denoted by $\Vb(M)$; in local coordinates $(x,y_1,\ldots,y_{n-1})$, with $x$ a boundary defining function, these are linear combinations of $x\pa_x$ and $\pa_{y_j}$ with $\CI(M)$ coefficients. Just as a dual metric is a linear combination of symmetric tensor products of coordinate vector fields, a dual metric in this totally characteristic setting, also called a dual b-metric, is a linear combination of
\[
  x\pa_x\otimes x\pa_x,\ \frac{1}{2}(x\pa_x\otimes \pa_{y_j}+\pa_{y_j}\otimes x\pa_x),\ \frac{1}{2}(\pa_{y_i}\otimes\pa_{y_j}+\pa_{y_j}\otimes\pa_{y_i}).
\]
One can think of this as a symmetric bilinear form; then a Lorentzian dual b-metric is a non-degenerate bilinear form of signature $(1,n-1)$. The actual metric is then a linear combination of
\begin{equation}
\label{EqBSym2}
  \frac{dx}{x}\otimes \frac{dx}{x},\ \frac{1}{2}\Big(\frac{dx}{x}\otimes dy_j+\,dy_j\otimes \frac{dx}{x}\Big),\ \frac{1}{2}(dy_i\otimes dy_j+dy_j\otimes dy_i).
\end{equation}
The corresponding wave operator is thus a totally characteristic, or b-, operator, $\Box\in\Diffb^2(M)$, i.e.\ is the sum of products of up to two factors of elements of $\Vb(M)$, with $\CI(M)$ coefficients.

Invariantly, $\Vb(M)$ is the space of all $\CI$ sections of a vector bundle, $\Tb M$, with local basis $x\pa_x$, $\pa_{y_j}$. The dual bundle of $\Tb M$ is denoted by $\Tb^*M$; it has a local basis of $\frac{dx}{x}$, $dy_j$. Thus, the tensors in \eqref{EqBSym2} span (over $\CI(M)$) the space of sections of the symmetric tensor power $S^2\Tb^*M$ in a local coordinate chart. A b-metric then is a non-degenerate smooth section of $S^2\Tb^*M$, i.e.\ a non-degenerate symmetric bilinear form on the fibers of $\Tb M$ smoothly depending on the base point; a Lorentzian b-metric is one of signature $(1,n-1)$.

In the concrete example of the compactification of Kerr-de Sitter space described above, we simply have $x\pa_x=-\pa_{t_*}$ and $\frac{dx}{x}=-dt_*$, hence $\Tb M$ and $\Tb^* M$ can be thought of as `uniform versions' of $TM^\circ$ and $T^*M^\circ$ up to the ideal boundary $x=0$. Further, a special class of elements of $\CI(M)$ is given by the $t_*$-independent smooth functions on $X$, and in general, elements of $\CI(M)$, by virtue of their smooth dependence on $x=e^{-t_*}$, differ from a $t_*$-independent function by a function bounded by $Ce^{-t_*}$. (Since the choice $x=e^{-t_*}$ of boundary defining function of future infinity is arbitrary --- any $x=e^{-\gamma t_*}$, $\gamma\in\CI(X)$, would work as well --- the space $\CI(M)$ is not quite natural, hence we will later amend it by adding to it the space $\Hb^\infty$ of conormal functions, see below, in order to obtain a space which \emph{is} independent of the choice of b-compactification.) Thus, in this example, smooth b-metrics are asymptotically stationary, i.e.\ differ from a stationary metric by a symmetric 2-tensor which is exponentially decaying (measuring its size using a stationary Riemannian metric, or in fact any smooth Riemannian b-metric) in $t_*$.

In order to compress notation for elements of $\Vb(M)$ applied to a function $u$, it is convenient to introduce the notation 
\[
  \bdiff u=(x\pa_x u)\frac{dx}{x}+\sum_j (\pa_{y_j} u) dy_j 
\]
in terms of local coordinates. This is simply the differential $du$ of $u$, rewritten in terms of the 1-forms $\frac{dx}{x}$ and $dy_j$ dual to the vector fields $x\pa_x$ and $\pa_{y_j}$, thus it is in fact invariantly defined. Note that when one writes e.g.\ $a(u,\bdiff u)$, one could instead, at least locally, write 
\[
  a(u,x\pa_x u,\pa_{y_1} u,\ldots,\pa_{y_{n-1}}u);
\]
the $\bdiff u$ notation is more concise and invariant. We note that $\bdiff$ preserves reality.

We now discuss Sobolev spaces from the point of view of b-analysis: We measure regularity with respect to $\Vb(M)$, and for non-negative integer $s$, one lets $\Hb^s(M)$ be the set of (complex-valued) $u\in L^2_{\bl}(M)$ such that $V_1\ldots V_j u\in L^2_{\bl}(M)$ for $j\leq s$ and $V_1,\ldots, V_j\in \Vb(M)$. Here, $L^2_{\bl}$ is the $L^2$-space with respect to any b-metric (such as the Kerr-de Sitter metric) which in local coordinates is thus given by a b-density, which is a positive smooth multiple of $x^{-1}\,|dx\,dy_1\ldots dy_{n-1}|$. Further, one introduces the weighted Sobolev spaces $\Hb^{s,\alpha}(M)=x^{\alpha}\Hb^s(M)$; we will often also work with the spaces $\Hb^{s,\alpha}(M;\RR)$ of real-valued elements without making this explicit in the notation. $\Hb^{s,\alpha}(M)$-sections of vector bundles are defined by local trivializations. In our Kerr-de Sitter setup, the Sobolev spaces $\Hb^{s,\alpha}$ on $\Omega$ are defined by restriction, with the norm of $u\in\Hb^{s,\alpha}$ given by
\[
  \|u\|_{\Hb^{s,\alpha}}^2 = \sum_{j\leq s, 1\leq i_1\leq\ldots\leq i_j\leq N} \|V_{i_1}\ldots V_{i_j}u\|_{L^2_\bl(\Omega)}^2,
\]
where $\{V_1,\ldots,V_N\}$ spans $\Vb(M)$ over $\CI(M)$. We point out that the b-differential $\bdiff$, defined locally by
$$
\bdiff u=(x\pa_x u)\frac{dx}{x}+\sum_j (\pa_{y_j} u) dy_j
$$
maps $\Hb^{s,\alpha}(M)$ to $\Hb^{s-1,\alpha}(M;\Tb^*M)$.

By the compactness of $\Omega$, norms on $\Hb^{s,\alpha}$ corresponding to different choices of spanning set and b-density are equivalent. We stress that this statement is a purely topological one, i.e.\ does \emph{not} rely on any metric; only the choice of compactification $M$ (and thus $\Omega$) was motivated by the stationary nature of the Kerr-de Sitter metric, but even then, different choices of boundary defining function $e^{-\gamma t_*}$, $\gamma\in\CI(X)$, as above, while changing the smooth structure of the compactification, yield the \emph{same} b-Sobolev spaces. Likewise, norms of $\Hb^{s,\alpha}$-sections of a vector bundle $E$ over $\Omega$ can be defined using any positive definite metric on $E$ --- again, by compactness of $\Omega$, all such choices give equivalent norms.

From the stationary point of view, the Sobolev space $\Hb^{s,\alpha}$ on $\Omega$ is the weighted Sobolev space $e^{-\alpha t_*}H^s$, with $H^s$ defined in \eqref{EqKdSSobolev}; for sections of vector bundles $E$ over $\Omega$, one first identifies $E$ with the pullback of $E|_X$ ($X$ a spatial slice as before, intersected with $\Omega$), equipping $E|_X$ with any positive definite fiber metric, by the map $\R\times X\ni(t_*,y)\mapsto y$, in order to obtain a $t_*$-independent bundle and metric on it, and one can then define norms of $E$-valued $\Hb^{s,\alpha}$-sections using local trivializations of $E|_X$. In summary, adopting the `b' point of view in the present context is a natural and geometric replacement for the use of the stationary nature of the spacetime, and a very useful one for the analysis, as we will see in the sequel.

Rephrasing and generalizing the setup of Theorem~\ref{ThmIntroWave-simple}, the wave equations we study include those of the form
\[
  \Box_{g(u,\bdiff u)}u=f+q(u,\bdiff u),
\]
where $g(0,0)=g_0$, and for each $p\in M$, $g_p(v_0,v):\RR\oplus\Tb^*_pM\to\Sym^2\Tb_p^* M$, depending smoothly on $p$, and
\[
  q(u,\bdiff u)=\sum_{j=1}^{N'} a_j u^{e_j}\prod_{k=1}^{N_j} X_{jk} u,\quad e_j,N_j\in \N_0,\ N_j+e_j\geq 2,
\]
with
\begin{equation}
\label{eq:init-reg}
  a_j\in\CI(M),\ X_{jk}\in\Vb(M)
\end{equation}
real; in fact, since $\CI(M)$ is not a natural space as discussed above, we relax \eqref{eq:init-reg} to
\begin{equation}
\label{eq:prec-reg}
  a_j\in\CI(M)+\Hb^\infty(M),\ X_{jk}\in(\CI+\Hb^\infty)\Vb(M);
\end{equation}
we note that the space $\CI(M)+\Hb^\infty(M)$ is independent of the choice of b-compact\-ification. Generalizing the forcing as well, the more natural version of Theorem~\ref{ThmIntroWave-simple} is, with further generalization given in Theorems~\ref{ThmIntroGeneralKG} and \ref{ThmIntroGeneralWave}:

\begin{thmi}
\label{ThmIntroWave}
  Let $s\in\R$. For $g(0,0)=g_0$ a Kerr-de Sitter metric with angular momentum $|a|\ll\bhm$ as above, and for $\alpha>0$ sufficiently small and $f\in\Hb^{\infty,\alpha}$ real-valued with sufficiently small $\Hb^{\derivsUdU,\alpha}$-norm (depending on $s$), the wave equation $\Box_{g(u,\bdiff u)} u=f+q(u,\bdiff u)$ in $\Omega^\circ$, with vanishing Cauchy data, and with $q$ as above with $N_j\geq 1$ for all $j$, has a unique real smooth (in $\Omega^\circ$) global forward solution of the form $u=u_0+\tilde u$, $\tilde u\in\Hb^{s,\alpha}$, $u_0=c\chi$, $c\in\R$, with $\chi\in\CI(M)$ identically $1$ near $\pa M$.

  Further, the analogous conclusion, in fact for $s=\infty$, holds for the Klein-Gordon equation $\Box-m^2$ with $m>0$ sufficiently small, without the presence of the $u_0$ term, i.e.\ for $\alpha>0,\ m>0$ sufficiently small, if $f\in\Hb^{\infty,\alpha}(\Omega)$ has sufficiently small $\Hb^{\derivsUdU,\alpha}$-norm, $(\Box_{g(u,\bdiff u)}-m^2)u=f+q(u,\bdiff u)$ in $\Omega^\circ$ has a unique, smooth in $\Omega^\circ$, global forward solution $u\in\Hb^{\infty,\alpha}$.
\end{thmi}

One can work over the complex numbers as well, the only change being that one needs to assume $g_p(v_0,v):\C\oplus\CxTb^*_pM\to\Sym^2\Tb_p^* M$ now, i.e.\ taking complex-valued arguments, but still producing a real symmetric 2-tensor. We will in fact consider the complex-valued setting in the rest of the paper, since the real-valued case follows as a special case.

The natural notion of \emph{asymptotically Kerr-de Sitter spaces} in the context of the present paper is suggested by this result, namely, one would reasonably thus call $(M,g(u,\bdiff u))$ for a function $u$ as in the above Theorem; hence, for us, an asymptotically Kerr-de Sitter space is a metric perturbation of a Kerr-de Sitter space within the class of sections of $S^2\Tb^*M$ with regularity $\CI+\Hb^{s,\alpha}$, $\alpha>0$, for $s$ large or $s=\infty$.

For the proof of Theorem~\ref{ThmIntroWave}, we refer to Corollaries~\ref{CorQuasilinearKdS} and \ref{CorQuasilinearKGKdS}, which are special cases of Theorems~\ref{ThmQuasilinearWave} and \ref{ThmQuasilinearKG}. For any fixed amount of regularity of the solution, our arguments only require a finite number of derivatives: Indeed, for sufficiently large $s_0,C\in\R$ and for $s\geq s_0$, it is sufficient to assume $f\in\Hb^{Cs,\alpha}$, with small $\Hb^{\derivsUdU,\alpha}$-norm, to ensure the existence of a unique global forward solution $u$ with $\Hb^{s,\alpha}$-regularity, i.e.\ with $\tilde u\in\Hb^{s,\alpha}$ in the case of wave equations, $u\in\Hb^{s,\alpha}$ in the case of Klein-Gordon equations; see Remark~\ref{RmkFiniteRegularity} for details.

We now discuss previous results on Kerr-de Sitter space and its perturbations. There seems
to be little work on non-linear equations in Kerr-de Sitter type settings; indeed the only
paper the authors are aware of is the earlier paper \cite{HintzVasySemilinear} of the authors in which
the semilinear Klein-Gordon equation was studied (with small data
well-posedness shown) with non-linearity
depending on $u$ only, so that the losses due to the trapping could
still be handled by a contraction mapping argument. In addition, the
same paper also analyzed non-linearities depending on $\bdiff u$
provided these had a special structure at the trapped set. There
is more work on the linear equation on perturbations of de Sitter-Schwarzschild and
Kerr-de Sitter spaces: a rather complete analysis of the asymptotic
behavior of solutions of the linear wave equation was given in
\cite{VasyMicroKerrdS}, upon which the linear analysis of the present
paper is ultimately based. Previously in exact Kerr-de Sitter space and for small angular momentum,
Dyatlov \cite{DyatlovQNM, Dyatlov:Exponential}
has shown exponential decay to constants, even across the event
horizon; see also the more recent work of Dyatlov \cite{Dyatlov:Asymptotics}.
Further, in de Sitter-Schwarzschild space
(non-rotating black holes) Bachelot \cite{Bachelot:Scattering} set up
the functional analytic scattering theory in the early 1990s, while later
S\'a Barreto and Zworski
\cite{Sa-Barreto-Zworski:Distribution} and Bony and H\"afner
\cite{Bony-Haefner:Decay} studied resonances and decay away from the
event horizon, Dafermos and
Rodnianski in \cite{Dafermos-Rodnianski:Sch-dS} showed polynomial
decay to constants in this setting, and
Melrose, S\'a Barreto and Vasy
\cite{Melrose-SaBarreto-Vasy:Asymptotics} improved this result to
exponential decay
to constants. There is also physics literature on the subject,
starting with Carter's discovery of this space-time
\cite{Carter:Hamilton, Carter:Global}, either
using explicit solutions in special cases, or numerical calculations,
see in particular \cite{Yoshida-Uchikata-Futamase:KerrdS}, and
references therein. We also refer to the paper of Dyatlov and Zworski
\cite{DyatlovZworskiTrapping} connecting recent mathematical advances
with the physics literature.

While it received more attention,
the linear, and thus the non-linear, equation on Kerr space (which has vanishing
cosmological constant) does not fit directly
into our setting; see the introduction of \cite{VasyMicroKerrdS} for an
explanation and for further references and
\cite{Dafermos-Rodnianski:Lectures} for more background and additional
references. Some of the key works in this area include
the polynomial decay on Kerr space which
was shown recently by Tataru and Tohaneanu
\cite{Tataru-Tohaneanu:Local, Tataru:Local}
and Dafermos, Rodnianski and Shlapentokh-Rothman \cite{Dafermos-Rodnianski:Black,
Dafermos-Rodnianski:Axi, Dafermos-Rodnianski-Shlapentokh:Decay}, after pioneering work of Kay and Wald
in \cite{Kay-Wald:Linear} and \cite{Wald:Stability} in the
Schwarzschild setting. Andersson and Blue
\cite{Andersson-Blue:Uniform} proved a decay result for
the Maxwell system on slowly rotating Kerr spaces; see also the
earlier work of Bachelot \cite{Bachelot:Gravitational} in the
Schwarzschild setting.
The crucial normal hyperbolicity of the trapping, corresponding to
null-geodesics that do not escape through the event horizons, in Kerr space was realized
and proved by Wunsch and Zworski \cite{WunschZworskiNormHypResolvent};
later Dyatlov extended and refined the result
\cite{DyatlovResonanceProjectors, DyatlovNormally}. Note that a
stronger version of normal hyperbolicity, called $r$-normal hyperbolicity and proved for Kerr in \cite{WunschZworskiNormHypResolvent} as well, is
a notion that is stable under perturbations.

On the non-linear side, Luk \cite{LukKerrNonlinear} established global existence for forward problems for semilinear wave equations on Kerr space under a null condition, and Dafermos, Holzegel and Rodnianski \cite{Dafermos-Holzegel-Rodnianski:Scattering} constructed \emph{backward} solutions for Einstein's equations on Kerr space. (There was also recent work by Marzuola, Metcalfe, Tataru and
Tohaneanu \cite{Marzuola-Metcalfe-Tataru-Tohaneanu:Strichartz} and Tohaneanu \cite{TohaneanuKerrStrichartz}
on Strichartz estimates, which are applied to the study of semilinear wave equations with power non-linearities, and by Donninger, Schlag and Soffer \cite{Donninger-Schlag-Soffer:Price} on $L^\infty$
estimates on Schwarzschild black holes,
following $L^\infty$ estimates of Dafermos and Rodnianski
\cite{Dafermos-Rodnianski:Red-shift, Dafermos-Rodnianski:Price},
of Blue and Soffer \cite{Blue-Soffer:Phase}
on non-rotating charged black holes giving
$L^6$ estimates, and of Finster, Kamran, Smoller and
Yau \cite{Finster-Kamran-Smoller-Yau:Decay,
Finster-Kamran-Smoller-Yau:Linear} on Dirac waves on Kerr.)

In the next section, \S\ref {sec:general}, we explain the ingredients of the proof of Theorem~\ref{ThmIntroWave}, and we also state natural generalizations. At the end of that section we provide a detailed roadmap through this paper.

The authors are very grateful to Semyon Dyatlov for providing a  
 preliminary version of his manuscript \cite{DyatlovNormally} and for
discussions about it, as well as for pointing out the reference
\cite{Klainerman:Global}. They are also very grateful to Maciej
Zworski and three anonymous referees for comments that greatly improved the exposition. They are also thankful to Gunther Uhlmann,
Richard Melrose and Rafe Mazzeo for comments and their interest in this project.

\section{Overview of the proof and the more general results}
\label{sec:general}

Having stated the result, we now explain {\em why} it holds. It is best to begin with the underlying linear equation; after all, the non-linearity is `just' a rather serious perturbation! In general, the analysis of b-differential operators (locally finite sums of finite products of elements of $\Vb(M)$), such as $\Box_g\in\Diffb^2(M)$, has two ingredients, corresponding to the two orders, smoothness and decay, of the Sobolev spaces:
\begin{enumerate}
\item b-regularity analysis. This provides the framework for understanding PDE at high b-frequencies, which in non-degenerate situations involves the b-principal symbol and perhaps a subprincipal term. This is sufficient in order to control solutions $u$ in $\Hb^{s,r}$ in terms of their $\Hb^{s',r}$-norm, $s'<s$, i.e.\ in terms of the norm of $u$ in a space with lower regularity, but no additional decay. Since for the inclusion $\Hb^{s,r}\to\Hb^{s',r'}$ to be compact one needs both $s>s'$ and $r>r'$, this does not control the problem modulo relatively compact errors.
\item Normal operator analysis. This provides a framework for understanding
the decay properties of solutions of the PDE. The normal operator is
obtained by freezing coefficients of the differential operator $L$ at $\pa M$ to obtain a
dilation-invariant b-operator $N(L)$. One then Mellin
transforms the normal operator in the normal variable to obtain a
family of operators $\hat L(\sigma)$, depending on the Mellin-dual
variable $\sigma$. The b-regularity analysis, in non-degenerate situations, gives control of this
family $\hat L(\sigma)$ in a Fredholm sense, uniformly as
$|\sigma|\to\infty$ with $\Im\sigma$ bounded. However, in any such
strip, $\hat L(\sigma)^{-1}$ will still typically have finitely many
poles $\sigma_j$; these poles, called {\em resonances}, dictate the
asymptotic behavior of solutions of the PDE.
\end{enumerate}
In order to have a
Fredholm operator $L$, one needs to work in spaces such as
$\Hb^{s,r}$, where $r$ is such that there are no resonances $\sigma_j$
with $\Im\sigma_j=-r$. One can also work in slightly more general
spaces, such as $\cX^{s,r}:=\Cx\oplus\Hb^{s,r}$, $r>0$, identified with a space of
distributions via $u=u_0+\tilde u$, $\tilde u\in\Hb^{\infty,\alpha}$,
$u_0=c\chi$, corresponding to $(c,\tilde u)\in \Cx\oplus\Hb^{s,r}$.

The global analysis of linear (and semilinear) wave-type equations on asymptotically Kerr-de Sitter spacetimes in this framework, which we will recall momentarily, was carried out in \cite{VasyMicroKerrdS,HintzVasySemilinear}, and the analysis in the present paper follows these earlier works closely.

Now, the b-regularity analysis for our non-elliptic equation involves
the (null)-bicharacteristic flow. In view of the version of H\"ormander's
theorem on propagation of singularities in this setting, and in view
of the a priori control on Cauchy data at $H_1$, what one would like
is that all bicharacteristics tend to $T^*_{H_1}M$ in one
direction. Moreover, for the purposes of the adjoint problem, which
effectively imposes Cauchy data at $H_2$, one
would like that the bicharacteristics tend to $\Tb^*_{H_2}M$ in the
other direction.

Unfortunately, bicharacteristics within $\Tb^*_XM$ can never leave this space, and thus will not tend to $T^*_{H_1}M$. This is mostly resolved, however, by the conormal bundle of the {\em horizons} at $X$, which give rise to a bundle of saddle points for the bicharacteristic flow. Since the flow is homogeneous, it is convenient to consider it in $\Sb^*M=(\Tb^*M\setminus o)/\RR^+$. The characteristic set $\Sigma\subset\Sb^*M$ has two components $\Sigma_\pm$, with $\Sigma_-$ forward-oriented (i.e.\ future oriented time functions increase along null-bicharacteristics in $\Sigma_-$), $\Sigma_+$ backward oriented. Then the images of the conormal bundles of the horizons in the cosphere bundle are submanifolds $L_\pm\subset\Sigma_\pm$ of $\Sb^*_X M$, with one-dimensional stable ($-$)/unstable ($+$) manifold $\cL_\pm$ transversal to $\Sb^*_XM$. (The flow within $L_\pm$ need not be trivial; if it is, one has {\em radial points}, as in the $a=0$ de Sitter-Schwarzschild space. However, for simplicity we refer to the $L_\pm$ estimates as {\em radial point estimates} in general.) The realistic ideal situation, called a {\em non-trapping} one, then is if all (null-)bicharacteristics in $\Sb^*_\Omega M\cap(\Sigma_+\setminus L_+)$ tend to $\Sb^*_{H_2}M\cup L_+$ in the backward direction, and $\Sb^*_{H_1}M\cup L_+$ in the forward direction, with a similar statement for $\Sigma_-$, with backward and forward interchanged. Notice that due to the assumption on the one-dimensional stable/unstable manifold being transversal to $\Sb^*_X M$, there cannot be non-trivial bicharacteristics in $\Sb^* M$ tending to $L_+$ in both the forward and the backward direction, since a bicharacteristic is either completely in $\Sb^*_X M$, or completely outside it. In this non-trapping setting the only subtlety is that the propagation estimates through $L_\pm$ require that the differentiability order $s$ and the decay order $r$ be related by $s>\frac{1}{2}+\beta r$ for a suitable $\beta>0$ (dictated by the Hamilton dynamics at $L_\pm$), i.e.\ the more decay one wants, the higher the regularity needs to be.

This is still not the case in Kerr-de Sitter space, though it is true
for neighborhoods of the static patch in de Sitter space, and its
perturbations. The additional ingredient for Kerr-de Sitter space is
normally hyperbolic trapping, introduced in this context by Wunsch and
Zworski \cite{WunschZworskiNormHypResolvent}, given by smooth submanifolds
$\Gamma^\pm\subset\Sigma_\pm$. Here $\Gamma^\pm$ are invariant
submanifolds for the Hamilton flow, given by the transversal
intersection of locally defined smooth, Hamilton flow invariant,
$\Gamma^\pm=\Gamma^\pm_+\cap\Gamma^\pm_-$, with
$\Gamma^\pm_-\subset\Sigma$ transversal to $\Sb^*_XM\cap\Sigma$, and $\Gamma^\pm_+\subset\Sb^*_XM\cap\Sigma$. Combining results of
\cite{DyatlovResonanceProjectors,DyatlovNormally} (which would work directly in a
dilation invariant setting) and \cite{HintzVasyNormHyp} we show
that for $r>0$ sufficiently small, one can propagate $\Hb^{s,r}$ estimates through
$\Gamma^\pm$. This suffices to complete the b-regularity setup if the
non-trapping requirement is replaced by: All (null-)bicharacteristics
in $\Sb^*_\Omega M\cap(\Sigma_+\setminus (L_+\cup \Gamma^+))$ tend to
either $\Sb^*_{H_2}M\cup L_+\cup\Gamma^+$ in the backward direction, and
$\Sb^*_{H_1}M\cup L_+\cup\Gamma^+$ in the forward direction, with the
tending to $\Gamma^+$ allowed in only {\em one} of the forward and
backward directions, with a similar
statement for $\Sigma_-$, with backward and forward
interchanged. Finally, this {\em is} satisfied in Kerr-de Sitter
space, and also in its b-perturbations (the whole setup is perturbation stable).

Next, one needs to know about the resonances of the operator. For the
wave operator, the only resonance with non-negative imaginary part is
$0$, with the kernel of $\hat L(\sigma)$ one dimensional,
consisting of constants. Since strips can only have finitely many
resonances, there is $r>0$ such that in $\Im\sigma\geq -r$ the only
resonance is $0$; then $\Hb^{s,r}\oplus\Cx$ works for our Fredholm
setup. For the Klein-Gordon equation with $m>0$ small, the $m=0$
resonance at $0$ moves to $\sigma_1=\sigma_1(m)$ inside $\Im\sigma<0$, see \cite{DyatlovQNM,HintzVasySemilinear}. Thus, one can either work with
$\Hb^{s,r'}$ where $r'$ is sufficiently small (depending on $m$), or
with $\Hb^{s,r}\oplus\Cx$, though with $\Cx$ now identified with
$cx^{i\sigma_1}\chi$.

We now discuss the non-linear terms. Here the basic point is that
$\Hb^{s,0}$ is an algebra if $s>n/2$, and thus for such $s$, products of elements of
$\Hb^{s,r}$ possess even more decay if $r>0$, but they become more
growing if $r<0$. Thus, one is forced to work with $r\geq 0$.

First, with the simplest
semilinear equation, with no derivatives in the non-linearity $q$ (so
$N_j\geq 2$ is replaced by $N_j=0$), the regularity losses due to the normally hyperbolic trapping are in principle sufficiently small
to allow for a contraction mapping principle (Picard iteration) based
argument. However, for the actual wave equation on Kerr-de Sitter space, the $0$-resonance
prohibits this, as the iteration maps outside the space $\Hb^{s,r}\oplus\C$. Thus, it is
the semilinear Klein-Gordon equation that is well-behaved from this
perspective, and this was solved by the authors in
\cite{HintzVasySemilinear}. On the other hand, if derivatives are
allowed, with an at least quadratic behavior in $\bdiff u$, then the
non-linearity annihilates the $0$-resonance. Unfortunately, since the
normally hyperbolic estimate loses $1+\ep$ derivatives, as opposed to
the usual real principal type/radial point loss of one derivative, the
solution operator for $\Box_g$ will not map $q(u,\bdiff u)$ back into the
desired Sobolev space, preventing a non-linear analysis based on the contraction mapping principle.

Fortunately, the Nash-Moser iteration is designed to deal with just
such a situation. In this paper we adapt the iteration to our
requirements, and in particular show that semilinear equations of the
kind just described are in fact solvable. In particular, we prove that
all the estimates used in the linear problems are {\em tame}. Here we
remark that Klainerman's early work on global solvability involved the
Nash-Moser scheme \cite{Klainerman:Global}, though this was later
removed by Klainerman and Ponce \cite{Klainerman-Ponce:Global}. In the
present situation the loss of derivatives seems much more serious,
however, due to the trapping, so it seems unlikely that the solution scheme can be made
more `classical'.

However, we are also interested in quasilinear equations. Quasilinear
versions of the above non-trapping scenario were studied by the first author \cite{HintzQuasilinearDS}, who showed the solvability of quasilinear
wave equations on perturbations of de Sitter space. The key ingredient
in dealing with quasilinear equations is to allow operators with
coefficients with regularity the same kind as what one is proving for
the solutions, in this case $\Hb^{s,r}$-regularity; the technical tools for dealing with such operators were developed in \cite{HintzQuasilinearDS}, following Beals and Reed \cite{BealsReedMicroNonsmooth}. All of the smooth linear
ingredients (microlocal elliptic regularity, propagation of
singularities, radial points) have their analogue for $\Hb^{s,r}$
coefficients if $s$ is sufficiently large. Thus, in
\cite{HintzQuasilinearDS} a Picard-type iteration,
$u_{k+1}=\Box_{g(u_k)}^{-1}(f+q(u_k,\bdiff u_k))$ was used to solve
the quasilinear wave equations on de Sitter space. Notice that
$\Box_{g(u_k)}$ has non-smooth coefficients; indeed, these lie in a weighted
b-Sobolev space.

In our Kerr-de Sitter situation there is normally hyperbolic
trapping. However, notice that as we work in decaying Sobolev spaces
modulo constants, $\Box_{g(u)}$ differs from a Kerr-de Sitter operator
with smooth coefficients, $\Box_{g(c)}$, by one with {\em exponentially (in $t_*$) decaying coefficients}. This means that one can combine the smooth coefficient normally hyperbolic theory, as in the work of Dyatlov
\cite{DyatlovResonanceProjectors}, with a tame estimate in $\Hb^{s,r}$
with $r<0$; the sign of $r$ here is a crucial gain since for $r<0$ the
propagation estimates through normally hyperbolic trapped sets behave in exactly the same way
as real principal type estimates. In combination this provides the
required tame estimates for Kerr-de Sitter wave equations: Using the notation of Theorem~\ref{ThmIntroWave} and its setup (thus, allowing $g$ to depend on derivatives of $u$ as well), and with $u=c\chi+\wt u$, $c\in\R$, $\wt u\in\Hb^{s+4,\alpha}$, the tame estimate takes the form
\begin{equation}
\label{EqIntroTameEst}
  \|w\|_{\cX^{s,\alpha}} \leq C(s,\|u\|_{\cX^{s_0,\alpha}})\bigl(\|f\|_{\Hb^{s+3,\alpha}} + \|f\|_{\Hb^{s_0,\alpha}}\|u\|_{\cX^{s+4,\alpha}}\bigr)
\end{equation}
for $w=b\chi+\wt w$, $b\in\R,\wt w\in\Hb^{s,\alpha}$, solving $\Box_{g(u,\bdiff u)}w=f$, where we use the norm $\|w\|_{\cX^{s,\alpha}}:=|b|+\|\wt w\|_{\Hb^{s,\alpha}}$, likewise for $u=a\chi+\wt u$, and $s_0$ is a suitable fixed number, while $s$ can be chosen arbitrarily large. The point of a tame estimate like \eqref{EqIntroTameEst} for $w$ is that one loses a \emph{fixed number of derivatives}, independent of $s$, relative to the forcing $f$, and moreover the estimate is \emph{linear in the high regularity norms} of the forcing $f$ and the coefficients, coming from $u$, of the operator $\Box_{g(u,\bdiff u)}$. The idea of the proof of the estimate \eqref{EqIntroTameEst} is to rewrite $\Box_{g(u,\bdiff u)}w=f$ as $\Box_{g(c,0)}w=f-(\Box_{g(u,\bdiff u)}-\Box_{g(c,0)})w$ and treating the second term on the right as an error term in the sense that it has better a priori decay than $w$, even though it has less regularity; then the normal operator analysis allows one to deduce further partial asymptotics for $w$ from the rewritten equation. Tame regularity estimates for the equation $\Box_{g(u,\bdiff u)}w=f$ itself, using the dynamical structure of Kerr-de Sitter space as described above, then allow one to regain the loss of regularity to a large extent. The tame-ness comes about by a careful analysis of these microlocal estimates, the fundamental ingredient being a simple Moser-type estimate for products of Sobolev functions, which on $\R^n$ is of the form $\|uv\|_{H^s}\lesssim\|u\|_{H^s}\|v\|_{H^{s_0}}+\|u\|_{H^{s_0}}\|v\|_{H^s}$ for $s\geq s_0$, $s_0$ sufficiently large but fixed. --- Given \eqref{EqIntroTameEst}, one can then use a simple version of the Nash-Moser iteration scheme given by Saint-Raymond \cite{SaintRaymondNashMoser} to complete the proof of the main theorem.

We emphasize that our treatment of quasilinear wave-type equations is
systematic and general; indeed, our discussion of quasilinear waves in \S\ref{SecQuasilinear} takes place in a general geometric setting (but we restrict to scalar equations for brevity): To wit, quasilinear equations which at $X=\pa M$
are modelled on a finite dimensional family $L=L(v_0)$, $v_0\in\Cx^d$
small corresponding to the zero resonances (thus the family is
$0$-dimensional without $0$-resonances!), of smooth b-differential operators on a vector
bundle with scalar
principal symbol which has
the bicharacteristic dynamics described above (radial sets, normally
hyperbolic trapping, etc.) fit into it, {\em provided two conditions hold
for the normal operator} (i.e.\ the dilation invariant model associated
to $L$ at $\pa M$):\footnote{The operator needs to be a second order differential operator, with principal symbol a Lorentzian dual metric, only if Cauchy hypersurfaces are used, as we do in the applications presented in this paper; otherwise, e.g.\ in situations where one can instead use complex absorption, the order of the operator is irrelevant.}
\begin{enumerate}
\item First, {\em the resonances for the model $L(v_0)$} have negative imaginary part, or if they have $0$ imaginary part, the non-linearity annihilates them.
\item Second, the {\em normally hyperbolic trapping estimates} of Dyatlov \cite{DyatlovResonanceProjectors} hold for $\hat L(\sigma)$ (as $|\Re\sigma|\to\infty$) in $\Im\sigma>-r_0$ for some $r_0>0$. In the semiclassical rescaling, with $\sigma=h^{-1}z$, $h=|\sigma|^{-1}$, this is a statement about $\hat L_{h,z}=h^m\hat L(h^{-1}z)$, $\Im z>-r_0 h$. By Dyatlov's recent result\footnote{This could presumably also be seen from the work of Nonnenmacher and Zworski \cite{NonnenmacherZworskiDecay} by checking that this extension goes through without significant changes in the proof.} \cite{DyatlovNormally} this indeed is the case if $\hat L_{h,z}$ satisfies that {\em at $\Gamma$} its skew-adjoint part, $\frac{1}{2i}(\hat L_{h,z}-\hat L_{h,z}^*)\in h\Diff_{\hbar}^1(X)$, for $z\in\RR$ has semiclassical principal symbol bounded above by $h(\numin-\ep)/2$ for some $\ep>0$, where $\numin$ is the minimal expansion rate in the normal directions at $\Gamma$; see \cite[Theorem~1]{DyatlovNormally} and the remark below it (which allows the non-trivial skew-adjoint part, denoted by $Q$ there, microlocally at $\Gamma$).
\end{enumerate}

It is important to point out that in view of the 
decay of the solutions either to $0$ if there are no real resonances, or to the space of resonant states corresponding to real resonances, the conditions must be checked for at most a finite dimensional family of
elements of the `smooth' algebra $\Psib(M)$, and moreover there is no need to
prove tame estimates, deal with rough coefficients, etc., for this
point, and one is in a dilation invariant setting, i.e.\ can simply
Mellin transform the problem. Hence, in
principle, solving wave-type equations on more complicated bundles is
reduced to analyzing these two aspects of the
associated linear model operator at infinity.

Concretely, we have the following two theorems on Kerr-de Sitter spaces:

\begin{thmi}
\label{ThmIntroGeneralKG}
Let $M$ be a Kerr-de Sitter space with angular momentum
$|a|<\frac{\sqrt{3}}{2}\bhm$ that satisfies \cite[(6.13)]{VasyMicroKerrdS} (this condition on $\Lambda,\bhm$ and $a$ ensures non-trapping \emph{classical} dynamics for the null-bicharacteristic flow for the normal operator family). Let $E$ a vector bundle over $M$ with a positive definite
metric $k$ on $E$, and let $L_{g(u,\bdiff u)}\in\Diffb^2(M;E)$ have principal
symbol $G=g^{-1}(u,\bdiff u)$ (times the identity), and suppose that $L_0=L_{g(0,0)}$
satisfies that
\begin{enumerate}
\item
the large parameter principal symbol of
$\frac{1}{2i|\sigma|}(L_0-L_0^*)$, with the adjoint taken relative to
$k\,|dg|$, at the trapped set $\Gamma$ is $<\numin/2$
as an endomorphism of $E$,
\item
$\hat L_0(\sigma)$ has no resonances in
$\Im\sigma\geq 0$.
\end{enumerate}
Then for $\alpha>0$ sufficiently small, there exists\footnote{See the
  proof of this theorem in \S\ref{SubsecGeneralProofs}, in
  particular \eqref{EqGeneralDCond}, for the value of $d$.} $d>0$ such
that the following holds: If $f\in\Hb^{\infty,\alpha}(\Omega)$ has a
sufficiently small $\Hb^{2d}$-norm, then the equation $L_{g(u,\bdiff
  u)}u=f+q(u,\bdiff u)$ has a unique, smooth in $M^\circ$, global forward solution $u\in\Hb^{\infty,\alpha}(\Omega)$.
\end{thmi}

In particular, the conditions at $\Gamma$ for the theorem hold if $|a|\ll\bhm$, $E$ a tensor bundle over $M$ of finite rank (e.g.\ the bundle $\Lambdab^*M$ of differential forms or $S^k\Tb^*M$ of symmetric $k$-tensors) , $L_{g(u,\bdiff u)}=\Box_{g(u,\bdiff u)}$ the tensor d'Alembertian, or indeed if $L_{g(u,\bdiff u)}-\Box_{g(u,\bdiff u)}$ is a $0$th order operator, since hyperbolicity is shown in \cite{VasyMicroKerrdS} in the full stated range of $a$, while for $a=0$, one can explicitly compute $\frac{1}{2i}(L_0-L_0^*)$ at $\Gamma$ and verify the stated bound, where one takes adjoints with respect to a suitable Riemannian fiber metric $k$; in fact, to obtain this bound for all permitted spacetime parameters $\bhm$ and $\Lambda$, one needs to take $k$ to be a pseudodifferential Riemannian inner product, which one can indeed allow as well in the statement of the Theorem. See \cite{HintzPsdoInner} for definitions and proofs.

\begin{thmi}
\label{ThmIntroGeneralWave}
Let $(M,g_0)$ be a Kerr-de Sitter space with angular momentum
$|a|<\frac{\sqrt{3}}{2}\bhm$ that satisfies \cite[(6.13)]{VasyMicroKerrdS}. Let $E$ be a vector bundle over $M$ with a positive definite
metric $k$ on $E$, and let $L_{g(u,\bdiff u)}\in\Diffb^2(M;E)$ have principal
symbol $G=g^{-1}(u,\bdiff u)$ (times the identity), where $g(0,0)=g_0$. Suppose further that $L_0=L_{g(0,0)}$
is such that $\hat L_0(\sigma)$ has a simple resonance at $0$, with resonant states spanned by $u_{0,1},\ldots,u_{0,d}$, and no
other resonances in $\Im\sigma\geq 0$.  Consider the family $\hat L_{g(u_0,\bdiff u_0)}(\sigma)$,
$u_0\in\Span\{u_{0,1},\ldots,u_{0,d}\}$ with small enough
norm. Suppose that
\begin{enumerate}
\item
this family only has resonances at $0$ in
$\Im\sigma\geq 0$, and these are given by $\Span\{u_{0,1},\ldots,u_{0,d}\}$,
\item
the large parameter principal symbol of
$\frac{1}{2i|\sigma|}(L_0-L_0^*)$, with the adjoint taken relative to
$k\,|dg|$, at the trapped set $\Gamma$ is $<\numin/2$,
\item
$q(u_0,\bdiff u_0)=0$ for $u_0\in\Span\{u_{0,1},\ldots,u_{0,d}\}$.
\end{enumerate}
Let $s\in\R$. Then for $\alpha>0$ sufficiently small (depending on $s$), there exists\footnote{The
  value of $d$ is given in \eqref{EqGeneralDCond} in the course of the
  proof of this theorem in \S\ref{SubsecGeneralProofs}.} $d>0$
such that the following holds: If $f\in\Hb^{\infty,\alpha}$ has a
sufficiently small $\Hb^{2d,\alpha}$-norm, then the equation
$L_{g(u,\bdiff u)} u=f+q(u,\bdiff u)$ has a unique, smooth in $M^\circ$, global forward solution of the form $u=u_0+\tilde u$, $\tilde u\in\Hb^{s,\alpha}$, $u_0=\chi\sum_{j=1}^d c_j u_{0,j}$, $\chi\in\CI(M)$ identically $1$ near $\pa M$.

If moreover $\Gamma$ is uniformly normally hyperbolic for $\hat L_{g(u_0,\bdiff u_0)}(\sigma)$, one can take $s=\infty$, i.e.\ obtain a solution $u$ as above, with $\tilde u\in\Hb^{\infty,\alpha}$ now.
\end{thmi}

Here, `uniformly normally hyperbolic' means that one has a continuous family 
$\Gamma=\Gamma_{u_0}$ of smooth trapped sets, with a continuous family of 
stable/unstable manifolds, with uniform bounds (within this family) on 
the normal expansion rates for the flow, which ensures that the 
normally hyperbolic estimates are uniform within the family (for small $u_0$); see the discussion around \eqref{EqNuMinDef} for details.

We stress again that Theorems~\ref{ThmIntroGeneralKG} and \ref{ThmIntroGeneralWave} hold under the general assumptions described above; we restrict to quasilinear equations on Kerr-de Sitter spaces merely for concreteness. See also Remark~\ref{RmkGeneralTheoremsMfds}.

Again, the conditions at $\Gamma$ for the
theorem hold if
$|a|\ll\bhm$, $E$ a tensor bundle, if $L_{g(u,\bdiff u)}-\Box_{g(u,\bdiff u)}$ is a $0$th order operator,
$\Box_{g(u,\bdiff u)}$  the tensor d'Alembertian,
since the structurally stable $r$-normally hyperbolic
statement is shown in \cite{VasyMicroKerrdS} (which implies the
uniform normal hyperbolicity required in the theorem), while for $a=0$,
$\frac{1}{2i}(L_0-L_0^*)$ can be computed explicitly at $\Gamma$, as
mentioned above, and upper bounds on this are stable under perturbations.

The uniform normal hyperbolicity condition at $\Gamma$ holds if $|a|<\frac{\sqrt{3}}{2}\bhm$, $E$ a tensor bundle,
$L_{g(u,\bdiff u)}=\Box_{g(u,\bdiff u)}$ the tensor d'Alembertian, with $g(u_0,\bdiff u_0)$ being a Kerr-de Sitter metric for
$u_0\in\Span\{u_{0,1},\ldots,u_{0,d}\}$ with small norm, since the
hyperbolicity of $\Gamma$ was shown in this generality in
\cite{VasyMicroKerrdS}. (However, the computation of $\frac{1}{2i}(L_0-L_0^*)$ for this range of $a$ is more involved and will not be pursued here.) While the assumption on $g(u_0,\bdiff u_0)$ here is somewhat unnatural from the point of view of applications, a closely related assumption is expected to hold in the context of Einstein's field equations; this is subject of ongoing research.

The plan of the rest of this paper is the following. In \S\ref{SecNonsmoothCalc} we show that the non-smooth pseudodifferential operators of \cite{HintzQuasilinearDS} facilitate tame estimates (operator bounds, composition, etc.), with \S\ref{SubsecTameMicrolocal} establishing tame elliptic estimates in \S\ref{SubsecTameElliptic}, tame real principal type and radial point estimates in \S\ref{SubsecTamePropagation} and tame estimates at normally hyperbolic trapping in \S\ref{SubsecNontrapping} for $r<0$. In \S\ref{SubsecTrapping}, we adapt Dyatlov's analysis at normally hyperbolic trapping given in \cite{DyatlovNormally} to our needs. Finally, in \S\ref{SecQuasilinear} we solve our quasilinear equations by first showing that the microlocal results of \S\ref{SubsecTameMicrolocal} combine with the high energy estimates for the relevant normal operators following from the discussion of \S\ref{SubsecTrapping} to give tame estimates for the forward propagator in \S\ref{SubsecForward}, and then showing in \S\ref{SubsecSolve} that the Nash-Moser iteration indeed allows for solving our wave equations. \S\ref{SubsecSolveKG} then explains the changes required for quasilinear Klein-Gordon equations. Finally, in \S\ref{SubsecGeneralProofs} we show how our methods apply in the general settings of Theorems~\ref{ThmIntroGeneralKG} and \ref{ThmIntroGeneralWave}.

\section{Tame estimates in the non-smooth operator calculus}
\label{SecNonsmoothCalc}

In this section we prove the basic tame estimates for the
$\Hb$-coefficient, or simply non-smooth, b-pseudodifferential
operators defined in \cite{HintzQuasilinearDS}.

We work on the half space $\Rnhalf$ with coordinates $z=(x,y)\in[0,\infty)\times\R^{n-1}$; the coordinates in the fiber of the b-cotangent bundle are denoted $\zeta=(\lambda,\eta)$, i.e.\ we write b-covectors as $\lambda\,\frac{dx}{x}+\eta\,dy$. We point out that under the change of coordinates $(0,\infty)_x\times\R^{n-1}_y\mapsto\R^n$, $(x,y)\mapsto(-\log x,y)$, the non-smooth b-ps.d.o.s considered here become quantizations of symbols with standard (on $\R^n$) Sobolev regularity in the coefficients. However, once non-smooth and smooth b-ps.d.o.s interact (as they will in the study of microlocal regularity results at the boundary $x=0$), the b-picture, locally working on $\Rnhalf$, will be conceptually much cleaner than the Euclidean picture, locally working on $\R^n$ with a certain uniformity in one direction; hence we emphasize the `b' point of view from the beginning.

\subsection{Mapping properties}
\label{SubsecMapping}

We start with the tame mapping estimate, Proposition~\ref{PropMapping}, which essentially states that for non-smooth pseudodifferential operators $A$, a high regularity norm of $Au$ can be estimated by a high regularity norm of $A$ times a low regularity norm of $u$, plus a low regularity norm of $A$ times a high regularity norm of $u$. This is stronger than the a priori continuity estimate one gets from the bilinear map $(A,u)\mapsto Au$, which would require a product of high norms of both. In case $A$ is a multiplication operator, this is essentially a b-version of a (weak) Moser estimate, see Corollary~\ref{CorHbModule}.

Recall from \cite{HintzQuasilinearDS} the symbol class
\[
  S^{m;0}\Hb^s:=\{a(z,\zeta)\colon \|\la\zeta\ra^{-m}a(z,\zeta)\|_{\Hb^s}\in L^\infty_\zeta\}
\]
with the norm
\[
  \|a\|_{S^{m;0}\Hb^s}=\left\|\frac{\la\xi\ra^s \hat a(\xi,\zeta)}{\la\zeta\ra^m}\right\|_{L^\infty_\zeta L^2_\xi},
\]
where $\hat a$ denotes the Mellin transform in $x$ (which is the standard Fourier transform in $-\log x$) and Fourier transform in $y$ of $a$. Left quantizations of such symbols, denoted $\Op(a)\in\Psi^{m;0}\Hb^s$, act on $u\in\CIdotc(\Rnhalf)$ by
\[
  \Op(a)u(x,y)=(2\pi)^{-n}\int x^{i\lambda}e^{iy\eta}a(x,y,\lambda,\eta)\hat u(\lambda,\eta)\,d\lambda\,d\eta.
\]
Also recall
\[
  S^{m;k}\Hb^s=\{a\in S^{m;0}\Hb^s\colon \pa_\zeta^\alpha a\in S^{m-|\alpha|;0}\Hb^s, |\alpha|\leq k\}.
\]
and $\Psi^{m;k}\Hb^s=\Op S^{m;k}\Hb^s$. For brevity, we will use the following notation for Sobolev, symbol class and operator class norms, with the distinction between symbolic and b-Sobolev norms being clear from the context:
\begin{gather*}
  \|u\|_s:=\|u\|_{\Hb^s},\qquad \|u\|_{s,r}:=\|u\|_{\Hb^{s,r}}, \\
  \|a\|_{m,s}:=\|a\|_{S^{m;0}\Hb^s},\qquad \|a\|_{(m;k),s}:=\|a\|_{S^{m;k}\Hb^s}, \\
  \|\Op(a)\|_{m,s}:=\|\Op(a)\|_{\Psi^{m;0}\Hb^s}\equiv\|a\|_{S^{m;0}\Hb^s}, \\
  \|\Op(a)\|_{(m;k),s}:=\|\Op(a)\|_{\Psi^{m;k}\Hb^s}\equiv\|a\|_{S^{m;k}\Hb^s}.
\end{gather*}
If $A$ is a b-operator acting on an element of a weighted b-Sobolev space with weight $r$ (which will be apparent from the context), then $\|A\|_{m,s}$ is to be understood as $\|x^{-r}Ax^r\|_{m,s}$, similarly for $\|A\|_{(m;k),s}$. Lastly, for $A\in\Hb^s\Psib^m$, i.e.\ $A$ is a finite sum of products $u_jA_j$, where $u_j\in\Hb^s$, $A_j\in\Psib^m$, we write $\|A\|_{\Hb^s\Psib^m}$, by an abuse of notation, for an unspecified $\Hb^s\Psib^m$-seminorm of $A$. Note that $\Hb^s\Psib^m\subset\Psi^{m;\infty}\Hb^s$ (with strict inclusion).

Recall the notation $x_+=\max(x,0)$ for $x\in\R$.
\begin{prop}
\label{PropMapping}
  (Extension of \cite[Proposition~3.9]{HintzQuasilinearDS}.) Let $s\in\R$, $A=\Op(a)\in\Psi^{m;0}\Hb^s$, and suppose $s'\in\R$ is such that $s\geq s'-m$, $s>n/2+(m-s')_+$. Then $A$ defines a bounded map $\Hb^{s'}\to\Hb^{s'-m}$, and for all $\mu,\nu$ with
  \[
    \mu>n/2+(m-s')_+,\quad \nu>n/2+(m-s')_+ +s'-s,
  \]
  there is a constant $C=C(s,s',m,\mu,\nu)>0$ such that
  \begin{equation}
  \label{EqMappingTame}
    \|Au\|_{s'-m}\leq C(\|A\|_{m,\mu}\|u\|_{s'} + \|A\|_{m,s}\|u\|_\nu).
  \end{equation}
\end{prop}

By the assumptions on $s$ and $s'$, one may always choose $\mu=s$ and $\nu=s'$ here, which recovers the non-tame estimate given in \cite{HintzQuasilinearDS}. Estimates of the form \eqref{EqMappingTame}, called `tame estimates' e.g.\ in \cite{HamiltonNashMoser,SaintRaymondNashMoser}, are crucial for applications in a Nash-Moser iteration scheme, as alluded to in \S\ref{sec:general}.

\begin{proof}[Proof of Proposition~\ref{PropMapping}]
  We compute
  \begin{align*}
    \|Au\|_{s'-m}^2&=\int\la\zeta\ra^{2(s'-m)}|\widehat{Au}(\zeta)|^2\,d\zeta \\
	  &\leq\int\la\zeta\ra^{2(s'-m)}\left(\int|\hat a(\zeta-\xi,\xi)\hat u(\xi)|\,d\xi\right)^2\,d\zeta.
  \end{align*}
  We split the inner integral into two pieces, corresponding to the domains of integration $|\zeta-\xi|\leq|\xi|$ and $|\xi|\leq|\zeta-\xi|$, which can be thought of as splitting up the action of $A$ on $u$ into a low-high and a high-low frequency interaction. We estimate
  \begin{align}
	\label{EqLowHigh}
	\begin{split}
    \int&\la\zeta\ra^{2(s'-m)}\left(\int_{|\zeta-\xi|\leq|\xi|}|\hat a(\zeta-\xi,\xi)\hat u(\xi)|\,d\xi\right)^2\,d\zeta \\
	  &\leq\int\left(\int_{|\zeta-\xi|\leq|\xi|}\frac{\la\zeta\ra^{2(s'-m)}\la\xi\ra^{2m}}{\la\zeta-\xi\ra^{2\mu}\la\xi\ra^{2s'}}\,d\xi\right) \\
	  &\hspace{3cm}\times\left(\int\frac{\la\zeta-\xi\ra^{2\mu}|\hat a(\zeta-\xi,\xi)|^2}{\la\xi\ra^{2m}} \la\xi\ra^{2s'}|\hat u(\xi)|^2\,d\xi\right)\,d\zeta,
	\end{split}
  \end{align}
  and we claim that the integral which is the first factor on the right hand side is uniformly bounded in $\zeta$: Indeed, if $s'-m\geq 0$, then we use $|\zeta|\leq 2|\xi|$ on the domain of integration, thus
  \[
    \int_{|\zeta-\xi|\leq|\xi|}\frac{\la\zeta\ra^{2(s'-m)}}{\la\zeta-\xi\ra^{2\mu}\la\xi\ra^{2(s'-m)}}\,d\xi\lesssim\int\frac{1}{\la\zeta-\xi\ra^{2\mu}}\,d\xi \in L^\infty_\zeta,
  \]
  since $\mu>n/2$; if, on the other hand, $s'-m\leq 0$, then $|\xi|\leq|\zeta-\xi|+|\zeta|$ gives
  \[
    \int_{|\zeta-\xi|\leq|\xi|}\frac{\la\xi\ra^{2(m-s')}}{\la\zeta-\xi\ra^{2\mu}\la\zeta\ra^{2(m-s')}}\,d\xi\lesssim\int\frac{1}{\la\zeta-\xi\ra^{2(\mu-(m-s'))}}+\frac{1}{\la\zeta-\xi\ra^{2\mu}}\,d\xi \in L^\infty_\zeta,
  \]
  since $\mu>n/2+(m-s')$; hence, from \eqref{EqLowHigh}, the $\Hb^{s'-m}$ norm of the low-high frequency interaction in $Au$ is bounded by $C_\mu\|a\|_{m,\mu}\|u\|_{s'}$.

  We estimate the norm of high-low interaction in a similar way: We have
  \begin{align}
	\label{EqHighLow}
	\begin{split}
    \int&\la\zeta\ra^{2(s'-m)}\left(\int_{|\xi|\leq|\zeta-\xi|}|\hat a(\zeta-\xi,\xi)\hat u(\xi)|\,d\xi\right)^2\,d\zeta \\
	  &\leq\int\left(\int_{|\xi|\leq|\zeta-\xi|}\frac{\la\zeta\ra^{2(s'-m)}\la\xi\ra^{2m}}{\la\zeta-\xi\ra^{2s}\la\xi\ra^{2\nu}}\,d\xi\right) \\
	  &\hspace{3cm}\times\left(\int\frac{\la\zeta-\xi\ra^{2s}|\hat a(\zeta-\xi,\xi)|^2}{\la\xi\ra^{2m}} \la\xi\ra^{2\nu}|\hat u(\xi)|^2\,d\xi\right)\,d\zeta.
	\end{split}
  \end{align}
  If $s'-m\geq 0$, the first inner integral on the right hand side is bounded by
  \[
    \int_{|\xi|\leq|\zeta-\xi|}\frac{1}{\la\zeta-\xi\ra^{2(s-s'+m)}\la\xi\ra^{2(\nu-m)}}\,d\xi\leq\int\frac{1}{\la\xi\ra^{2(s-s'+\nu)}}\,d\xi,
  \]
  where we use $s\geq s'-m$, and this integral is finite in view of $\nu>n/2+s'-s$; if $s'-m\leq 0$, then
  \[
    \int_{|\xi|\leq|\zeta-\xi|}\frac{1}{\la\zeta\ra^{2(m-s')}\la\zeta-\xi\ra^{2s}\la\xi\ra^{2(\nu-m)}}\,d\xi\leq\int\frac{1}{\la\xi\ra^{2(\nu-m+s)}}\,d\xi,
  \]
  which is finite in view of $\nu>n/2+m-s$. In summary, we need $\nu>n/2+\max(m,s')-s=n/2+(m-s')_++s'-s$ and can then bound the $\Hb^{s'-m}$ norm of the high-low interaction by $C_\nu\|a\|_{m,s}\|u\|_\nu$. The proof is complete.
\end{proof}

Using $\Hb^s\subset S^{0;0}\Hb^s$, we obtain the following weak version of the Moser estimate for the product of two b-Sobolev functions:
\begin{cor}
\label{CorHbModule}
  Let $s>n/2$, $|s'|\leq s$. If $u\in\Hb^s,v\in\Hb^{s'}$, then $uv\in\Hb^{s'}$, and one has an estimate
  \[
    \|uv\|_{s'}\leq C(\|u\|_\mu\|v\|_{s'}+\|u\|_s\|v\|_\nu)
  \]
  for all $\mu>n/2+(-s')_+,\nu>n/2+s'_+-s$, where $C=C(s,s',\mu,\nu)$. In particular, for $u,v\in\Hb^s$,
  \[
    \|uv\|_s\leq C(\|u\|_\mu\|v\|_s+\|u\|_s\|v\|_\mu)
  \]
  for $\mu>n/2$, with $C=C(s,\mu)$.
\end{cor}

\subsection{Operator compositions}
\label{SubsecComp}

We give a tame estimate for the norms of expansion and remainder terms arising in the composition of two non-smooth operators:
\begin{prop}
\label{PropNonsmoothComp}
  Suppose $s,m,m'\in\R$, $k,k'\in\N_0$ are such that $s>n/2$, $s\leq s'-k$ and $k\geq m+k'$. Suppose $P=p(z,\Db)\in\Psi^{m;k}\Hb^s$, $Q=q(z,\Db)\in\Psi^{m';0}\Hb^{s'}$. Put
  \begin{gather*}
    E_j:=\sum_{|\beta|=j}\frac{1}{\beta!}(\pa_\zeta^\beta p \Db_z^\beta q)(z,\Db), \\
	R:=P\circ Q-\sum_{0\leq j<k} E_j.
  \end{gather*}
  Then $E_j\in\Psi^{m+m'-j;0}\Hb^s$ and $R\in\Psi^{m'-k';0}\Hb^s$, and for $\mu>n/2$ fixed,
  \begin{gather*}
    \|E_j\|_{m+m'-j,s}\leq C(\|P\|_{(m;j),\mu}\|Q\|_{m,s+j}+\|P\|_{(m;j),s}\|Q\|_{m',\mu+j}), \\
	\|R\|_{m'-k',s}\leq C(\|P\|_{(m;k),\mu}\|Q\|_{m',s+k}+\|P\|_{(m;k),s}\|Q\|_{m',\mu+k}).
  \end{gather*}
\end{prop}
\begin{proof}
  The statements about the $E_j$ follow from Corollary~\ref{CorHbModule}. For the purpose of proving the estimate for $R$, we define
  \[
    p_0=\pa_\zeta^k p\in S^{m-k;0}\Hb^s,\quad \Db_z^k q\in S^{m';0}\Hb^{s'-k},
  \]
  where we write $\pa_\zeta^k=(\pa_\zeta^\beta)_{|\beta|=k}$, similarly for $\Db_z^k$. Notice that in particular $p_0\in S^{0;0}\Hb^s$. Then $R=r(z,\Db)$ with
  \[
    |\hat r(\eta;\zeta)|\lesssim \int\left(\int_0^1 p_0(\eta-\xi;\zeta+t\xi)\,dt\right) q_0(\xi;\zeta)\,d\xi
  \]
  by Taylor's formula, hence
  \begin{align*}
    \int&\frac{\la\eta\ra^{2s}|\hat r(\eta;\zeta)|^2}{\la\zeta\ra^{2m'}}\,d\eta \\
	  &\lesssim \int\left(\int_{|\eta-\xi|\leq|\xi|}\frac{\la\eta\ra^{2s}}{\la\eta-\xi\ra^{2\mu}\la\xi\ra^{2s}}\,d\xi\right) \\
	  &\hspace{2cm}\times \left(\int\biggl(\int_0^1\la\eta-\xi\ra^{2\mu}|p_0(\eta-\xi,\zeta+t\xi)|^2\,dt\biggr)\frac{\la\xi\ra^{2s}|q_0(\xi;\zeta)|^2}{\la\zeta\ra^{2m'}}\,d\xi\right)\,d\eta \\
	  &\hspace{0.2cm}+ \int\left(\int_{|\xi|\leq |\eta-\xi|}\frac{\la\eta\ra^{2s}}{\la\eta-\xi\ra^{2s}\la\xi\ra^{2\mu}}\,d\xi\right) \\
	  &\hspace{2cm}\times \left(\int\biggl(\int_0^1\la\eta-\xi\ra^{2s}|p_0(\eta-\xi,\zeta+t\xi)|^2\,dt\biggr)\frac{\la\xi\ra^{2\mu}|q_0(\xi;\zeta)|^2}{\la\zeta\ra^{2m'}}\,d\xi\right)\,d\eta,
  \end{align*}
  which implies the claimed estimate for $k'=0$. For $k'>0$, we use a
  trick of Beals and Reed \cite{BealsReedMicroNonsmooth} as in the
  proof of Theorem~3.12 in \cite{HintzQuasilinearDS} to reduce the
  statement to the case $k'=0$: Recall that the idea is to split up $q(z,\zeta)$ into a `trivial' part $q_0$ with compact support in $\zeta$ and $n$ parts $q_i$, where $q_i$ has support in $|\zeta_i|\geq 1$, and then writing
  \[
    P\circ Q_i = \sum_{j=0}^{k'} c_{jk'} P\Db_{z_i}^{k'-j}\circ (\Db_{z_i}^j q_i)(z,\Db)\Db_{z_i}^{-k'}
  \]
  for some constants $c_{jk'}\in\R$ using the Leibniz rule; then what we have proved above for $k'=0$ can be applied to the $j$-th summand on the right hand side, which we expand to order $k-j$, giving the result.
\end{proof}

\subsection{Reciprocals of and compositions with $\Hb^s$ functions}
\label{SubsecRecComp}

We also need sharp\-er bounds for reciprocals and compositions of b-Sobolev functions on a compact $n$-dimensional manifold with boundary. Localizing using a partition of unity, we can simply work on $\Rnhalf$.

\begin{prop}
\label{PropHbRec}
  (Extension of \cite[Lemma~4.1]{HintzQuasilinearDS}.) Let $s>n/2+1$, $u,w\in\Hb^s$, $a\in\CI$, and suppose that $|a+u|\geq c_0$ near $\supp w$. Then $w/(a+u)\in\Hb^s$, and one has an estimate
  \begin{equation}
  \label{EqHbRecEst}
    \left\|\frac{w}{a+u}\right\|_s\leq C(\|u\|_\mu,\|a\|_{C^N})c_0^{-1}\max(c_0^{-\lceil s\rceil},1)\bigl(\|w\|_s+\|w\|_\mu(1+\|u\|_s)\bigr).
  \end{equation}
  for any fixed $\mu>n/2+1$ and some $s$-dependent $N\in\N$.
\end{prop}
\begin{proof}
  Choose $\psi_0,\psi\in\CI$ such that $\psi_0\equiv 1$ on $\supp w$, $\psi\equiv 1$ on $\supp\psi_0$, and such that moreover $|a+u|\geq c_0>0$ on $\supp\psi$. Then we have $\|w/(a+u)\|_0\leq c_0^{-1}\|w\|_0$. We now iteratively prove higher regularity of $w/(a+u)$ and an accompanying `tame' estimate: Let us assume $w/(a+u)\in\Hb^{s'-1}$ for some $1\leq s'\leq s$. Let $\Lambda_{s'}=\lambda_{s'}(\Db)\in\Psib^{s'}$ be an operator with principal symbol $\la\zeta\ra^{s'}$. Then
  \begin{equation}
    \begin{split}
	\label{EqHsRecEllEstimate}
    \Bigl\|\Lambda_{s'}&\frac{w}{a+u}\Bigr\|_0\leq\Bigl\|(1-\psi)\Lambda_{s'}\frac{\psi_0 w}{a+u}\Bigr\|_0+\Bigl\|\psi \Lambda_{s'}\frac{\psi_0 w}{a+u}\Bigr\|_0 \\
	    &\lesssim \Bigl\|\frac{w}{a+u}\Bigr\|_0+c_0^{-1}\Bigl\|\psi(a+u)\Lambda_{s'}\frac{w}{a+u}\Bigr\|_0 \\
	    &\leq c_0^{-1}\|w\|_0+c_0^{-1}\left(\|\psi\Lambda_{s'}w\|_0+\Bigl\|\psi[\Lambda_{s'},a+u]\frac{w}{a+u}\Bigr\|_0\right) \\
	    &\lesssim c_0^{-1}\biggl(\|w\|_{s'}+\Bigl\|\frac{w}{a+u}\Bigr\|_{s'-1}+\Bigl\|\psi[\Lambda_{s'},u]\frac{w}{a+u}\Bigr\|_0\biggr),
	\end{split}
  \end{equation}
  where we used that the support assumptions on $\psi_0$ and $\psi$ imply $(1-\psi)\Lambda_{s'}\psi_0\in\Psib^{-\infty}$, and $\psi[\Lambda_{s'},a]\in\Psib^{s'-1}$. Hence, in order to prove that $w/(a+u)\in\Hb^{s'}$, it suffices to show that $[\Lambda_{s'},u]\colon \Hb^{s'-1}\to\Hb^0$. Let $v\in \Hb^{s'-1}$. Since
  \begin{align*}
    (\Lambda_{s'}uv)\ftrans(\zeta)&=\int \lambda_{s'}(\zeta)\hat u(\zeta-\xi)\hat v(\xi)\,d\xi \\
	(u\Lambda_{s'}v)\ftrans(\zeta)&=\int\hat u(\zeta-\xi)\lambda_{s'}(\xi)\hat v(\xi)\,d\xi,
  \end{align*}
  we have, by taking a first order Taylor expansion of $\lambda_{s'}(\zeta)=\lambda_{s'}(\xi+(\zeta-\xi))$ around $\zeta=\xi$,
  \[
    ([\Lambda_{s'},u]v)\ftrans(\zeta)=\sum_{|\beta|=1}\int\left(\int_0^1\pa_\zeta^\beta \lambda_{s'}(\xi+t(\zeta-\xi))\,dt\right)(\Db_z^\beta u)\ftrans(\zeta-\xi)\hat v(\xi)\,d\xi,
  \]
  thus, writing $u'=\Db_z u\in\Hb^{s-1}$,
  \[
    |([\Lambda_{s'},u]v)\ftrans(\zeta)|\lesssim\int\left(\int_0^1\la\xi+t(\zeta-\xi)\ra^{s'-1}\,dt\right)\big|\widehat{u'}(\zeta-\xi)\big||\hat v(\xi)|\,d\xi.
  \]
  To obtain a tame estimate for the $L^2_\zeta$ norm of this expression, we again use the method of decomposing the integral into low-high and high-low components: The low-high component is bounded by
  \begin{align*}
    &\int\left(\int_{|\zeta-\xi|\leq|\xi|}\frac{\sup_{0\leq t\leq 1}\la\xi+t(\zeta-\xi)\ra^{2(s'-1)}}{\la\zeta-\xi\ra^{2(\mu-1)}\la\xi\ra^{2(s'-1)}}\,d\xi\right) \\
	  &\hspace{3cm}\times\left(\int\la\zeta-\xi\ra^{2(\mu-1)}\big|\widehat{u'}(\zeta-\xi)\big|^2\la\xi\ra^{2(s'-1)}|\hat v(\xi)|^2\,d\xi\right)\,d\zeta;
  \end{align*}
  the first inner integral, in view of $s'\geq 1$, so the $\sup$ is bounded by $\la\xi\ra^{2(s'-1)}$, which cancels the corresponding term in the denominator, is finite for $\mu>n/2+1$. For the high-low component, we likewise estimate
  \begin{align*}
    &\int\left(\int_{|\xi|\leq|\zeta-\xi|}\frac{\sup_{0\leq t\leq 1}\la\xi+t(\zeta-\xi)\ra^{2(s'-1)}}{\la\zeta-\xi\ra^{2(s-1)}\la\xi\ra^{2\nu}}\,d\xi\right) \\
	  &\hspace{3cm}\times\left(\int\la\zeta-\xi\ra^{2s}\big|\widehat{u'}(\zeta-\xi)\big|^2\la\xi\ra^{2\nu}|\hat v(\xi)|^2\,d\xi\right)\,d\zeta,
  \end{align*}
  and the first inner integral on the right hand side is bounded by
  \[
    \int_{|\xi|\leq|\zeta-\xi|}\frac{1}{\la\zeta-\xi\ra^{2(s-s')}\la\xi\ra^{2\nu}}\,d\xi\leq\int\frac{1}{\la\xi\ra^{2(s-s'+\nu)}}\,d\xi
  \]
  because of $s\geq s'$, which is finite for $\nu>n/2+s'-s$. We conclude that
  \[
    \|[\Lambda_{s'},u]v\|_0\leq C_{\mu\nu}(\|u\|_\mu\|v\|_{s'-1}+\|u\|_{s'}\|v\|_\nu),
  \]
  for $\mu>n/2+1,\nu>n/2+s'-s$. Plugging this into \eqref{EqHsRecEllEstimate} yields
  \[
    \left\|\frac{w}{a+u}\right\|_{s'}\lesssim c_0^{-1}\left(\|w\|_{s'}+(1+\|u\|_\mu)\left\|\frac{w}{a+u}\right\|_{s'-1}+\|u\|_{s'}\left\|\frac{w}{a+u}\right\|_\nu\right),
  \]
  where the implicit constant in the inequality is independent of $c_0,w$ and $u$. Using the abbreviations $q_\sigma:=\|w/(a+u)\|_\sigma$, $u_\sigma=\|u\|_\sigma$, $w_\sigma=\|w\|_\sigma$ and fixing $\mu>n/2+1$, this means
  \[
    q_{s'}\lesssim c_0^{-1}(w_{s'}+(1+u_\mu)q_{s'-1}+u_{s'}q_\nu), \quad\nu>n/2+s'-s,
  \]
  with the implicit constant being independent of $c_0,w,a,u,\mu$. We will use this for $s'\leq\gamma:=\lfloor n/2\rfloor+1$ with $\nu=s'-1$, and for $s'>\gamma$, we will take $\nu=\gamma$, thus obtaining a tame estimate for $q_s$. In more detail, for $1\leq s'\leq\gamma$, we have
  \[
    q_{s'}\lesssim c_0^{-1}(w_{s'}+(1+u_{s'})q_{s'-1}),
  \]
  which gives, with $C_0=\max(1,c_0^{-1})$,
  \[
    q_\gamma\lesssim c_0^{-1}w_\gamma\sum_{j=0}^{\gamma-1}(c_0^{-1}(1+u_\gamma))^j+(c_0^{-1}(1+u_\gamma))^\gamma q_0\lesssim c_0^{-1}C_0^\gamma w_\gamma(1+u_\gamma)^\gamma
  \]
  using the bound $q_0\leq c_0^{-1}w_0\leq c_0^{-1}w_\gamma$. For $\gamma<s'\leq s$, we have
  \[
    q_{s'}\lesssim c_0^{-1}(w_s+u_s q_\gamma+(1+u_\mu)q_{s'-1}),
  \]
  thus for integer $k\geq 1$ with $\gamma+k\leq s$,
  \begin{align*}
    q_{\gamma+k}&\leq c_0^{-1}(w_s+u_s q_\gamma)\sum_{j=0}^{k-1}(c_0^{-1}(1+u_\mu))^j + (c_0^{-1}(1+u_\mu))^k q_\gamma \\
	   &\lesssim c_0^{-1}C_0^{k-1}(1+u_\mu)^k(w_s+(1+u_s)q_\gamma) \\
	   &\lesssim c_0^{-1}C_0^{\gamma+k}(1+u_\mu)^{\gamma+k}(w_s+(1+u_s)w_\gamma),
  \end{align*}
  where we used $\mu>\gamma$ in the last inequality, thus proving the estimate \eqref{EqHbRecEst} in case $s$ is an integer; in the general case, we just use $q_{\gamma'}\leq q_\gamma$ for $\gamma'<\gamma$, in particular for $\gamma'=s-\lceil s-\gamma\rceil$, and use the above with $q_{\gamma+k}$ replaced by $q_{\gamma'+k}$.
\end{proof}

As in \cite{HintzQuasilinearDS}, one thus obtains regularity results for compositions, but now with sharper estimates. To illustrate how to obtain these, let us prove an extension of \cite[Proposition~4.5]{HintzQuasilinearDS}. Let $M$ be a compact $n$-dimensional manifold with boundary, $s>n/2+1$, $\alpha\geq 0$.

\begin{prop}
\label{PropCompWithAnalytic}
  Let $u\in\Hb^{s,\alpha}(M)$. If $F\colon\Omega\to\C$, $F(0)=0$, is holomorphic in a simply connected neighborhood $\Omega$ of the range of $u$, then $F(u)\in\Hb^{s,\alpha}(M)$, and
  \begin{equation}
  \label{EqCompWithAnalyticEst}
    \|F(u)\|_{s,\alpha}\leq C(\|u\|_{\mu,\alpha})(1+\|u\|_{s,\alpha})
  \end{equation}
  for fixed $\mu>n/2+1$. Moreover, there exists $\eps>0$ such that $F(v)\in\Hb^{s,\alpha}(M)$ depends continuously on $v\in\Hb^{s,\alpha}(M)$, $\|u-v\|_{s,\alpha}<\eps$.
\end{prop}
\begin{proof}
  Observe that $u(M)$ is compact. Let $\gamma\subset\C$ denote a smooth contour which is disjoint from $u(M)$, has winding number $1$ around every point in $u(M)$, and lies within the region of holomorphicity of $F$. Then, writing $F(z)=zF_1(z)$ with $F_1$ holomorphic in $\Omega$, we have
  \[
    F(u)=\frac{1}{2\pi i}\oint_\gamma F_1(\zeta)\frac{u}{\zeta-u}\,d\zeta,
  \]
  Since $\gamma\ni\zeta\mapsto u/(\zeta-u)\in\Hb^{s,\alpha}(M)$ is continuous by Proposition~\ref{PropHbRec}, we obtain, using the estimate \eqref{EqHbRecEst},
  \[
    \|F(u)\|_{s,\alpha} \leq C(\|u\|_\mu)\bigl(\|u\|_{s,\alpha}+\|u\|_{\mu,\alpha}(1+\|u\|_s)\bigr),
  \]
  which implies \eqref{EqCompWithAnalyticEst} in view of $\alpha\geq 0$. The continuous (in fact, Lipschitz) dependence of $F(v)$ on $v$ is a consequence of Proposition~\ref{PropHbRec} and Corollary~\ref{CorHbModule}.
\end{proof}

We also study compositions $F(u)$ for $F\in\CI(\R;\C)$ and real-valued $u$.

\begin{prop}
\label{PropCompWithSmooth}
  (Extension of \cite[Proposition~4.7]{HintzQuasilinearDS}.) Let $F\in\CI(\R;\C)$, $F(0)=0$. Then for $u\in\Hb^{s,\alpha}(M;\R)$, we have $F(u)\in\Hb^{s,\alpha}(M)$, and one has an estimate
  \begin{equation}
  \label{EqCompWithSmoothEst}
    \|F(u)\|_{s,\alpha}\leq C(\|u\|_{\mu,\alpha})(1+\|u\|_{s,\alpha})
  \end{equation}
  for fixed $\mu>n/2+1$. In fact, $F(u)$ depends continuously on $u$ in the same sense as in Proposition~\ref{PropCompWithAnalytic}.
\end{prop}
\begin{proof}
  The proof is the same as in \cite{HintzQuasilinearDS}, using almost analytic extensions, only we now use the sharper estimate \eqref{EqHbRecEst} to obtain \eqref{EqCompWithSmoothEst}.
\end{proof}

\begin{prop}
\label{PropCompWithSmooth2}
  (Extension of \cite[Proposition~4.8]{HintzQuasilinearDS}.) Let $F\in\CI(\R;\C)$, and $u'\in\CI(M;\R),u''\in\Hb^{s,\alpha}(M;\R)$; put $u=u'+u''$. Then $F(u)\in\CI(M)+\Hb^{s,\alpha}(M)$, and one has an estimate
  \[
    \|F(u)-F(u')\|_{s,\alpha}\leq C(\|u'\|_{C^N},\|u''\|_{\mu,\alpha})(1+\|u''\|_{s,\alpha})
  \]
  for fixed $\mu>n/2+1$ and some $N\in\N$. In fact, $F(u)$ depends
  continuously on $u$.
\end{prop}
\begin{proof}
  The proof is the same as in \cite{HintzQuasilinearDS}, but now uses the sharper estimate \eqref{EqHbRecEst}.
\end{proof}

\section{Microlocal regularity: tame estimates}
\label{SubsecTameMicrolocal}

When stating microlocal regularity estimates (like elliptic regularity, real principal type propagation, etc.) for operators with coefficients in $\Hb^s(\Rnhalf)$, we will give two quantitative statements, one for `low' regularities $\sigma\lessapprox n/2$, in which we will not make use of any tame estimates established earlier, and one for `high' regularities $n/2\lessapprox\sigma\lessapprox s$, in which the tame estimates will be used.

To concisely write down tame estimates, we use the following notation: The right hand side of a tame estimate will be a real-valued function, denoted by $L$, of the form
\begin{equation}
\label{EqDefTame}
\begin{split}
  L(p^\ell_1,\ldots,p^\ell_a&; p^h_1,\ldots,p^h_b; u^\ell_1,\ldots,u^\ell_c; u^h_1,\ldots,u^h_d) \\
    &=\sum_{j=1}^d c_j(p^\ell_1,\ldots,p^\ell_a)u^h_j + \sum_{j=1}^b\sum_{k=1}^c c_{jk}(p^\ell_1,\ldots,p^\ell_a)p^h_j u^\ell_k
\end{split}
\end{equation}
here, the $c_j$ and $c_{jk}$ are continuous functions. In applications, $p^{\ell/h}_j$ will be a low/high regularity norm of the coefficients of a non-smooth operator, and $u^{\ell/h}_j$ will be a low/high regularity norm of a function that an operator is applied to. The important feature of such functions $L$ is that they are linear in the $u^{\ell/h}_j$, and all $p^h_j,u^h_j$, corresponding to high regularity norms, only appear in the first power.

\subsection{Elliptic regularity}
\label{SubsecTameElliptic}

Concretely, we have the following quantitative {\em microlocal} elliptic estimate, which will be used to control solutions to equations $Pu=f$ microlocally away from the characteristic set of $P$.

\begin{prop}
\label{PropTameElliptic}
  (Cf.\ \cite[Theorem~5.1]{HintzQuasilinearDS}.) Let $m,s,r\in\R$ and $\zeta_0\in\Sb^*\Rnhalf$. Suppose $P'=p'(z,\Db)\in\Hb^s\Psib^m(\Rnhalf)$ has a homogeneous principal symbol $p'_m$. Moreover, let $R\in\fpsib^{m-1;0}\Hb^{s-1}(\Rnhalf)$. Let $P=P'+R$, and suppose $p_m\equiv p'_m$ is elliptic at $\zeta_0$. Then there exists $B\in\Psib^0(\Rnhalf)$ elliptic at $\zeta_0$ such that the following holds: Let $\tilde s\in\R$ be such that $\tilde s\leq s-1$ and $s>n/2+1+(-\tilde s)_+$, and suppose that $u\in\Hb^{\tilde s+m-1,r}(\Rnhalf)$ satisfies
	\[
	  Pu=f\in\Hb^{\tilde s,r}(\Rnhalf).
	\]
	Then $Bu\in\Hb^{\tilde s+m}$, and for $\tilde s\leq n/2+t$, $t>0$, the estimate
	\begin{equation}
	\label{EqTameEllLow}
	\begin{split}
	  \|Bu\|_{\tilde s+m,r}&\leq C(\|P'\|_{(m;1),n/2+1+(-\tilde s)_+ +t},\|R\|_{m-1,n/2+(-\tilde s)_+ +t}) \\
	    &\hspace{1cm}\times(\|u\|_{\tilde s+m-1,r}+\|f\|_{\tilde s,r})
	\end{split}
	\end{equation}
	holds; recall here that $P'\in\Hb^s\Psib^m\subset\Psi^{m;\infty}\Hb^s$, justifying the use of the norm on $P'$ in this estimate. For $\tilde s>n/2$, $\eps>0$, there is a tame estimate
	\begin{equation}
	\label{EqTameEllHigh}
	\begin{split}
	  \|Bu\|_{\tilde s+m,r}&\leq L(\|P'\|_{(m;1),n/2+1+\eps},\|R\|_{m-1,n/2+\eps}; \|P'\|_{(m;1),s},\|R\|_{m-1,s-1}; \\
	    &\hspace{3cm} \|u\|_{n/2+m-1+\eps,r}, \|f\|_{n/2-1+\eps,r}; \|u\|_{\tilde s+m-1,r}, \|f\|_{\tilde s,r}).
	\end{split}
	\end{equation}
\end{prop}

Let us point out that in our application of such an estimate to the study of nonlinear equations it will be irrelevant what exactly the low regularity norms in \eqref{EqTameEllHigh} are; in fact, it will be sufficient to know that there is \emph{some} tame estimate of the general form \eqref{EqTameEllHigh}, and this in turn is in fact clear without any computation, namely it follows directly from the fact that we have tame estimates for all `non-smooth' operations involved in the proof of this proposition. Thus, a simpler, but less precise, statement of Proposition~\ref{PropTameElliptic} is the following: Let $P\in\Hb^s\Psib^m$ (we ignore $R$ for brevity) be elliptic at $\zeta_0$, then for sufficiently large $d>0$, and for $s,\tilde s$ with $s\geq\tilde s+d$, $\tilde s\geq d$, one has an estimate
\[
  \|Bu\|_{\tilde s+m} \leq C(\|P\|_{(m;1),d})\bigl(\|u\|_{\tilde s+m-1} + \|f\|_{\tilde s} + \|P\|_{(m;1),s}(\|u\|_d+\|f\|_d)\bigr),
\]
likewise with the weight $r$ added. Similar remarks apply to all further tame microlocal regularity results below. The only point where the precise numerology does matter is when one wants to find an explicit bound on the number of required derivatives for the forcing term in Theorems~\ref{ThmIntroWave}, \ref{ThmIntroGeneralKG} and \ref{ThmIntroGeneralWave}, as we will do.

\begin{proof}[Proof of Proposition~\ref{PropTameElliptic}]
  We can assume that $r=0$ by conjugating $P$ by $x^{-r}$. Choose $a_0\in S^0$ elliptic at $\zeta_0$ such that $p_m$ is elliptic (and non-vanishing, which only matters near the zero section) on $\supp a_0$. Let $\Lambda_m\in\Psib^m$ be a b-ps.d.o with full symbol $\lambda_m(\zeta)$ independent of $z$, whose principal symbol is $\la\zeta\ra^m$, and define
  \[
    q(z,\zeta):=a_0(z,\zeta)\lambda_m(\zeta)/p_m(z,\zeta)\in S^{0;\infty}\Hb^s,\quad Q=q(z,\Db),
  \]
  then by Proposition~\ref{PropHbRec} and Corollary~\ref{CorHbModule}, we have
  \begin{equation}
  \label{EqQEstimate}
    \|Q\|_{(0;k),\sigma}\leq C(\|P'\|_{(m;k),n/2+1+\eps})(1+\|P'\|_{(m;k),\sigma}),\quad\sigma>n/2+1,\eps>0.
  \end{equation}
  Put $B=a_0(z,\Db)\Lambda_m$, then
  \[
    Q\circ P'=B+R'
  \]
  with $R'\in\fpsi^{m-1;0}\Hb^{s-1}$; by Proposition~\ref{PropNonsmoothComp}, we have for $n/2<\sigma\leq s-1$
  \begin{equation}
  \label{EqRPrimeEstimate}
    \|R'\|_{m-1,\sigma} \lesssim \|Q\|_{(0;1),\mu}\|P'\|_{(m;1),\sigma+1}+\|Q\|_{(0;1),\sigma}\|P'\|_{(m;1),\mu+1}, \quad\mu>n/2.
  \end{equation}
  Now, since $Bu=QP'u-R'u=Qf-QRu-R'u$, we need to estimate the $\Hb^{\tilde s}$-norms of $Qf$, $QRu$ and $R'u$, which we will do using Proposition~\ref{PropMapping}. In the low regularity regime, we have, for $t>0$ and $\tilde s\leq n/2+t$, using \eqref{EqQEstimate} and \eqref{EqRPrimeEstimate}:
  \begin{align*}
    \|Qf\|_{\tilde s}&\lesssim \|Q\|_{0,n/2+(-\tilde s)_+ +t}\|f\|_{\tilde s}\leq C(\|P'\|_{m,n/2+1+(-\tilde s)_+ +t})\|f\|_{\tilde s}, \\
	\|R'u\|_{\tilde s}&\lesssim \|R'\|_{m-1,n/2+(-\tilde s)_+ +t}\|u\|_{\tilde s+m-1} \\
	  &\hspace{1cm}\leq C(\|P'\|_{(m;1),n/2+1+(-\tilde s)_+ +t})\|u\|_{\tilde s+m-1}, \\
	\|QRu\|_{\tilde s}&\leq C(\|P'\|_{m,n/2+1+(-\tilde s)_+ +t})\|R\|_{m-1,n/2+(-\tilde s)_+ +t}\|u\|_{\tilde s+m-1},
  \end{align*}
  giving \eqref{EqTameEllLow}. In the high regularity regime, in fact for $0\leq\tilde s\leq s-1$, we have, for $\eps>0$,
  \begin{align*}
    \|Qf\|_{\tilde s}&\lesssim \|Q\|_{0,n/2+\eps}\|f\|_{\tilde s}+\|Q\|_{0,s}\|f\|_{n/2-1+\eps} \\
	  &\hspace{1cm}\leq C(\|P'\|_{m,n/2+1+\eps})(\|f\|_{\tilde s}+(1+\|P'\|_{m,s})\|f\|_{n/2-1+\eps}), \\
	\|R'u\|_{\tilde s}&\lesssim\|R'\|_{m-1,n/2+\eps}\|u\|_{\tilde s+m-1}+\|R'\|_{m-1,s-1}\|u\|_{n/2+m-1+\eps} \\
	   & \leq C(\|P'\|_{(m;1),n/2+1+\eps})(\|u\|_{\tilde s+m-1}+(1+\|P'\|_{(m;1),s})\|u\|_{n/2+m-1+\eps}), \\
    \|QRu\|_{\tilde s}&\leq L(\|P'\|_{m,n/2+1+\eps},\|R\|_{m-1,n/2+\eps};\|P'\|_{m,s},\|R\|_{m-1,s-1}; \\
	  &\hspace{6cm} \|u\|_{n/2+m-1+\eps}; \|u\|_{\tilde s+m-1}),
  \end{align*}
  giving \eqref{EqTameEllHigh}. The proof is complete.
\end{proof}

There is a similar tame microlocal elliptic estimate for operators of the form $P=P'+P''+R$ with $P',R$ as above and $P''\in\Psib^m$, as in part (2) of \cite[Theorem~5.1]{HintzQuasilinearDS}, where the tame estimate now also involves the $C^N$-norm of the `smooth part' $P''$ of the operator for some ($s$-dependent) $N$. Since in our application $P''$ will only depend on finitely many complex parameters, there is no need to prove an estimate which is also tame with respect to the $C^N$-norm of $P''$; however, this could easily be done in principle.

\subsection{Real principal type propagation; radial points}
\label{SubsecTamePropagation}

Tame estimates for real principal type propagation and propagation near radial points can be deduced from a careful analysis of the proofs of the corresponding results in \cite{HintzQuasilinearDS}. The main observation is that the regularity requirements, given in the footnotes to the proofs of these results in \cite{HintzQuasilinearDS}, indicate what regularity is needed to estimate the corresponding terms: For example, an operator in $A\in\Psi^{m;0}\Hb^s$ with $m\geq 0$ maps $\Hb^{m/2}$ to $\Hb^{-m/2}$ under the condition $s>n/2+m/2$, which is to say that one has a bound
\[
  \|Au\|_{-m/2}\lesssim\|A\|_{m-1,n/2+m/2+\eps}\|u\|_{m/2},\quad\eps>0.
\]
This means that the only places where one needs to use tame operator bounds for operators with coefficients of regularity $s$ are those where the condition for mapping properties etc.\ to hold reads $s\gtrapprox\sigma$ where $\sigma$ is the regularity of the target space, i.e.\ where $\sigma$ is comparable to the regularity $s$ of the coefficients.

We again only prove the tame real principal type estimate in the interior; the estimate near the boundary is proved in the same way, see also the discussion at the end of \S\ref{SubsecTameElliptic}.

\begin{prop}
\label{PropTameRealPrType}
  (Cf.\ \cite[Theorem~6.6]{HintzQuasilinearDS}.) Let $m,r,s,\tilde s\in\R$. Suppose $P_m\in\Hb^s\Psib^m(\Rnhalf)$ has a real, scalar, homogeneous principal symbol $p_m$, and let $P_{m-1}\in\Hb^{s-1}\Psib^{m-1}(\Rnhalf)$, $R\in\fpsib^{m-2;0}\Hb^{s-1}(\Rnhalf)$. Let $P=P_m+P_{m-1}+R$. Suppose $s$ and $\tilde s$ are such that
  \[
    \tilde s\leq s-1,\quad s>n/2+7/2+(2-\tilde s)_+.
  \] 
  Let $\gamma\colon[0,T]\to\Sb^*\Rnhalf$ be a segment of a null-bicharacteristic of $p_m$ with $\gamma(0)=\zeta_0$. Then for all $A\in\Psib^0$ elliptic at $\zeta_0$, there exist $B\in\Psib^0$ elliptic at $\gamma(T)$ and $G\in\Psib^0$ elliptic on $\gamma([0,T])$ such that the following holds: If $u\in\Hb^{\tilde s+m-3/2,r}(\Rnhalf)$ satisfies $Pu=f\in\Hb^{\tilde s,r}(\Rnhalf)$, and $\zeta_0\notin\WFb^{\tilde s+m-1,r}(u)$,\footnote{Recall that for $u\in\Hb^{-\infty,r}$, we say $\zeta_0\notin\WFb^{\sigma,r}(u)$ if and only if there exists $A\in\Psib^0$, elliptic at $\zeta_0$, such that $Au\in\Hb^{\sigma,r}$.} then $\gamma(t)\notin\WFb^{\tilde s+m-1,r}(u)$ for all $t\in[0,T]$; quantitatively, for $\tilde s\leq n/2+1$, $\eps>0$,
  \begin{equation}
  \label{EqTameRealPrTypeLow}
  \begin{split}
    \|&Bu\|_{\tilde s+m-1,r} \\
	  &\leq C(\|P_m\|_{\Hb^{n/2+7/2+(2-\tilde s)_+ +\eps}\Psib^m}, \|P_{m-1}\|_{\Hb^{n/2+1+(3/2-\tilde s)_+ +\eps}\Psib^{m-1}},\|R\|_{n/2+1+(-\tilde s)_+}) \\
	  &\hspace{1cm} \times (\|u\|_{\tilde s+m-3/2,r}+\|Au\|_{\tilde s+m-1,r}+\|Gf\|_{\tilde s,r}).
  \end{split}
  \end{equation}
  Moreover, for $\tilde s>n/2+1$, $\eps>0$, there is a tame estimate
  \begin{equation}
  \label{EqTameRealPrTypeHigh}
  \begin{split}
    \|Bu\|_{\tilde s+m-1,r}&\leq L(\|P_m\|_{\Hb^{n/2+7/2+\eps}\Psib^m}, \|P_{m-1}\|_{\Hb^{n/2+1+\eps}\Psib^{m-1}}, \|R\|_{n/2+\eps}; \\
	  &\hspace{1cm} \|P_m\|_{\Hb^s\Psib^m},\|P_{m-1}\|_{\Hb^{s-1}\Psib^{m-1}},\|R\|_{m-2,s-1}; \\
	  &\hspace{1cm} \|u\|_{n/2-1/2+m+\eps}; \|u\|_{\tilde s+m-3/2,r},\|Au\|_{\tilde s+m-1,r},\|Gf\|_{\tilde s,r}).
  \end{split}
  \end{equation}
\end{prop}

A simpler, but less precise, statement of Proposition~\ref{PropTameRealPrType}, ignoring $P_{m-1}$ and $R$ for brevity, is the following: For sufficiently large $d>0$, and for $s,\tilde s$ with $s\geq\tilde s+d$, $\tilde s\geq d$, one has an estimate
\begin{align*}
  \|Bu\|_{\tilde s+m-1} \leq C(\|P_m\|_{\Hb^d\Psib^m})\bigl(&\|u\|_{\tilde s+m-3/2} + \|f\|_{\tilde s} \\
    &\qquad + \|P_m\|_{\Hb^s\Psib^m}(\|Au\|_d+\|Gf\|_d)\bigr),
\end{align*}
likewise with the weight $r$ added.

\begin{proof}[Proof of Proposition~\ref{PropTameRealPrType}]
  We follow the proof of the regularity result in \cite{HintzQuasilinearDS} and state the estimates needed to establish \eqref{EqTameRealPrTypeLow} and \eqref{EqTameRealPrTypeHigh} along the way. Using the notation of the proof of \cite[Theorem~6.6]{HintzQuasilinearDS}, but now calling the regularization parameter $\delta$, in particular $\check A_\delta\in\Psib^{\tilde s+(m-1)/2}$ is the regularized commutant, which depends on a positive constant $M$ chosen below, and putting $\tilde f=f-Ru$, we have, assuming $m\geq 1$ and $\tilde s\geq(5-m)/2$ for now,
  \begin{align*}
    \Re&\la i\check A_\delta^*[P_m,\check A_\delta]u,u\ra \\
	  &=\frac{1}{2}\la i(P_m-P_m^*)\check A_\delta u,\check A_\delta u\ra - \Re\la i\check A_\delta\tilde f,\check A_\delta u\ra+\Re\la i\check A_\delta P_{m-1}u,\check A_\delta u\ra \\
	  &\equiv I+II+III.
  \end{align*}
  For $\eps>0$, we can bound the first term by
  \[
    |I| \lesssim \|P_m\|_{\Hb^{n/2+1+(m-1)/2+\eps}\Psib^m}\|\check A_\delta u\|_{(m-1)/2}^2,
  \]
  the second one by
  \[
	|II| \lesssim \|\check A_\delta f\|_{-(m-1)/2}^2 + \|Ru\|_{\tilde s}^2+\|\check A_\delta u\|_{(m-1)/2}^2,
  \]
  where in turn
  \[
    \|Ru\|_{\tilde s}\lesssim
	  \begin{cases}
	    \|R\|_{m-2;n/2+(-\tilde s)_+ +t}\|u\|_{\tilde s+m-2}, & \tilde s\leq n/2+t, \\
		\|R\|_{m-2;n/2+\eps}\|u\|_{\tilde s+m-2} + \|R\|_{m-2;s-1}\|u\|_{n/2+m-2+\eps}, & \tilde s\geq 0
	  \end{cases}
  \]
  for $t>0$ by Proposition~\ref{PropMapping}. We estimate the third term by
  \[
    |III| \lesssim \|P_{m-1}\|_{\Hb^{\max(n/2+\eps,(m-1)/2)}\Psib^{m-1}}\|\check A_\delta u\|_{(m-1)/2}^2+|\la[\check A_\delta,P_{m-1}]u,\check A_\delta u\ra|
  \]
  and further, with $R_2\in\Psib^{\tilde s+(m-1)/2-1}\circ\Psi^{m-1;0}\Hb^{s-2}$ denoting a part of the expansion of $[\check A_\delta,P_{m-1}]$ as defined after \cite[Footnote~28]{HintzQuasilinearDS},
  \begin{align*}
    |\la[\check A_\delta,P_{m-1}]u,\check A_\delta u\ra|&\leq C(M)\|P_{m-1}\|_{\Hb^{n/2+1+(m/2-1)_+ +\eps}\Psib^{m-1}}\|u\|_{\tilde s+m-3/2}^2 \\
	  &\hspace{3cm}+\|R_2 u\|_{-(m-1)/2}^2+\|\check A_\delta u\|_{(m-1)/2}^2,
  \end{align*}
  where
  \begin{align*}
    \|R_2&u\|_{-(m-1)/2} \\
	& \leq C(M)
	   \begin{cases}
	     \|P_{m-1}\|_{\Hb^{n/2+1+(1-\tilde s)_+ +\eps}\Psib^{m-1}}\|u\|_{\tilde s+m-2}, & \tilde s\leq n/2+1+\eps,\\
		 \|P_{m-1}\|_{\Hb^{n/2+1+\eps}\Psib^{m-1}}\|u\|_{\tilde s+m-2} \\
		 \hspace{1cm}+\|P_{m-1}\|_{\Hb^{s-1}\Psib^{m-1}}\|u\|_{n/2+m-1+\eps}, & \tilde s\geq 1.
	   \end{cases}
  \end{align*}
  Therefore, we obtain, see \cite[Equation~(6.24)]{HintzQuasilinearDS},
  \begin{equation}
  \label{EqOperatorCommutator}
  \begin{split}
    \Re&\left\la\Bigl(i\check A_\delta^*[P_m,\check A_\delta]+B_\delta^*B_\delta+M^2(\Lambda\check A_\delta)^*(\Lambda\check A_\delta)-E_\delta\Bigr)u,u\right\ra \\
	  &\hspace{2cm}\geq -|\la E_\delta u,u\ra|-\|\check A_\delta f\|_{-(m-1)/2}^2-L^2+\|B_\delta u\|_{L^2_\bl}^2,
  \end{split}
  \end{equation}
  where
  \[
    M=M(\|P_m\|_{\Hb^{n/2+1+(m-1)/2+\eps}\Psib^m}, \|P_{m-1}\|_{\Hb^{\max(n/2+\eps,(m-1)/2)}\Psib^{m-1}}),
  \]
  and $L$ is `tame'; more precisely, for $\tilde s\leq n/2+t$, $t>0$,
  \begin{align*}
    L&\leq C(M,\|P_{m-1}\|_{\Hb^{n/2+1+\max(m/2-1,1-\tilde s)_+ +\eps}\Psib^{m-1}},\|R\|_{m-2;n/2+(-\tilde s)_+ +t})\|u\|_{\tilde s+m-3/2},
  \end{align*}
  and for $\tilde s\geq 1$,
  \begin{align*}
    L&=L(M,\|P_{m-1}\|_{\Hb^{n/2+1+(m/2-1)_+ +\eps}\Psib^{m-1}},\|R\|_{m-2,n/2+\eps}; \\
	  &\hspace{1cm}\|P_{m-1}\|_{\Hb^{s-1}\Psib^{m-1}},\|R\|_{m-2;s-1}; \|u\|_{n/2+m-1+\eps};\|u\|_{\tilde s+m-3/2}).
  \end{align*}
  Next, in order to exploit the positive commutator of the principal symbols of $P_m$ and $\check A_\delta$ in the estimate \eqref{EqOperatorCommutator}, we introduce operators $J^\pm\in\Psib^{\pm(\tilde s+(m-1)/2-1)}$ with principal symbols $j^\pm$ such that $J^+J^- -I\in\Psib^{-\infty}$; then
  \[
    iJ^-\check A_\delta^*[P_m,\check A_\delta]=\Op(j^-\check a_\delta H_{p_m}\check a_\delta)+R_1+R_2+R_3+R_4,
  \]
  see \cite[Equation~(6.27)]{HintzQuasilinearDS}, where
  \[
    |\la R_j u,(J^+)^*u\ra|\leq C(M) \|P_m\|_{\Hb^{n/2+2+m/2+\eps}\Psib^m}\|u\|_{\tilde s+m-3/2}^2,\quad j=1,3,4,
  \]
  and $R_2\in\Psib^{\tilde s+(m-1)/2-1}\circ\Psi^{m;0}\Hb^{s-2}$, hence
  \begin{align*}
    |\la R_2&u,(J^+)^*u\ra| \\
	  &\leq C(M)
	    \begin{cases}
		  (1+\|P_m\|_{\Hb^{n/2+2+(3/2-\tilde s)_+ +\eps}\Psib^m}^2)\|u\|_{\tilde s+m-3/2}^2 & \forall\ \tilde s, \\
		  (1+\|P_m\|_{\Hb^{n/2+2+\eps}\Psib^m}^2)\|u\|_{\tilde s+m-3/2}^2 \\
		  \hspace{3cm}+\|P_m\|_{\Hb^s\Psib^m}^2\|u\|_{n/2-1/2+m+\eps}^2 & \tilde s\geq 3/2.
		\end{cases}
  \end{align*}
  Thus, further following the proof in \cite{HintzQuasilinearDS} to equation (6.28) and beyond, it remains to bound
  \[
    \Re\la\Op(j^-f_\delta/j^+)(J^+)^*u,(J^+)^*u\ra + \Re\la R'u,(J^+)^*u\ra,\quad R'\in\Psi^{\tilde s+3(m-1)/2;0}\Hb^{s-1},
  \]
  from below, which is accomplished by
  \begin{gather*}
    |\la R'u,(J^+)^*u\ra|\leq C(M)\|P_m\|_{\Hb^{n/2+1+m/2+\eps}\Psib^m}\|u\|_{\tilde s+m-3/2}^2, \\
	\Re\la\Op(j^-f_\delta/j^+)(J^+)^*u,(J^+)^*u\ra \geq -C(M)\|P_m\|_{\Hb^{n/2+3+m/2+\eps}\Psib^m}\|u\|_{\tilde s+m-3/2}^2.
  \end{gather*}
  Lastly, for general $m\in\R$, we rewrite the equation $Pu=f$ as $P\Lambda^+(\Lambda^- u)=f+PRu$ with $\Lambda^\pm\in\Psib^{\mp(m-m_0)}$, $R\in\Psib^{-\infty}$, where $m_0\geq 1$; hence, replacing $P$ by $P\Lambda^+$, $u$ by $\Lambda^- u$ and $m$ by $m_0$ in the above estimates is equivalent to just replacing $m$ by $m_0$ in the b-Sobolev norms of the coefficients of $P$. Choosing $m_0=1+2(2-\tilde s)_+$ as in \cite{HintzQuasilinearDS} then implies the estimates \eqref{EqTameRealPrTypeLow} and \eqref{EqTameRealPrTypeHigh} with $B=B_0$, $G$ an elliptic multiple of $\check A_0$, and $A$ elliptic on the microsupport of $E_0$.
\end{proof}

\begin{figure}[!ht]
  \centering
  \includegraphics{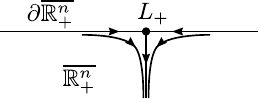}
  \caption{Null-bicharacteristic flow near a radial set, here a sink $L_+\subset\Sb^*_{\pa\Rnhalf}\Rnhalf$ for the flow within the boundary $\Sb^*_{\pa\Rnhalf}\Rnhalf$, with an unstable direction transversal to the boundary. At a source $L_-$ within the boundary, with transversal stable direction, the direction of the flow would be reversed.}
\label{FigRadial}
\end{figure}

In a similar manner, we can analyze the proof of the radial point estimate, see Figure~\ref{FigRadial} for an illustration, obtaining, in the notation of \cite[\S{6.4}]{HintzQuasilinearDS}:
\begin{prop}
\label{PropTameRadial}
  Let $m,r,s,\tilde s\in\R$, $\alpha>0$. Let $P=P_m+P_{m-1}+R$, where $P_j=P'_j+P''_j$, $j=m,m-1$, with $P'_m\in\Hb^{s,\alpha}\Psib^m(\Rnhalf)$ and $P''_m\in\Psib^m(\Rnhalf)$ having real, scalar, homogeneous principal symbols $p'_m$ and $p''_m$, respectively; moreover $P'_{m-1}\in \Hb^{s-1,\alpha}\Psib^{m-1}(\Rnhalf)$, $P''_{m-1}\in\Psib^{m-1}(\Rnhalf)$ and $R=R'+R''$ with $R'\in\Psib^{m-2;0}\Hb^{s-1,\alpha}(\Rnhalf)$ and $R''\in\Psib^{m-2}(\Rnhalf)$. Suppose that the conditions (1)-(4) in \cite[\S{6.4}]{HintzQuasilinearDS} hold for $p=p''_m$, and
  \begin{equation}
  \label{EqRadialSubpr}
    \sigma_{\bl,m-1}\left(\frac{1}{2i}\Bigl((P''_m+P''_{m-1})-(P''_m+P''_{m-1})^*\Bigr)\right)=\pm\hat\beta\beta_0\rho^{m-1}\tn{ at }L_\pm,
  \end{equation}
  where $\hat\beta\in\CI(L_\pm)$ is self-adjoint at every point. Finally, assume that $s$ and $\tilde s$ satisfy
  \begin{equation}
  \label{EqRadialCond}
    \tilde s\leq s-1, \quad s>n/2+7/2+(2-\tilde s)_+.
  \end{equation}
  Suppose $u\in\Hb^{\tilde s+m-3/2,r}(\Rnhalf)$ is such that $Pu=f\in\Hb^{\tilde s,r}(\Rnhalf)$.
  \begin{enumerate}
    \item \label{EnumThmRadIntoBdy} If $\tilde s+(m-1)/2-1+\inf_{L_\pm}(\hat\beta-r\tilde\beta)>0$, let us assume that in a neighborhood of $L_\pm$, $\cL_\pm\cap\{x>0\}$ is disjoint from $\WFb^{\tilde s+m-1,r}(u)$.
	\item \label{EnumThmRadFromBdy} If $\tilde s+(m-1)/2+\sup_{L_\pm}(\hat\beta-r\tilde\beta)<0$, let us assume that a punctured neighborhood of $L_\pm$, with $L_\pm$ removed, in $\Sigma\cap\Sb^*_{\pa\Rnhalf}\Rnhalf$ is disjoint from $\WFb^{\tilde s+m-1,r}(u)$.
  \end{enumerate}
  Then in both cases, $L_\pm$ is disjoint from $\WFb^{\tilde s+m-1,r}(u)$.
  
  Quantitatively, for every neighborhood $U$ of $L_\pm$, there exist $B_0,B_1\in\Psib^0$ elliptic at $L_\pm$, $A\in\Psib^0$ with microsupport in the respective a priori control region in the two cases above, with\footnote{We recall that for an operator $A=\Op(a)\in\Psib^m$, the operator wave front set $\WFb'(A)\subset\Sb^*\Rnhalf$ is the complement of the largest open set in which the symbol $a$ is of order $-\infty$.} $\WFb'(A),\WFb'(B_j)\subset U$, $j=1,2$, and $\chi\in\CIc(U)$, such for $\tilde s\leq n/2+1$, $\eps>0$, we have, with implicit dependence of the appearing constants on seminorms of the smooth operators $P''_m,P''_{m-1}$ and $R''$:
  \begin{equation}
  \label{EqTameRadialLow}
  \begin{split}
    \|B_0&u\|_{\tilde s+m-1,r}\leq C(\|P'_m\|_{\Hb^{n/2+7/2+(2-\tilde s)_+ +\eps,\alpha}\Psib^m}, \\
	  &\hspace{1cm}\|P'_{m-1}\|_{\Hb^{n/2+1+(3/2-\tilde s)_+ +\eps,\alpha}\Psib^{m-1}},\|R'\|_{m-2,n/2+1+(-\tilde s)_+}) \\
	  & \times (\|u\|_{\tilde s+m-3/2,r}+\|Au\|_{\tilde s+m-1,r}+\|B_1 f\|_{\tilde s,r}+\|\chi f\|_{\tilde s-1,r}).
  \end{split}
  \end{equation}
  Moreover, for $\tilde s>n/2+1$, $\eps>0$, there is a tame estimate
  \begin{equation}
  \label{EqTameRadialHigh}
  \begin{split}
    \|B_0&u\|_{\tilde s+m-1,r}\leq L(\|P'_m\|_{\Hb^{n/2+7/2+\eps,\alpha}\Psib^m}, \|P'_{m-1}\|_{\Hb^{n/2+1+\eps,\alpha}\Psib^{m-1}}, \|R'\|_{m-2,n/2+\eps}; \\
	  & \|P'_m\|_{\Hb^{s,\alpha}\Psib^m},\|P'_{m-1}\|_{\Hb^{s-1,\alpha}\Psib^{m-1}},\|R'\|_{m-2,s-1}; \|u\|_{n/2-1/2+m+\eps}, \|f\|_{n/2-1+\eps}; \\
	  & \|u\|_{\tilde s+m-3/2,r},\|Au\|_{\tilde s+m-1,r},\|B_1 f\|_{\tilde s,r},\|\chi f\|_{\tilde s-1,r}).
  \end{split}
  \end{equation}
\end{prop}
\begin{proof}
  One detail changes as compared to the previous proof: While it still suffices to only assume microlocal regularity $B_2 f\in\Hb^{\tilde s,r}$ at $L_\pm$, we now in addition need to assume local regularity $\chi f\in\Hb^{\tilde s-1,r}$, which is due to the use of elliptic regularity in the proof given in \cite{HintzQuasilinearDS}.
\end{proof}

\begin{rmk}
\label{RmkRadialPointsMfds}
  The above proposition provides tame estimates at (generalized) radial sets provided an open neighborhood of the radial set can be embedded as an open subset of $\Rnhalf$. While this holds for Kerr-de Sitter spaces, it does not hold in general. Thus, making the non-smooth positive commutator argument (which is not (micro)local within the radial set) work in general either requires analyzing the behavior of non-smooth operators under changes of coordinates and building a suitable invariant calculus; or, more easily, if we assume that $P=P_m+P_{m-1}+R$ with $P_{m-j}\in\Hb^{s-j,\alpha}\Psib^{m-j}$ (as before), $j=0,1$, and $R\in\Hb^{s-1,\alpha}\Psib^{m-2}$, thus making a manifestly invariant assumption on $R$ and thus on $P$, one can directly introduce localizers in the argument and estimate the resulting error terms, see Remark~\ref{RmkNonsmoothTrappingMfds} for details.
\end{rmk}

\subsection{Non-trapping estimates at normally hyperbolic trapping}
\label{SubsecNontrapping}

We now extend the proof of non-trapping estimates on weighted b-Sobolev spaces at normally hyperbolically trapped sets given in \cite[Theorem~3.2]{HintzVasyNormHyp} to the non-smooth setting.

To set this up, let $P_0\in\Psib^m(\Rnhalf)$ with
\begin{equation}
\label{EqP0Subprincipal}
  \frac{1}{2i}(P_0-P_0^*)=E_1\in\Psib^{m-1}(\Rnhalf),
\end{equation}
where the adjoint is taken with respect to a fixed smooth b-density; an example to keep in mind here and in what follows is $P_0=\Box_g$ for a smooth Lorentzian b-metric $g$ on $\Rnhalf$, considered a coordinate patch of Kerr-de Sitter space, in which case $E_1=0$, and the threshold weight in Theorem~\ref{ThmTameHypTr} below is $r=0$. Let $p_0$ be the principal symbol of $P_0$. Let us use the coordinates $(z;\zeta)=(x,y;\lambda,\eta)$ on $\Tb^*\Rnhalf$ and write $M=\Rnhalf,X=\pa\Rnhalf$. With $\Sigma\subset\Sb^* M$ denoting the characteristic set of $P_0$, we make the following assumptions:
\begin{enumerate}
  \item \label{EnumNTGamma} $\Gamma\subset\Sigma\cap\Sb^*_X M$ is a smooth submanifold disjoint from the image of $T^*X\wozero$, so $xD_x$ is elliptic near $\Gamma$,
  \item $\Gamma_+$ is a smooth submanifold of $\Sigma\cap\Sb^*_X M$ in a neighborhood $U_1$ of $\Gamma$,
  \item $\Gamma_-$ is a smooth submanifold of $\Sigma$ transversal to $\Sigma\cap\Sb^*_X M$ in $U_1$,
  \item $\Gamma_+$ has codimension $2$ in $\Sigma$, $\Gamma_-$ has codimension $1$,
  \item $\Gamma_+$ and $\Gamma_-$ intersect transversally in $\Sigma$ with $\Gamma_+\cap\Gamma_-=\Gamma$,
  \item the vector field $V$ is tangent to both $\Gamma_+$ and $\Gamma_-$, and thus to $\Gamma$,
  \item \label{EnumNTTrapped} $\Gamma_+$ is backward trapped for the Hamilton flow, $\Gamma_-$ is forward trapped; in particular, $\Gamma$ is a trapped set.
\setcounter{CounterEnumi}{\value{enumi}}
\end{enumerate}

See Figure~\ref{FigTrapping} for the setting.

\begin{figure}[!ht]
  \centering
  \includegraphics{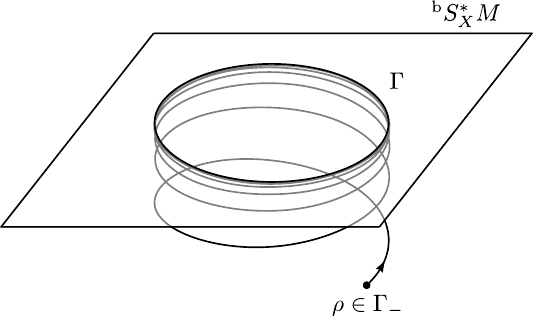}
  \caption{An exemplary situation with trapping: Shown are the (projection from $\Sb^*M$ to the base $M$ of the) trapped set $\Gamma$, the b-cosphere bundle over $X$ as well as a forward bicharacteristic starting at a point $\rho\in\Gamma_-$.}
\label{FigTrapping}
\end{figure}

In view of condition \itref{EnumNTGamma}, we can take
\[
  \rho=\la\lambda\ra=(1+\lambda^2)^{1/2}\tn{ near }\Gamma,
\]
appropriately extended to $\rcTb^*M$, as the inverse of a boundary defining function of fiber infinity $\Sb^*M$ in $\rcTb^*M$.\footnote{We are working locally here, but we remark that invariantly, $\lambda$ is the b-principal symbol (restricted to $\Tb^*_X M\wozero$) of the operator $xD_x$, $x$ a boundary defining function, which is independent of the choice of $x$.} Then, let
\[
  V=\rho^{-m+1}H_{p_0},
\]
be the rescaled Hamilton vector field of $p_0$. We make quantitative assumptions related to condition \itref{EnumNTTrapped}: Let $\phi_+\in\CI(\Sb^*M)$ be a defining function of $\Gamma_+$ in $\Sb^*_X M$, and let $\phi_-\in\CI(\Sb^*M)$ be a defining function of $\Gamma_-$. Thus, $\Gamma_+$ is defined within $\Sb^*M$ by $x=0,\phi_+=0$. Let
\[
  \hat p_0=\rho^{-m}p_0.
\]
We then assume that
\begin{enumerate}
\setcounter{enumi}{\value{CounterEnumi}}
  \item $\phi_+$ and $\phi_-$ satisfy
    \begin{equation}
	\label{EqVphipm}
	  V\phi_+=-c_+^2\phi_+ + \mu_+ x + \nu_+\hat p_0, \quad V\phi_-=c_-^2\phi_- + \nu_-\hat p_0,
	\end{equation}
	with $c_\pm>0$ smooth near $\Gamma$ and $\mu_+,\nu_\pm$ smooth near $\Gamma$. This is consistent with the (in)stability of $\Gamma_-$ ($\Gamma_+$),
  \item $x$ satisfies
    \begin{equation}
	\label{EqVx}
	  Vx=-c_\pa x, \quad c_\pa > 0,
	\end{equation}
	which is consistent with the stability of $\Gamma_-$,
  \item \label{EnumNTQuant3} near $\Gamma$,
    \begin{equation}
	\label{EqVrho}
	  \rho^{-1}V\rho=c_f x
	\end{equation}
	for some smooth $c_f$, which holds in view of our choice of $\rho$.
\end{enumerate}

Here we recall from \cite[Lemma~5.1]{DyatlovResonanceProjectors}, see also \cite[Lemma~2.4]{DyatlovNormally}, that in the closely related semiclassical setting (see the discussion prior to Theorem~\ref{ThmNHypSemiGlobal}) one can arrange for any $\ep>0$ that
\begin{equation}
\label{EqNormalExpRates}
  0<\numin -\ep<c_\pm^2<\nu_{\max} +\ep,
\end{equation}
where $\numin$ and $\nu_{\max}$ are the minimal and maximal normal expansion rates; see \cite[Equations~(5.1) and (5.2)]{DyatlovResonanceProjectors} for the definition of the latter, with $\numin$ also given in \eqref{EqNuMinDef} below.\footnote{The functions $c_\pm$ in \cite{DyatlovResonanceProjectors,DyatlovNormally} are called $c_\pm^2$ here.} In particular, if $M$ is replaced by $[0,\infty)\times X$, and if $P_0$ is dilation invariant, then the semiclassical and the b-settings are equivalent via the Mellin transform and a rescaling, see e.g.\ \cite[\S3.1]{VasyMicroKerrdS}; since in our general case $c_\pm|_{\Sb^*_XM}$ is what matters, we can replace $P_0$ by $N(P_0)$, and in particular \eqref{EqNormalExpRates} applies, with the expansion rate calculated using $p_0|_{\Tb^*_XM}$.\footnote{We rely on the setup described in \cite{DyatlovNormally} here, which, in contrast to \cite{DyatlovResonanceProjectors}, does not use a normal form for the operator $P_0$.}

We now perturb $P_0$ by a non-smooth operator $\tilde P$, that is, we consider the operator
\begin{equation}
\label{EqNTP} P=P_0+\tilde P,\quad \tilde P=\tilde P_m+\tilde P_{m-1}+\tilde R,
\end{equation}
where for some fixed $\alpha>0$, we have $\tilde P_{m-j}\in\Hb^{s-j,\alpha}\Psib^{m-j}$, $j=0,1$, and $\tilde R\in\Psib^{m-2;0}\Hb^{s-1,\alpha}$.

We then have the following tame non-trapping estimate at $\Gamma$:
\begin{thm}
\label{ThmTameHypTr}
  Using the above notation and making the above assumptions, let $s,\tilde s\in\R$ be such that
  \begin{equation}
  \label{EqNTCond}
    \tilde s\leq s-1, \quad s>n/2+7/2+(2-\tilde s)_+.
  \end{equation}

  Then for $r<-\sup_\Gamma \rho^{-m+1}\sigma_{\bl,m-1}(E_1)/c_\pa$ and for any neighborhood $U$ of $\Gamma$, there exist $B_0\in\Psib^0(M)$ elliptic at $\Gamma$ and $B_1,B_2\in\Psib^0(M)$ with $\WFb'(B_j)\subset U$, $j=0,1,2$, $\WFb'(B_2)\cap\Gamma_+=\emptyset$, and $\chi\in\CIc(U)$, such that the following holds for $u\in\Hb^{\tilde s+m-3/2,r}(\Rnhalf)$ solving the equation $Pu=f\in\Hb^{\tilde s,r}(\Rnhalf)$: For $\tilde s\leq n/2+1$, $\eps>0$,
  \begin{equation}
  \label{EqTameNHypLow}
  \begin{split}
    \|B_0&u\|_{\tilde s+m-1,r}\leq C(\|\tilde P_m\|_{\Hb^{n/2+7/2+(2-\tilde s)_+ +\eps,\alpha}\Psib^m}, \\
	  &\hspace{1cm}\|\tilde P_{m-1}\|_{\Hb^{n/2+1+(3/2-\tilde s)_+ +\eps,\alpha}\Psib^{m-1}},\|\tilde R\|_{m-2,n/2+1+(-\tilde s)_+}) \\
	  & \times (\|u\|_{\tilde s+m-3/2,r}+\|B_2 u\|_{\tilde s+m-1,r}+\|B_1 f\|_{\tilde s,r}+\|\chi f\|_{\tilde s-1,r}).
  \end{split}
  \end{equation}
  Moreover, for $\tilde s>n/2+1$, $\eps>0$, there is a tame estimate
  \begin{equation}
  \label{EqTameNHypHigh}
  \begin{split}
    \|B_0&u\|_{\tilde s+m-1,r}\leq L(\|\tilde P_m\|_{\Hb^{n/2+7/2+\eps,\alpha}\Psib^m}, \|\tilde P_{m-1}\|_{\Hb^{n/2+1+\eps,\alpha}\Psib^{m-1}}, \|\tilde R\|_{m-2,n/2+\eps}; \\
	  & \|\tilde P_m\|_{\Hb^{s,\alpha}\Psib^m},\|\tilde P_{m-1}\|_{\Hb^{s-1,\alpha}\Psib^{m-1}},\|\tilde R\|_{m-2,s-1}; \|u\|_{n/2-1/2+m+\eps}, \|f\|_{n/2-1+\eps}; \\
	  & \|u\|_{\tilde s+m-3/2,r},\|B_2 u\|_{\tilde s+m-1,r},\|B_1 f\|_{\tilde s,r},\|\chi f\|_{\tilde s-1,r}).
  \end{split}
  \end{equation}

  On the other hand, for $r>-\inf_\Gamma \rho^{-m+1}\sigma_{\bl,m-1}(E_1)/c_\pa$ and for appropriate $B_2$ with $\WFb'(B_2)\cap\Gamma_-=\emptyset$, the estimates \eqref{EqTameNHypLow} and \eqref{EqTameNHypHigh} hold as well. These estimates are understood in the sense that if all quantities on the right hand side are finite, then so is the left hand side, and the inequality holds.
\end{thm}

We leave the statement of a less precise version of this Theorem in the spirit of the discussion following the statement of Proposition~\ref{PropTameRealPrType} to the reader.

\begin{proof}[Proof of Theorem~\ref{ThmTameHypTr}]
  The main part of the argument, in particular the choice of the commutant, is a slight modification of the positive commutator argument of \cite[Theorem~3.2]{HintzVasyNormHyp}; the handling of the non-smooth terms is a modification of the proof of the radial point estimate, \cite[Theorem~6.10]{HintzQuasilinearDS}. In particular, the positivity comes from differentiating the weight $x^{-r}$ in the commutant. To avoid working in weighted b-Sobolev spaces for the non-smooth problem, we will conjugate $P$ by $x^{-r}$, giving an advantageous (here meaning negative) contribution to the imaginary part of the subprincipal symbol of the conjugated operator near $\Gamma$.

  \emph{Throughout this proof, we denote operators and their symbols by the corresponding capital and lower case letters, respectively.}
  
  Concretely, put $\sigma=\tilde s+m-1$, and define
  \begin{gather*}
    u_r:=x^{-r}u \in\Hb^{\sigma-1/2}, \quad f_r:=x^{-r}f \in\Hb^{\sigma-m+1}, \\
	  P_r:=x^{-r}Px^r=P_{0,r}+\tilde P_r,\quad P_{0,r}=x^{-r}P_0x^r,\tilde P_r=x^{-r}\tilde Px^r,
  \end{gather*}
  where
  \[
    \tilde P_r=\tilde P_{m,r}+\tilde P_{m-1,r}+\tilde R_r,\quad \tilde P_{m-j,r}\in\Hb^{s-j,\alpha}\Psib^{m-j},\tilde R_r\in\Psib^{m-2;0}\Hb^{s-1,\alpha};
  \]
  then $P_r u_r = f_r$, and we must show a non-trapping estimate for $u_r$ on \emph{unweighted} b-Sobolev spaces. A simple computation shows that
  \[
    \frac{1}{2i}(P_{0,r}-P_{0,r}^*)-\left(\frac{1}{2i}(P_0-P_0^*)-\Op(rx^{-1}H_{p_0} x)\right) \in \Psib^{m-2};
  \]
  but $x^{-1}H_{p_0}x=-\rho^{m-1}c_\pa$ with $c_\pa>0$ near $\Gamma$ by \eqref{EqVx}, hence, using \eqref{EqP0Subprincipal},
  \begin{equation}
  \label{EqNTSubprincipal}
    \frac{1}{2i}(P_{0,r}-P_{0,r}^*)=E_1 + E_1' + B
  \end{equation}
  with $B,E_1'\in\Psib^{m-1}$, where $B$ has principal symbol $b=rc_\pa\rho^{m-1}$ near $\Gamma$, and $\WFb'(E_1')\cap\Gamma=\emptyset$. Notice that by assumption on $r$, $B+E_1$ is elliptic on $\Gamma$.

  We now turn to the positive commutator argument: Fix $0<\beta<\min(1,\alpha)$ and define
  \[
    \rho_+ = \phi_+^2 + x^\beta.
  \]
  Let $\chi_0(t)=e^{-1/t}$ for $t>0$ and $\chi_0(t)=0$ for $t<0$, further $\chi\in\CIc([0,R))$ for $R>0$ to be chosen below, $\chi\equiv 1$ near $0$, $\chi'\leq 0$, and finally $\psi\in\CIc((-R,R))$, $\psi\equiv 1$ near $0$. Define for $\kappa>0$, specified later,
  \[
    a=\rho^{\sigma-(m-1)/2}\chi_0(\rho_+ - \phi_-^2 + \kappa)\chi(\rho_+)\psi(\hat p_0).
  \]
  On $\supp a$, we have $\rho_+\leq R$, thus the argument of $\chi_0$
  is bounded above by $R+\kappa$. Moreover, $\phi_-^2\leq R+\kappa$
  and $x\leq R^{1/\beta}$, therefore $a$ is supported in any given
  neighborhood of $\Gamma$ if one chooses $R$ and $\kappa$
  small. Notice that $a$ is merely a \emph{conormal} symbol which does not grow at the boundary. However, b-analysis for operators with conormal coefficients can easily be seen to work without much additional work, in fact, a logarithmic change of variables essentially reduces such a b-analysis on $\Rnhalf$ to the analysis of operators corresponding to uniform symbols on $\R^n$. Moreover, the proofs of composition results of smooth and non-smooth b-ps.d.o.s presented in \cite{HintzQuasilinearDS} go through without changes if one uses b-ps.d.o.s with non-growing conormal, instead of smooth, symbols.\footnote{A somewhat more direct way of dealing with this issue goes as follows: Assume, as one may, that $\ell:=\beta^{-1}\in\N$. Then even though $a$ is not a smooth symbol of $\Rnhalf$ with the standard smooth structure, it becomes smooth if one changes the smooth structure of $\Rnhalf$ by blowing up the boundary to the $\ell$-th order, i.e.\ by taking $x'=x^\beta$ as a boundary defining function, thus obtaining a manifold $M_\ell$, which is $\Rnhalf$ as a topological manifold, but with a different smooth structure; in particular, the function $x=(x')^\ell$ is smooth on $M_\ell$ in view of $\ell\in\N$. Moreover, the blow-down map $M_\ell\to\Rnhalf$ induces isomorphisms (see e.g.\ \cite[\S{4.18}]{MelroseAPS})
  \[
    \Hb^{s',\gamma}(\Rnhalf)\cong\Hb^{s',\ell\gamma}(M_\ell),\quad s',\gamma\in\R.
  \]
  Therefore, one can continue to work on $\Rnhalf$, tacitly assuming that all functions and operators live on, and all computations are carried out on, $M_\ell$.}

  Define the regularizer $\varphi_\delta(\zeta)=(1+\delta\rho)^{-1}$ near $\Gamma$, and put $a_\delta=\varphi_\delta a$. Put $\tilde V=\rho^{-m+1}H_{\tilde p_{m,r}}$ and define $\tilde c_\pa,\tilde c_f\in\Hb^{s-1,\alpha}$ near $\Gamma$ by $\tilde Vx=-\tilde c_\pa x$, $\rho^{-1}\tilde V\rho=\tilde c_f x$. Then, with $p_{m,r}=p_{0,r}+\tilde p_{m,r}$, we obtain, using \eqref{EqVphipm}-\eqref{EqVrho}:
  \begin{align}
    a_\delta & H_{p_{m,r}} a_\delta = \varphi_\delta^2\rho^{2\sigma}\chi_0^2\chi^2\psi^2(\sigma-(m-1)/2-\delta\rho\varphi_\delta)(c_f+\tilde c_f)x \nonumber\\
	  &\hspace{1cm} - \varphi_\delta^2\rho^{2\sigma}\chi_0\chi_0'\chi^2\psi^2(2c_+^2\phi_+^2+\beta c_\pa x^\beta-2\mu_+\phi_+ x-2\nu_+\phi_+\hat p_0 \nonumber\\
	  &\hspace{4cm}+ 2c_-^2\phi_-^2 +2\nu_-\phi_-\hat p_0 - \tilde V\phi_+^2 + \beta\tilde c_\pa x^\beta + \tilde V\phi_-^2) \nonumber\\
	  &\hspace{1cm}+\varphi_\delta^2\rho^{2\sigma}\chi_0^2\chi\chi'\psi^2(V\rho_++\tilde V\rho_+) + \varphi_\delta^2\rho^{2\sigma}\chi_0^2\chi^2\psi\psi'(V\hat p_0+\tilde V\hat p_0) \nonumber\\
  \label{EqNTCommutator}&=-c_+^2 a_{+,\delta}^2 - c_-^2 a_{-,\delta}^2 + a_{+,\delta}h_{+,\delta}p_{m,r}+a_{-,\delta}h_{-,\delta}p_{m,r} + e_\delta + g_\delta - f_\delta,
  \end{align}
  where, writing $\hat p_0=\rho^{-m}p_{m,r}-\rho^{-m}\tilde p_{m,r}$ in the second and third line,
  \begin{align*}
    a_{\pm,\delta}&=\varphi_\delta\rho^\sigma\sqrt{2\chi_0\chi_0'}\chi\psi\phi_\pm, \\
	h_{\pm,\delta}&=\pm\varphi_\delta\rho^{\sigma-m}\sqrt{2\chi_0\chi_0'}\chi\psi\nu_\pm, \\
	e_\delta&=\varphi_\delta^2\rho^{2\sigma}\chi_0^2\chi\chi'\psi^2(V\rho_+ +\tilde V\rho_+), \\
	g_\delta&=\varphi_\delta^2\rho^{2\sigma}\chi_0^2\chi^2\psi\psi'(V\hat p_0+\tilde V\hat p_0), \\
	f_\delta&=\varphi_\delta^2\rho^{2\sigma}\chi_0\chi^2\psi^2\Bigl[\bigl(\beta(c_\pa+\tilde c_\pa)x^\beta-2\mu_+\phi_+ x-\tilde V\phi_+^2+\tilde V\phi_-^2 \\
	  &\hspace{3cm} + 2(\nu_+\phi_+ - \nu_-\phi_-)\rho^{-m}\tilde p_{m,r}\bigr)\chi_0' \\
	  &\hspace{4cm} -(\sigma-(m-1)/2-\delta\rho\varphi_\delta)(c_f+\tilde c_f)x\chi_0\Bigr]
  \end{align*}
  Note that in the definition of $f_\delta$, by the choice of $\beta$ and using the fact that $\chi_0$ is bounded by a constant multiple of $\chi_0'$ on its support, the constant being uniform for $R+\kappa<1$, the term $c_\pa x^\beta$ dominates all other terms on the support of $f_\delta\in S^{2\sigma;\infty}\Hb^{s-1}$ for $R$ and $\kappa$ small enough, hence $f_\delta\geq 0$, and its contribution will be controlled by virtue of the sharp G\aa rding inequality. The term arising from $e_\delta$ will be controlled using the a priori regularity assumption of $u_r$ on $\Gamma_-$, and $g_\delta$, which is supported away from the characteristic set, will be controlled using elliptic regularity.
  
  Proceeding with the argument, we first make the simplification $\tilde R_r=0$ by replacing $f$ by $f-\tilde R_r u_r$, and we assume $m\geq 1$ and $\tilde s\geq(5-m)/2$ for now. Then we have, as in the proof of \cite[Theorem~6.10]{HintzQuasilinearDS},
  \begin{align*}
    \Re\la i A_\delta^*&[P_{0,r}+\tilde P_{m,r},A_\delta]u_r,u_r\ra+\left\la\frac{1}{2i}(P_{0,r}-P_{0,r}^*)A_\delta u_r,A_\delta u_r\right\ra \\
	  & = - \left\la\frac{1}{2i}(\tilde P_{m,r}-\tilde P_{m,r}^*)A_\delta u_r,A_\delta u_r\right\ra \\
	  &\hspace{4cm} - \Re\la iA_\delta f,A_\delta u_r\ra+\Re\la iA_\delta\tilde P_{m-1,r}u_r,A_\delta u_r\ra.
  \end{align*}
  Estimating each term on the right hand side as in the proof of \cite[Theorem~6.10]{HintzQuasilinearDS} and using \eqref{EqNTSubprincipal}, we obtain for any $\mu>0$:
  \begin{equation}
  \label{EqNhypPairing}
    \Re\Big\la\bigl(A_\delta^*(i[P_{0,r}+\tilde P_{m,r},A_\delta] + E_1+E_1'+B)A_\delta\bigr)u_r,u_r\Big\ra \geq -C_\mu - \mu\|A_\delta u_r\|_{(m-1)/2}^2.
  \end{equation}
  Here and in what follows, we in particular absorb all terms involving $\|u_r\|_{\sigma-1/2}$ into the constant $C_\mu$. On the left hand side, the $E_1'$-term can be dropped because of $\WFb'(E_1')\cap\WFb'(A)=\emptyset$ for sufficiently localized $a$. Moreover, the principal symbol of $E_1+B$ near $\Gamma$ is $e_1+b=-q^2$ with $q$ an elliptic symbol of order $(m-1)/2$, since, by assumption on $r$, we have $e_1+rc_\pa\rho^{m-1}<0$ near $\Gamma$. Therefore, we can write $E_1+B=-Q^*Q+E_1''+E_2$, where $E_1''\in\Psib^{m-1}$, $E_2\in\Psib^{m-2}$, $\WFb'(E_1'')\cap\Gamma=\emptyset$. Again, the resulting term in the pairing \eqref{EqNhypPairing} involving $E_1''$ can be dropped; also, the term involving $E_2$ can be dropped at the cost of changing the constant $C_\mu$, since $u_r\in\Hb^{\sigma-1/2}$.
  
  Hence, introducing $J^\pm\in\Psib^{\pm(\sigma-(m-1)/2-1)}$, with real principal symbols, satisfying $I-J^+J^-\in\Psib^{-\infty}$, we get
  \begin{equation}
  \label{EqNTPosCommIM}
    \Re\left\la\Op(j^-a_\delta H_{p_{m,r}} a_\delta)u_r,(J^+)^*u_r\right\ra - \|QA_\delta u_r\|_0^2 \geq -C_\mu - \mu\|A_\delta u_r\|_{(m-1)/2}^2.
  \end{equation}
  We now plug the commutator relation \eqref{EqNTCommutator} into this estimate. We obtain several terms, which we bound as follows: First, since $j^-e_\delta\in (\CI+\Hb^{s-1,\alpha})S^{\sigma+(m-1)/2+1}$ uniformly, $\Op(j^-e_\delta)$ is a uniformly bounded family of maps $\Hb^\sigma\to\Hb^{-(m+1)/2}$; thus, choosing $\tilde E\in\Psib^0$ with $\WFb'(\tilde E)\subset U$ and with $\WFb'(I-\tilde E)$ disjoint from $\supp e_\delta$, we conclude
  \[
    |\la\Op(j^-e_\delta)u_r,(J^+)^*u_r\ra|\leq C+|\la\Op(j^-e_\delta)u_r,(J^+)^*\tilde E u_r\ra|\leq C+\|B_2 u_r\|_\sigma^2
  \]
  for some $B_2\in\Psib^0$ with $\WFb'(B_2)\cap\Gamma_+=\emptyset$.

  Next, the term $\la\Op(j^-g_\delta)u_r,(J^+)^*u_r\ra$ is uniformly bounded, as detailed in the proof of \cite[Theorem~6.10]{HintzQuasilinearDS}. Moreover, by the sharp G\aa rding inequality, see the argument in the proof of \cite[Theorem~6.6]{HintzQuasilinearDS},
  \[
    \Re\la \Op(-j^-f_\delta)u_r,(J^+)^*u_r\ra \leq C.
  \]

  Further, we obtain two terms involving $h_{\pm,\delta}$; introducing $B_3\in\Psib^0$ elliptic on $\WFb'(A)$, these can be bounded for $\mu>0$ by
  \begin{align*}
    |\la\Op(&j^-a_{\pm,\delta}h_{\pm,\delta}p_{m,r})u_r,(J^+)^*u_r\ra| \\
	  &\leq C+|\la\Op(j^-a_{\pm,\delta}h_{\pm,\delta})(P_{0,r}+\tilde P_{m,r})u_r,(J^+)^*u_r\ra| \\
	  &\leq C+|\la H_{\pm,\delta}f_r,A_{\pm,\delta}u_r\ra|+|\la\Op(j^-a_{\pm,\delta}h_{\pm,\delta})\tilde P_{m-1,r}u_r,(J^+)^*u_r\ra| \\
	  &\leq C+\mu\|A_{\pm,\delta}u_r\|_0^2+C_\mu\|B_3 f_r\|_{\sigma-m}^2.
  \end{align*}
  Here, for the first estimate, we employ \cite[Theorem~3.12 (3)]{HintzQuasilinearDS} to obtain
  \begin{align*}
    \Op(j^-a_{\pm,\delta}h_{\pm,\delta})&\tilde P_{m,r}-\Op(j^-a_{\pm,\delta}h_{\pm,\delta}\tilde p_{m,r}) \\
	  &=:\Upsilon_\delta \in \Psib^{\sigma+(m-1)/2;0}\Hb^{s-1}+\Psib^{\sigma-(m-1)/2-1}\circ\Psi^{m;0}\Hb^{s-1},
  \end{align*}
  and $\Upsilon_\delta$ is easily seen to be uniformly bounded from $\Hb^{\sigma-1/2}$ to $\Hb^{-m/2}$, whereas $(J^+)^*u_r\in\Hb^{m/2}$, thus $|\la\Upsilon_\delta u_r,(J^+)^*u_r\ra|\leq C$. For the second estimate, we simply use $(P_{0,r}+\tilde P_{m,r})u_r=f_r-\tilde P_{m-1,r}u_r$, and for the third estimate, we apply the Peter--Paul inequality to the first pairing; to bound the second pairing, we use the boundedness of $\tilde P_{m-1,r}\colon\Hb^{\sigma-1/2}\to\Hb^{\sigma-m+1/2}$.
  
  Finally, including the terms $c_\pm^2 a_{\pm,\delta}^2$ into the estimate obtained from \eqref{EqNTPosCommIM} by making use of the above estimates, we obtain
  \begin{align*}
    \|C_+A_{+,\delta}u_r\|_0^2&+\|C_- A_{-,\delta} u_r\|_0^2 + \|QA_\delta u_r\|_0^2 \\
	 &\leq C_\mu+\mu\|A_{+,\delta}u_r\|_0^2 + \mu\|A_{-,\delta}u_r\|_0^2 + \mu\|A_\delta u_r\|_{(m-1)/2}^2 \\
	 &\hspace{1.7cm} + \|B_2 u_r\|_\sigma^2 + \|B_1 f_r\|_{\sigma-m+1}^2 + C_\mu\|\chi f_r\|_{\sigma-m}^2,
  \end{align*}
  where $B_1\in\Psib^0$ is elliptic on $\WFb'(A)$ with $\WFb'(B_1)\subset U$, and $\chi\in\CIc(M)$ is identically $1$ near the projection of $\Gamma\subset\Sb^*M$ to the base $M$. Since $c_+$ and $c_-$ have positive lower bounds near $\Gamma$, we can absorb the terms on the right involving $A_{\pm,\delta}$ into the left hand side by choosing $\mu$ sufficiently small, at the cost of changing the constant $C_\mu$; likewise, $\rho^{-(m-1)/2}q$ has a positive lower bound near $\supp a$, hence the term on the right involving $A_\delta$ can be absorbed into the left hand side for small $\mu$. Dropping the first two terms on the left hand side, we obtain the $\Hb^\sigma$-regularity of $u_r$ at $\Gamma$, hence $\WFb^{\sigma,r}(u)\cap\Gamma=\emptyset$, and a corresponding tame estimate, which follows from a careful analysis of the above argument as in the proof of Proposition~\ref{PropTameRealPrType}.

  Next, we remove the restriction $m\geq 1$: Let $m_0\geq 1$. The idea, as before, is to rewrite $Pu=f$ as $P\Lambda^+(\Lambda^-u)=f+PRu$, where $\Lambda^\pm\in\Psib^{\pm(m_0-m)}$, with real principal symbols, satisfy $\Lambda^+\Lambda^-=I+R$. We now have to be a bit careful though to not change the imaginary part of the subprincipal symbol of $P\Lambda^+$ at $\Gamma$. Concretely, we choose $\Lambda^+$ self-adjoint with principal symbol $\lambda^+=\rho^{m_0-m}$ near $\Gamma$; then
  \[
    P_0\Lambda^+-(P_0\Lambda^+)^* = \Lambda^+(P_0-P_0^*) + [P_0,\Lambda^+].
  \]
  Clearly, $\Lambda^+(P_0-P_0^*)\in x\Psib^{m_0-1}+\Psib^{m_0-2}$, and the principal symbol of the second term is
  \[
    \sigma_{\bl,m_0-1}([P_0,\Lambda^+])= -iH_{p_0}\lambda^+ = -ix(m_0-m)\rho^{m_0-1}c_f
  \]
  near $\Gamma$ by \eqref{EqVrho}, hence, using \eqref{EqP0Subprincipal},
  \[
    P_0\Lambda^+-(P_0\Lambda^+)^* = \Lambda^+E_1 + xE_1'+E_1''+E_2
  \]
  with $E_1',E_1''\in\Psib^{m_0-1},E_2\in\Psib^{m_0-2}$ and $\WFb'(E_1'')\cap\Gamma=\emptyset$; therefore, the first part of the proof with $P$ and $u$ replaced by $P\Lambda^+$ and $\Lambda^- u$, respectively, applies. The proof of the theorem in the case $r<-\sup_\Gamma \rho^{-m+1}e_1/c_\pa$ is complete.

  When the role of $\Gamma_+$ and $\Gamma_-$ is reversed, there is an overall sign change, and we thus get a advantageous (now meaning positive) contribution to the subprincipal part of the conjugated operator $P_r$ for $r>-\inf_\Gamma \rho^{-m+1}e_1/c_\pa$; the rest of the argument is unchanged.
\end{proof}

\begin{rmk}
\label{RmkNonsmoothTrappingMfds}
  As in the radial point estimate, see Remark~\ref{RmkRadialPointsMfds}, the assumption that an open neighborhood of the trapped set embeds into $\Rnhalf$ as an open set, true for Kerr-de Sitter space but false for general operators on a manifold $M$, can be removed easily if we assume $\wt R\in\Hb^{s-1,\alpha}\Psib^{m-2}(M)$: To see this, one needs to make sense of the commutator computation involving the non-smooth part $\tilde P$ of $P$. We begin by establishing a bound on $\tilde P_{m,r}-\tilde P_{m,r}^*\in\cL(\Hb^{(m-1)/2}(M),\Hb^{-(m-1)/2}(M))$, which we obtain by writing (we drop the subscripts for brevity)
\[
  \tilde P - \tilde P^* = \sum_j \phi_j \tilde P-\tilde P^*\phi_j,
\]
where $\{\phi_j\}$ is a partition of unity on $M$ subordinated to a cover by local coordinate charts; then, writing $\phi^{(0)}=\phi_j$ for any fixed $j$, and choosing $\phi^{(1)}\in\CIc(M)$ supported in the same coordinate chart as $\phi^{(0)}$ and identically $1$ near $\supp\phi^{(0)}$, we further write
\begin{align*}
  \phi^{(0)}\tilde P - \tilde P^*\phi^{(0)} &= \bigl(\phi^{(0)}\tilde P\phi^{(1)}-(\phi^{(0)}\tilde P\phi^{(1)})^*\bigr) \\
   &\quad + \bigl(\phi^{(0)}\tilde P(1-\phi^{(1)}) - (1-\phi^{(1)})\tilde P^*\phi^{(0)}\bigr),
\end{align*}
where the first term is a bounded operator between the aforementioned spaces by the local (in $\Rnhalf$) argument, while the second term belongs to the class $\Hb^s\Psib^{-\infty}+\Psib^{-\infty}\Hb^s$ and is thus bounded on the relevant spaces as well. Here, we remark that the regularity requirements on $s$ for $\Hb^s\Psib^{-\infty}$ to map $\Hb^{(m-1)/2}$ into $\Hb^{-(m-1)/2}$ are the same as or weaker than the requirements on $s$ from the local argument, while there are \emph{no requirements} other than, say, $s>n/2$, in order to have $\Psib^{-\infty}\Hb^s$ map $\Hb^{\sigma_1}$ into $\Hb^{\sigma_2}$ for \emph{any} $\sigma_1,\sigma_2\in\R$, $\sigma_1\geq 0$, since $\Hb^s\cdot\Hb^{\sigma_1}\subset\Hb^{\min(s,\sigma_1)}$ (with a tame bound for the product) gets mapped into $\Hb^\infty\subset\Hb^{\sigma_2}$ by $\Psib^{-\infty}$.

  Similarly, one can introduce localizers in all commutator estimates, with the resulting error terms satisfying tame bounds by the a priori assumptions on $u$; for example, the pairing $\la J^-A^*[P,A]u,u\ra$ (with $J^-$ as in the proof) is equal to an error term (generated by the localizers as above) plus the sum of
  \[
    \la\phi^{(0)}J^-A^*\phi^{(1)}[\phi^{(2)}P\phi^{(3)},\phi^{(4)}A\phi^{(5)}]\phi^{(6)}u,\phi^{(7)}u\ra,
  \]
  where $\phi^{(0)}$ runs over a partition of unity on $M$ subordinated to a cover by coordinate charts, and where for $k=1,\ldots,7$, $\phi^{(k)}\in\CIc(M)$ is a cutoff, with $\phi^{(k)}\equiv 1$ on $\supp\phi^{(k-1)}$, and $\supp\phi^{(k)}$ contained in the same coordinate chart as $\phi^{(0)}$; the error term can be estimated by a tame bound involving the a priori regularity of $u$. All operators now have Schwartz kernels supported within a single coordinate chart, and all functions on which the operators act have support in the same coordinate chart, hence the local non-smooth theory, as used in the local positive commutator estimate, shows that the operator acting on $u$ in the above pairing is equal to
  \[
    \Op(\phi^{(0)}j^-a H_{\phi^{(2)}p}\phi^{(4)}a)\phi^{(6)} = \Op(\phi^{(0)}j^-a H_p a)\phi^{(6)}.
  \]
  At this point, one can plug in the terms of the symbolic positive commutator calculation of $a H_p a$, all of which now get multiplied by $\phi^{(0)}$. Since we have an invariant calculus for smooth operators, summing over the partition of unity (of which $\phi^{(0)}$ is a member) recovers the usual positive commutator calculation, with a non-smooth error term of the form $\Op(-\phi^{(0)}j^-f)\phi^{(6)}$ coming from each coordinate chart; but each of these error terms separately has a sign (since $f$ does) and is thus controlled by the sharp G\aa rding inequality for non-smooth operators in a coordinate chart.

  In particular, we see that there are no further regularity requirements for the invariant estimate, the requirements given in the statement of the theorem being sufficient.
\end{rmk}

\subsection{Trapping estimates at normally hyperbolic trapping}
\label{SubsecTrapping}

Complementing the results above on negatively weighted spaces, we
recall results of Dyatlov from
\cite{DyatlovResonanceProjectors,DyatlovNormally} on semiclassical
estimates for smooth operators at normally hyperbolic trapping, which
via the Mellin transform correspond to estimates on non-negatively
weighted spaces. Here we present the results in the semiclassical
setting, then in \S\ref{SubsecForward} we relate this to the
solvability of linear equations with Sobolev coefficients in
Theorem~\ref{ThmNHypSemiGlobal} and Theorem~\ref{ThmExpansion}. The
results of Dyatlov we use followed and built on the breakthrough earlier work of Wunsch and Zworski
\cite{WunschZworskiNormHypResolvent}, as well as Nonnenmacher and Zworski
\cite{Nonnenmacher-Zworski:Correlations}. (In the analytic category
there were earlier results of G\'erard and Sj\"ostrand
\cite{Gerard-Sjostrand:Resonances}.) The
advantage of the framework of \cite{DyatlovResonanceProjectors,DyatlovNormally} for us, especially as espoused in \cite{DyatlovNormally}, is the explicit size of the `spectral gap' (discussed below), which was also shown by Nonnenmacher and Zworski \cite{Nonnenmacher-Zworski:Correlations}, the explicit inclusion of a subprincipal term of the correct sign, and the relative ease with which the parameter dependence can be analyzed.

We first recall the semiclassical setting of \cite{DyatlovNormally} for\footnote{We write `$h$' for the semiclassical parameter, and use the subscript `$\semi$' for spaces of semiclassical operators, symbols and distributions.}
$$
\tilde P_0=\tilde P_0(h),\tilde Q_0=\tilde Q_0(h)\in\Psih^m(X),
$$
both formally self-adjoint, with $\tilde
Q_0$ having non-negative principal symbol,
$\tilde P_0-i\tilde Q_0$ elliptic in the standard sense.
In fact, the results in \cite{DyatlovNormally} are stated in the special case $m=0$, but by 
ellipticity of $\tilde P_0-i\tilde Q_0$ in the standard sense, it is 
straightforward to allow general $m$; see also the remark \cite[bottom of p.~2]{DyatlovNormally}.
The main assumption, see \cite[p.~3]{DyatlovNormally}, then is that $\tilde P_0$ has
normally hyperbolic trapping semiclassically at $\tilde\Gamma\subset
T^*X$ compact,\footnote{Our $\tilde\Gamma$ is the intersection of
  what is called $K$ in \cite{DyatlovNormally} with the semiclassical characteristic set of $P$, and
  similarly our $\tilde\Gamma_\pm$ are the intersection of what is called $\Gamma_\pm$ in \cite{DyatlovNormally} with the characteristic set of $P$.} with all bicharacteristics of $\tilde
P_0$, except those in the stable ($-$) and unstable ($+$)
submanifolds $\tilde\Gamma_\pm$, entering the elliptic set of
$\tilde Q_0$ in the forward (the exception being for only the $-$ sign), resp.\ backward ($+$) direction,
and $\gamma<\numin/2$, where $\numin>0$ is the minimal normal
expansion rate of the flow at $\tilde\Gamma$, discussed above and in \eqref{EqNuMinDef}. If $\tilde Q_0$ is microlocally in
$h\Psih(X)$ near $\tilde\Gamma$, with $h^{-1}\tilde Q_0$ having a
non-negative principal symbol there,
the Theorem 1 in \cite{DyatlovNormally} shows that there is $h_0>0$ such that for $\Im z>-\gamma$,
\begin{equation}\label{eq:tilde-P_0-est-orig}
  \|v\|_{H^s_h}\lesssim h^{-2}\|(\tilde P_0-i\tilde Q_0-hz)v\|_{H^{s-m}_h},\ h<h_0.
\end{equation}
In view of $\tilde\Gamma$ lying in a compact subset of $T^*X$, the
order $s$ is irrelevant in the sense that the estimate for one value
of $s$ implies that for all other via elliptic estimates; thus, one
may just take $s=0$, and even replace $s-m$ by $0$, in which case this
is an $L^2$-estimate, as stated explicitly in
\cite{DyatlovNormally}. We emphasize that while parts of
\cite{DyatlovResonanceProjectors} use the model forms of operators, this is no longer
the case in \cite{DyatlovNormally}, thus apart from the already
discussed modifications for the differential operator and Sobolev
space orders, the results apply verbatim in our setting, without
the necessity to arrange these model forms. Furthermore,
\cite[Lemma~5.1]{DyatlovResonanceProjectors}, which we use below, is also directly
applicable in our setting.

Suppose now that one has a family of operators $\tilde P_0(\omega)$ depending on another parameter, $\omega$, in
a compact space $S$, with $\tilde P_0,\tilde Q_0$ depending continuously on
$\omega$, with values in $\Psih^m(X)$, satisfying all of the
assumptions listed above. Suppose moreover that this family is \emph{uniformly normally hyperbolic}, i.e.\ satisfies the normally
hyperbolic assumptions with $\tilde\Gamma,\tilde\Gamma_\pm$
continuously depending on $\omega$ in the $\CI$ topology, and uniform
bounds for the normal expansion rates in the sense that both
$\nu$ and the constant $C'$ in
\begin{equation}
\label{EqNuMinDef}
  \sup_{\rho\in\Gamma}\|d e^{\mp tH_p}(\rho)|_{\cV_\pm}\|\leq C'e^{-\nu t},\ t\geq 0,
\end{equation}
with $\cV_\pm$ the unstable and stable normal tangent bundles at
$\Gamma$, can be chosen uniformly (cf.\ \cite[Equation~(5.1)]{DyatlovResonanceProjectors}); $\numin$ is then the $\sup$ of
these possible choices of $\nu$. (Note that since the trapped set
dynamics involves arbitrarily large times, it is {\em not}
automatically stable, unlike the dynamics away from the trapped set.) In this case the implied constant $C$
in \eqref{eq:tilde-P_0-est-orig}, as well as $h_0$, is uniform in $\omega$. Note that
$r$-normal hyperbolicity for every $r$ implies the local uniformity of the
normal dynamics by structural stability; see
\cite[\S1]{WunschZworskiNormHypResolvent} and
\cite[\S5.2]{DyatlovResonanceProjectors}.

To see this uniformity in $C$, we first point out that in
\cite[Lemma~5.1]{DyatlovResonanceProjectors} the construction of
$\phi_\pm$ can be done continuously with values in $\CI$ in this case.
Then in the proof of \eqref{eq:tilde-P_0-est-orig} given in \cite{DyatlovNormally},
we only need to observe that the direct estimates provided are
certainly uniform in this case for families $\tilde P_0,\tilde Q_0$, and furthermore for the main argument,
using semiclassical defect measures, one can pass to an $L^2$-bounded subsequence
$u_j$ such that $(\tilde P_0(\omega_j)-i\tilde Q_0(\omega_j)-\lambda_j)u_j=O(h^2)$,
with $\omega_j\to\omega$ for some $\omega\in S$ in addition to
$h^{-1}\lambda_j$ converging to some $\tilde\lambda$.
Concretely, all results in \cite[\S2]{DyatlovNormally} are based on elliptic or (positive) commutator
identities or estimates which are uniform in this setting. In
particular, \cite[Lemma~2.3]{DyatlovNormally} is valid with
$P_j=P(\omega_j)\to P$, $W_j=W(\omega_j)\to W$ with convergence in
$\Psih(X)$.
(This uses that one can take
$A_j(h_j)$ in Definition~2.1, with $A_j\to A$, since the difference between
$A_j(h_j)$ and $A(h_j)$ is bounded by a constant times the squared
$L^2$-norm of $u_j$ times the operator norm bound of $A_j(h_j)-A(h_j)$, with
the latter going to $0$.)
Then with $\Theta_{+,j}$ in place of $\Theta_+$, one still gets Lemma~3.1, which means that Lemma~3.2 still holds with $\phi_+$ (the limiting
$\phi_{+,j}$) using Lemma~2.3. Then the displayed equation above \cite[Equation~(3.9)]{DyatlovNormally} still
holds with the limiting $\tilde P_0=\tilde P_0(\omega)$, again by Lemma~2.3, and then one can finish the
argument as in \cite{DyatlovNormally}.
With this
modification, one obtains the desired uniformity. This in particular
allows one to apply \eqref{eq:tilde-P_0-est-orig} even if $\tilde
P_0$ and
$\tilde Q_0$ depend on $z$ (in a manner consistent with the other
requirements), which can also be dealt with more directly using
the model form in \cite[Lemma~4.3]{DyatlovResonanceProjectors}. It
also allows for uniform estimates for families depending on a small
parameter in $\Cx$, denoted by $v_0$ below, needed in \S\ref{SecQuasilinear}.

Allowing $\tilde P_0$ and $\tilde Q_0$ depending on $z$ means, in
particular, that we can replace the requirement on $h^{-1}\tilde Q_0$
by the principal symbol of $h^{-1}\tilde Q_0$ being $>-\beta$,
$\beta<\numin/2$, and drop $z$, so one has
\begin{equation}\label{eq:tilde-P_0-est}
  \|v\|_{H^s_h}\lesssim h^{-2}\|(\tilde P_0-i\tilde Q_0)v\|_{H^{s-m}_h},\ h<h_0.
\end{equation}

At this point it is convenient to rewrite this estimate, removing $\tilde Q_0$ from the right hand side at the cost (or benefit!) of making it microlocal.\footnote{An alternative would be using the gluing result of Datchev and Vasy \cite{DatchevVasyGluing}, which is closely related in approach.} {\em From here on it is convenient to change the conventions and not require that $\tilde P_0$ is formally self-adjoint (though it is at the principal symbol level, namely it has a real principal symbol); translating back into the previous notation, one would replace $\tilde P_0$ by its (formally) self-adjoint part, and absorb its skew-adjoint part into $\tilde Q_0$.} Namely, we have

\begin{thm}\label{ThmNHypSemiMic}
Suppose $\tilde P_0$ satisfies the above assumptions, in particular
the semiclassical principal symbol of $\frac{1}{2ih}(\tilde P_0-\tilde P_0^*)$ being $<\beta<\numin/2$ at $\tilde\Gamma$.\footnote{The apparent sign change here as compared to before comes from the fact that for formally self-adjoint $\tilde P_0,\tilde Q_0$, one has $\frac{1}{2ih}\bigl((\tilde P_0-i\tilde Q_0)-(\tilde P_0-i\tilde Q_0)^*\bigr)=-h^{-1}\tilde Q_0$; notice the minus sign on the right hand side.}
With $\tilde B_j$ analogous to Theorem~\ref{ThmTameHypTr}, with wave front set sufficiently close to $\tilde\Gamma$, we have, for sufficiently small $h>0$ and for all $N$ and $s_0$,
\begin{equation}\begin{aligned}\label{EqNHypSemiMic}
\|\tilde B_0 u\|_{H^s_h}
&\lesssim h^{-2}\| \tilde B_1 \tilde P_0 u\|_{H^{s-m+1}_h}+h^{-1}\|\tilde B_2 u\|_{H^s_h}+h^N\|u\|_{H^{s_0}_h}.
\end{aligned}\end{equation}
\end{thm}

Note that the differential orders are actually irrelevant here due to
wave front set conditions.

\begin{proof}
Take $\tilde Q_0\in\Psih^0(X)$ with non-negative principal symbol such that $\WFh'(\tilde Q_0)$ is disjoint from
$\WFh'(\tilde B_0)$, and so that all
backward bicharacteristics from points not in $\tilde\Gamma_+$, as well as
forward bicharacteristics from points not in $\tilde\Gamma_-$, reach the
elliptic set of $\tilde Q_0$, and with $\tilde B_1$ elliptic on the complement of
the elliptic set of $\tilde Q_0$.
Let $\tilde B_3\in\Psih^0(X)$ to be such that $\WFh'(I-\tilde B_3)$ is disjoint
from $\WFh'(\tilde B_0)$ but $\WFh'(\tilde Q_0)\cap\WFh'(\tilde B_3)=\emptyset$. Let
$\tilde A_+\in\Psih^0(X)$ have wave front set near $\tilde\Gamma_+$, with
$$
\WFh'(I-\tilde A_+) \cap\WFh'(\tilde B_3)\cap\tilde\Gamma_+=\emptyset
$$
and with
$$
\WFh'(\tilde A_+)\cap\WFh'(I-\tilde B_3)\cap\tilde\Gamma_-=\emptyset,
$$
and with no
backward bicharacteristic from $\WFh'(\tilde B_0)$ reaching
$$
\WFh'(\tilde A_+)\cap\WFh'(I-\tilde B_3)\cap\tilde\Gamma_+.
$$
Take $\tilde Q_1$
elliptic on $\tilde\Gamma$, with $\WFh'(\tilde Q_1)\cap\WFh'(I-\tilde B_3)=\emptyset$,
again with non-negative principal symbol, with no backward bicharacteristic from $\WFh'(\tilde Q_1)$ reaching
\[
  \WFh'(\tilde A_+)\cap\WFh'(I-\tilde B_3).
\]
Thus,
all backward and forward bicharacteristics of $\tilde P_0$ reach
the elliptic set of $\tilde Q_1$ or $\tilde Q_0$. See Figure~\ref{FigNhypMic} for the setup.

\begin{figure}[!ht]
  \centering
  \includegraphics{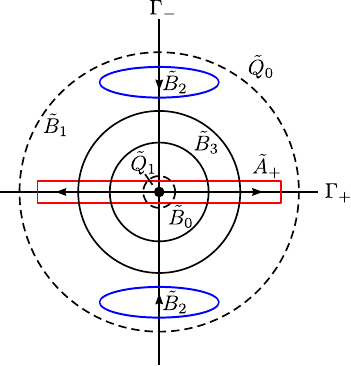}
  \caption{Setup for the proof of the microlocalized normally hyperbolic trapping estimate \eqref{EqNHypSemiMic}: Indicated are the backward and forward trapped sets $\Gamma_+$ and $\Gamma_-$, respectively, which intersect at $\Gamma$ (large dot). We use complex absorbing potentials $\tilde Q_0$ (with $\WFh'(\tilde Q_0)$ outside the large dashed circle) and $\tilde Q_1$ (with $\WFh'(\tilde Q_1)$ inside the small dashed circle). We obtain an estimate for $\tilde B_0 u$ by combining \eqref{eq:tilde-P_0-est} with microlocal propagation from the elliptic set of $\tilde B_2$.}
\label{FigNhypMic}
\end{figure}

Then
\begin{equation*}
(\tilde P_0-i\tilde Q_0) \tilde B_3 u=\tilde B_3 \tilde P_0 u+\tilde A_+[\tilde P_0,\tilde B_3]u+(I-\tilde A_+)[\tilde P_0,\tilde B_3]u-i\tilde Q_0\tilde B_3u,
\end{equation*}
so
\begin{equation}\begin{aligned}\label{EqCxAbsorbError}
\tilde B_0 u&=\tilde B_0\tilde B_3 u+\tilde B_0(I-\tilde B_3)u\\
&=\tilde B_0(\tilde P_0-i\tilde Q_0)^{-1}\tilde B_3 \tilde P_0 u+\tilde B_0(\tilde P_0-i\tilde Q_0)^{-1}\tilde A_+[\tilde P_0,\tilde B_3]u\\
&\qquad+\tilde B_0(\tilde P_0-i\tilde Q_0)^{-1}(I-\tilde A_+)[\tilde P_0,\tilde B_3]u\\
&\qquad-i\tilde B_0(\tilde P_0-i\tilde Q_0)^{-1}\tilde Q_0\tilde B_3u+\tilde B_0(I-\tilde B_3)u,
\end{aligned}\end{equation}
and by \eqref{eq:tilde-P_0-est}, for $h<h_0$,
$$
\|(\tilde P_0-i\tilde Q_0)^{-1}\tilde B_3 \tilde P_0u\|_{H^s_h}\lesssim h^{-2}\|\tilde B_3 \tilde P_0 u\|_{H^{s-m}_h}.
$$
Now, $\tilde Q_0\tilde B_3,\tilde B_0(I-\tilde B_3)\in h^\infty\Psih^{-\infty}(X)$, so the corresponding terms in
\eqref{EqCxAbsorbError} can be
absorbed into $h^N\|u\|_{H^{s_0}_h}$. On the other hand, since
$\WFh'((I-\tilde A_+)[\tilde P_0,\tilde B_3])$ is disjoint from $\tilde\Gamma_+$, the
backward bicharacteristics from it reach the elliptic set of $\tilde B_2$,
and so we have the microlocal real principal type estimate for $u$:
$$
\|(I-\tilde A_+)[\tilde P_0,\tilde B_3]u\|_{H^{s-m}_h}\lesssim h\|\tilde B_2 u\|_{H^{s-1}_h}+\|\tilde B_1\tilde P_0 u\|_{H^{s-m}_h}
$$
as $(I-\tilde A_+)[\tilde P_0,\tilde B_3]\in h\Psih^{m-1}(X)$, so by
\eqref{eq:tilde-P_0-est},
$$
\|(\tilde P_0-i\tilde Q_0)^{-1}(I-\tilde A_+)[\tilde P_0,\tilde B_3]u\|_{H^s_h}\lesssim h^{-1}\|\tilde B_2 u\|_{H^{s-1}_h}+h^{-2}\|\tilde B_1\tilde P_0 u\|_{H^{s-m}_h}.
$$
Thus, 
\eqref{EqNHypSemiMic}
follows if we can estimate
$\|\tilde B_0(\tilde P_0-i\tilde Q_0)^{-1}\tilde A_+[\tilde P_0,\tilde B_3]u\|_{H^s_h}$. Now,
$\WFh'(\tilde A_+[\tilde P_0,\tilde B_3])\cap\tilde\Gamma_-=\emptyset$ by arrangement. In order
to microlocalize, we now introduce a nontrapping model,
$\tilde P_0-i(\tilde Q_0+\tilde Q_1)$. We claim that
$$
v=\bigl(\tilde P_0-i(\tilde Q_0+\tilde Q_1)\bigr)^{-1}\tilde A_+[\tilde P_0,\tilde B_3]u-(\tilde P_0-i\tilde Q_0)^{-1}\tilde A_+[\tilde P_0,\tilde B_3]u
$$
satisfies
\begin{equation}\label{EqRoundOff}
\|v\|_{H^{s'}_h}\lesssim
h^N\|u\|_{H^{s_0}_h}
\end{equation}
for all $s',N$. Notice that for any $s''$ one certainly has
$$
\|v\|_{H^{s''}_h}\lesssim h^{-1}\|u\|_{H^{s''-1}_h}
$$
by \eqref{eq:tilde-P_0-est} plus its non-trapping analogue.
To see \eqref{EqRoundOff}, notice that
$$
(\tilde P_0-i\tilde Q_0)v=i\tilde Q_1\bigl(\tilde P_0-i(\tilde Q_0+\tilde Q_1)\bigr)^{-1} \tilde A_+[\tilde P_0,\tilde B_3]u,
$$
so by \eqref{eq:tilde-P_0-est}, with $s_0$ replaced by any $s_0'$ (since $s_0$ was arbitrary), and for any $N$,
$$
\|v\|_{H^{s'}_h}\lesssim h^{-2}\|\tilde Q_1\bigl(\tilde P_0-i(\tilde Q_0+\tilde Q_1)\bigr)^{-1}\tilde A_+[\tilde P_0,\tilde B_3]u\|_{H^{s'-m}_h}\lesssim h^{N}\|u\|_{H^{s_0}_h},
$$
since $\tilde P_0-i(\tilde Q_0+\tilde Q_1)$ is non-trapping, hence $\bigl(\tilde P_0-i(\tilde Q_0+\tilde Q_1)\bigr)^{-1}$ propagates semiclassical wave front sets along forward bicharacteristics, and no backward bicharacteristic from $\WFh'(\tilde Q_1)$ can reach $\WFh'(\tilde A_+[\tilde P_0,\tilde B_3])\subset\WFh'(\tilde A_+)\cap\WFh'(I-\tilde B_3)$, proving the claim. Then, since backward bicharacteristics
from $\WFh'(\tilde B_0)$ do not encounter $\WFh'(\tilde A_+[\tilde P_0,\tilde B_3])\cap\tilde\Gamma_+$ before
reaching the elliptic set of $\tilde Q_0$ or $\tilde Q_1$,
we conclude that
\begin{equation*}\begin{aligned}
&\|\tilde B_0(\tilde P_0-i\tilde Q_0)^{-1} \tilde A_+[\tilde P_0,\tilde B_3]u\|_{H^s_h}\\
&\qquad\leq
\|\tilde B_0
(\tilde P_0-i\tilde Q_0-i\tilde Q_1)^{-1}\tilde A_+[\tilde P_0,\tilde B_3]u\|_{H^s_h}+\|\tilde B_0v\|_{H^s_h}\\
&\qquad\lesssim h\|\tilde B_2 u\|_{H^s_h}+\|\tilde B_1\tilde P_0 u\|_{H^{s-m+1}_h}+h^N\|u\|_{H^{s_0}_h}.
\end{aligned}\end{equation*}
This proves \eqref{EqNHypSemiMic}, and thus the theorem.
\end{proof}

\section{Quasilinear wave and Klein-Gordon equations}
\label{SecQuasilinear}

\subsection{Forward solution operators}
\label{SubsecForward}

We now generalize the setting considered in \cite[\S{7.2}]{HintzQuasilinearDS} for the study of quasilinear equations on static asymptotically de Sitter spaces to allow for normally hyperbolic trapping, as discussed in the previous section.

For simplicity, we first describe the scalar setup: Working on a compact manifold $M$ with boundary $X$, we assume that the operator $P$ is of the form $P=P_0+\tilde P$, continuously depending on a small parameter
\[
  v=(v_0,\tilde v)\in\cX^{\tilde s,\alpha}:=\C\oplus\Hb^{\tilde s,\alpha},\quad \alpha>0,
\]
where
\begin{align}
\label{EqP0Form} P_0&=P_0(v_0)=\Box_{g(v_0)}+L(v_0)\in\Diffb^2(M), \\
    &\hspace{1cm}L(v_0)\in\Diffb^1(M),\quad L(0)-L(0)^*\in\Diffb^0(M), \nonumber\\
  \tilde P&=\tilde P(v)\in\Hb^{\tilde s,\alpha}\Diffb^2(M)+\Hb^{\tilde s-1,\alpha}\Diffb^1(M) \nonumber
\end{align}
for a smooth b-metric $g$ on $M$ that continuously depends on one complex parameter; here, we identify $v=(v_0,\tilde v)$ with the distribution $v_0\chi+\tilde v$ on $M$, where $\chi\in\CI(M)$ is a cutoff, identically $1$ near $X$. An example to keep in mind for the remainder of the section is the wave operator on an (asymptotically) Kerr-de Sitter space or a metric $\Hb^{\tilde s+1,\alpha}$-perturbation thereof. We assume:
\begin{enumerate}
  \item The characteristic set $\Sigma\subset\Sb^*_X M$ of $P_0$ has the form $\Sigma=\Sigma_+\cup\Sigma_-$ with $\Sigma_\pm$ a union of connected components of $\Sigma$,
  \item $P_0$ has normally hyperbolic trapping at $\Gamma^\pm\subset\Sigma_\pm$ for small $v_0$, as detailed in assumptions \itref{EnumNTGamma}-\itref{EnumNTQuant3} in \S\ref{SubsecTamePropagation},
  \item $P_0$ has radial sets $L_\pm\subset\Sb^*_X M$, which, in appropriate directions transverse to $L_\pm$, are sources ($-$)/sinks ($+$) for the null-bicharacteristic flow within $\Sb^*_X M$, with a one-dimensional stable/unstable manifold intersecting $\Sb^*_X M$ trans\-versally; for details, see \cite[\S{6.4}]{HintzQuasilinearDS}. In particular, there are $\beta_0,\tilde\beta\in\CI(L_\pm)$, $\beta_0,\tilde\beta>0$, such that for a homogeneous degree $-1$ boundary defining function $\rho$ of fiber infinity in $\rcTb^*M$ and with $V=\rho H_{p_0}$,
    \begin{equation}
	\label{EqRadialParams}
	  \rho^{-1} V\rho|_{L_\pm}=\mp\beta_0, \qquad -x^{-1}Vx|_{L_\pm}=\mp\tilde\beta\beta_0.
	\end{equation}
\setcounter{CounterEnumi}{\value{enumi}}
\end{enumerate}

We will set up initial value problems by introducing artificial
boundaries as in \cite{HintzQuasilinearDS,HintzVasySemilinear}: We
denote by $\ft_1$ and $\ft_2$ two smooth functions on $M$ and
put\footnote{In the Kerr-de Sitter setting described in the
  introduction, we can take $\ft_1=1-x$, which equals
  $\ft_1=1-e^{-t_*}$ in $M^\circ$, and
  $\ft_2=\mp(r-(r_\pm\pm\delta))$, $\delta>0$ small, near $r=r_\pm$,
  or indeed $\ft_2=-(r-(r_++\delta))(r-(r_--\delta))$ globally, which
  is positive precisely in $(r_--\delta,r_++\delta)$.}
\[
  \Omega=\ft_1^{-1}([0,\infty))\cap\ft_2^{-1}([0,\infty)),
\]
where we assume that:
\begin{enumerate}
\setcounter{enumi}{\value{CounterEnumi}}
  \item $\Omega$ is compact,
  \item putting $H_j:=\ft_j^{-1}(0)$, the $H_j$ intersect the boundary $\pa M$ transversally, and $H_1$ and $H_2$ intersect only in the interior of $M$, and they do so transversally,
  \item the differentials of $\ft_1$ and $\ft_2$ have opposite timelike characters near their respective zero sets within $\Omega$; more specifically, $d\ft_1$ is future timelike, $d\ft_2$ past timelike,
  \item there is a boundary defining function $x$ of $M$ such that $dx/x$ is timelike on $\Omega\cap\pa M$ with timelike character opposite to the one of $d\ft_1$, i.e.\ $dx/x$ is past oriented,
  \item the dynamical structure of the null-geodesic flow (within $\Sb^*_\Omega M$) of the metric $g$ is as follows:
    All bicharacteristics in $\Sigma_\Omega:=\Sigma\cap\Sb^*_\Omega M$
    from any point in
    $\Sigma_\Omega\cap(\Sigma_+\setminus(L_+\cup\Gamma^+))$ flow
    (within $\Sigma_\Omega$) to $\Sb^*_{H_1}M\cup L_+\cup\Gamma^+$ in
    the forward direction (i.e.\ either enter $\Sb^*_{H_1}M$ in finite
    time or tend to the radial set $L_+$ or the trapped set
    $\Gamma^+$) and to $\Sb^*_{H_2}M\cup L_+\cup\Gamma^+$ in the
    backward direction, and from any point in
    $\Sigma_\Omega\cap(\Sigma_-\setminus(L_-\cup\Gamma^-))$ to
    $\Sb^*_{H_2}M\cup L_-\cup\Gamma^-$ in the forward direction and to
    $\Sb^*_{H_1}M\cup L_-\cup\Gamma^-$ in the backward direction, with
    tending to $\Gamma^\pm$ allowed in only one of the two
    directions. In conventional scattering theory `non-trapping' is usually
    understood to mean that all bicharacteristics in the
    characteristic set escape to infinity; in microlocal terms
    \cite{Me94} this
    means exactly that they tend to radial sets in both the forward
    and backward directions. Thus, in the sense of scattering theory
    our assumption is that the normally hyperbolic trapping describes all the trapping present here.
\setcounter{CounterEnumi}{\value{enumi}}
\end{enumerate}

Recall the space $\Hb^{s,r}(\Omega)^{\bullet,-}$ of distributions which are supported ($\bullet$) at the `artificial' boundary hypersurface $H_1$ and extendible ($-$) at $H_2$, and the other way around for $\Hb^{s,r}(\Omega)^{-,\bullet}$, see \cite[Appendix~B]{HormanderAnalysisPDE}: Thus, elements of $\Hb^{s,r}(\Omega)^{\bullet,-}$ are restrictions to $\Omega^\circ$ of functions in $\Hb^{s,r}(M)$ which are supported in the future of $H_1$, i.e.\ in $\ft_1^{-1}([0,\infty))$. Then we have the following global energy estimates, with the proof exactly as in the reference:
\begin{lemma}
\label{LemmaWaveGlobalEnergy}
  (Cf.\ \cite[Lemma~7.3]{HintzQuasilinearDS}.) Suppose $\tilde s>n/2+2$. There exists $r_0<0$ such that for $r\leq r_0$, $-\tilde r\leq r_0$, there is $C>0$ such that for $u\in\Hb^{2,r}(\Omega)^{\bullet,-}$, $v\in\Hb^{2,\tilde r}(\Omega)^{-,\bullet}$, one has
  \begin{align*}
    \|u\|_{\Hb^{1,r}(\Omega)^{\bullet,-}} & \leq C \|Pu\|_{\Hb^{0,r}(\Omega)^{\bullet,-}}, \\
	\|v\|_{\Hb^{1,\tilde r}(\Omega)^{-,\bullet}} & \leq C \|P^*v\|_{\Hb^{0,\tilde r}(\Omega)^{-,\bullet}}.
  \end{align*}
\end{lemma}
We recall from \cite{HintzQuasilinearDS} that this result does not rely on the dynamical structure of $P$ at the boundary, but only on the timelike nature of $dx/x$ and of $d\ft_1$ and $d\ft_2$ near $H_1$ and $H_2$, respectively, see also \cite[Remark~7.4]{HintzQuasilinearDS}.

Let us stress that we assume the parameter $v$ to be \emph{small} so that in particular the skew-adjoint part of $P_0(v_0)$ is small and does not affect the radial point and normally hyperbolic trapping estimates which are used in what follows; the general case without symmetry assumptions on $P_0(0)$ will be discussed in \S\ref{SubsecGeneralProofs}. Using a duality argument and the tame estimates for elliptic regularity and the propagation of singularities (real principal type, radial points, normally hyperbolic trapping) given in Propositions~\ref{PropTameElliptic}, \ref{PropTameRealPrType} and \ref{PropTameRadial} and Theorem~\ref{ThmTameHypTr}, we thus obtain solvability and higher regularity:

\begin{lemma}
\label{LemmaWaveGlobalSolve}
  (Cf.\ \cite[Lemma~7.5]{HintzQuasilinearDS}.) Let $0\leq s\leq\tilde s$ and assume $\tilde s>n/2+6$, $s_0>n/2+1/2$. There exists $r_0<0$ such that for $r\leq r_0$, there is $C>0$ with the following property: If $f\in\Hb^{s-1,r}(\Omega)^{\bullet,-}$, then there exists a unique $u\in\Hb^{s,r}(\Omega)^{\bullet,-}$ such that $Pu=f$, and $u$ moreover satisfies
  \begin{equation}
  \label{EqGlobalTame}
    \|u\|_{\Hb^{s,r}(\Omega)^{\bullet,-}} \lesssim \|f\|_{\Hb^{s-1,r}(\Omega)^{\bullet,-}} + \|f\|_{\Hb^{s_0,r}(\Omega)^{\bullet,-}}\|v\|_{\cX^{\tilde s,\alpha}}.
  \end{equation}
  Here, the implicit constant depends only on $s$ and $\|v\|_{\cX^{n/2+6+\eps,\alpha}}$ for $\eps>0$.
\end{lemma}
\begin{proof}
  The proof proceeds as the proof given in the reference, using Theorem~\ref{ThmTameHypTr} to obtain microlocal regularity at the trapped set. The tame estimate \eqref{EqGlobalTame} in particular is obtained by iterative use of the aforementioned microlocal regularity estimates; the given bound for $s_0$ comes from an inspection of the norms in these estimates which correspond to the terms called $u^\ell_*$ in \eqref{EqDefTame}.
\end{proof}

We deduce analogues of \cite[Corollaries~7.6-7.7]{HintzQuasilinearDS}:
\begin{cor}
\label{CorWaveGlobalCont}
  Let $0\leq s\leq\tilde s$ and assume $\tilde s>n/2+6$, $s_0>n/2+1/2$. There exists $r_0<0$ such that for $r\leq r_0$, there is $C>0$ with the following property: If $u\in\Hb^{s,r}(\Omega)^{\bullet,-}$ is such that $Pu\in\Hb^{s-1,r}(\Omega)^{\bullet,-}$, then the estimate \eqref{EqGlobalTame} holds.
\end{cor}

\begin{cor}
\label{CorWavePropSing}
  Let $s_0>n/2+1/2$, $s_0\leq s'\leq s\leq\tilde s$, and assume $\tilde s>n/2+6$; moreover, let $r<0$. Then there is $C>0$ such that the following holds: Any $u\in\Hb^{s',r}(\Omega)^{\bullet,-}$ with $Pu\in\Hb^{s-1,r}(\Omega)^{\bullet,-}$ in fact satisfies $u\in\Hb^{s,r}(\Omega)^{\bullet,-}$, and obeys the estimate
  \begin{align*}
    \|u\|_{\Hb^{s,r}(\Omega)^{\bullet,-}}&\lesssim \|Pu\|_{\Hb^{s-1,r}(\Omega)^{\bullet,-}}+\|u\|_{\Hb^{s',r}(\Omega)^{\bullet,-}} \\
	  &\hspace{1cm}+(\|Pu\|_{\Hb^{s_0,r}(\Omega)^{\bullet,-}}+\|u\|_{\Hb^{s_0+1,r}(\Omega)^{\bullet,-}})\|v\|_{\cX^{\tilde s,\alpha}}.
  \end{align*}
\end{cor}
\begin{proof}
  The proof of the two corollaries is as in the cited reference. For the radial point estimate involved in the proof of Corollary~\ref{CorWavePropSing}, we need the additional assumption $s'-1+\sup_{L_\pm}(r\tilde\beta)>0$, which however is automatically satisfied since $s'\geq 1$ and the $\sup$ is negative for $r<0$.
\end{proof}

We now note that the Mellin transformed normal operator $\hat N(P)(\sigma)$ satisfies global large parameter estimates corresponding to the semiclassical microlocal estimates of Theorem~\ref{ThmNHypSemiMic}. In order to state this precisely we recall the connection between the b-structure, the normal operator (and the large parameter algebra) and the Mellin transform of the latter.

The weighted b-Sobolev spaces $\Hb^{s,\gamma}([0,\infty)\times X)$ are isometric to the large parameter Sobolev spaces on $X$ on the line $\Im\sigma=-\gamma$ in $\Cx$ via the Mellin transform $\cM$; see \cite[Equation~(3.8)]{VasyMicroKerrdS}. Further, the latter can be described in terms of semiclassical Sobolev spaces, namely the restriction $r_{-\gamma}\circ\cM$ to $\Im\sigma=-\gamma$ of the Mellin transform identifies $\Hb^{s,\gamma}([0,\infty)\times X)$ with
\[
\langle|\sigma|\rangle^{-s} L^2(\RR; H^s_{\langle|\sigma|\rangle^{-1}}(X));
\]
see \cite[Equation~(3.9)]{VasyMicroKerrdS}. Here, the norm on the space $H_h^s$ on $\R^n$ is defined by
\[
  \|u\|_{H_h^s}^2 = \int \la h\xi\ra^{2s}|\hat u(\xi)|^2\,d\xi,\quad h>0,
\]
and in general using a partition of unity.

Now, in order to relate b-microlocal analysis with semiclassical analysis, we first identify $\varpi+\sigma\frac{dx}{x}\in\Tb^*([0,\infty)\times X)$, $\varpi\in T^*X$, with $(\sigma,\varpi)\in \RR\times T^*X$. Under the semiclassical rescaling, say by $|\sigma|^{-1}$, one identifies the latter with $h=|\sigma|^{-1}$, $\tilde\varpi=|\sigma|^{-1}\varpi$. In particular, if a conic set is disjoint from $T^*X$ in $\Tb^*([0,\infty)\times X)$, then its image under the semiclassical identification lies in a compact subset of $T^*X$. Thus, for $B\in\Psib^0([0,\infty)\times X)$ dilation invariant, the large parameter principal symbol and wave front set of the Mellin conjugate $\cM B\cM^{-1}=\hat B$ of $B$ are exactly those of $B$ under the above identification of $\varpi+\sigma\frac{dx}{x}\in\Tb^*([0,\infty)\times X)$, $\varpi\in T^*X$, with $(\sigma,\varpi)\in \RR\times T^*X$, and then the analogous statement also holds for $\hat B$ considered as an element $\tilde B$ of $\Psih(X)$ under the semiclassical identification.  In particular, one has, for $B\in\Psib^0([0,\infty)\times X)$ dilation invariant, with $\WFb'(B)\cap T^*X=\emptyset$, that $\tilde B\in\Psi_{|\sigma|^{-1}}^{-\infty}(X)$, with semiclassical wave front set in a compact subset of $T^*X$. Correspondingly, for any $s_0$,
$$
\|Bu\|_{\Hb^{s,\gamma}}^2\lesssim \int_{\sigma\in\RR,\ |\sigma|>h_0^{-1}} |\sigma|^{2s}\|\hat B \cM u(.-i\gamma)\|^2_{L^2}\,d\sigma+\|u\|_{\Hb^{s_0,\gamma}}^2.
$$

Now if $P_0=P_0(v_0)\in\Psib^m(M)$, then $N(P_0)$ is dilation invariant on
$[0,\infty)\times X$, and its conjugate by the Mellin transform is
$\hat P_0=\hat N(P_0)$, whose rescaling $\tilde P_0=|\sigma|^{-m}\hat
P_0$ is an element of $\Psih^m(X)$. Further, with $P_0$ b-normally
hyperbolic in the sense discussed above (with the convention changed
regarding formal self-adjointness, as stated before Theorem~\ref{ThmNHypSemiMic}), $\tilde P_0$ is normally
hyperbolic in the sense of \cite{DyatlovNormally}. Fix a smooth b-density on $M$ near
$X$, identified with $[0,\ep_0)\times X$ as above; we require this to
be of the product form $\frac{|dx|}{x}\,\nu$, $\nu$ a smooth density
on $X$; we compute adjoints with respect to this density. Then for any
$B\in\Psib^m(M)$, $\widehat{B^*}(\sigma)=(\hat B(\overline{\sigma}))^*$, see \cite[\S3.3]{VasyMicroKerrdS} for
differential operators, and by a straightforward calculation using the
Mellin transform in general. In particular, if $B=B^*$, then $\hat B(\sigma)=(\hat B(\sigma))^*$ for $\sigma\in\RR$.
Relaxing \eqref{EqP0Form} momentarily, we then assume that
\begin{equation}
\label{EqSkewAdjointPart}
  \frac{1}{2i}(P_0-P_0^*)\in \Psib^{m-1}(M),\qquad \sigma_{\bl,m-1}\Big(\frac{1}{2i}(P_0-P_0^*)\Big)\Big|_\Gamma<|\sigma|^{m-1}\numin/2,
\end{equation}
with $\numin$ the minimal normal expansion rate for the Hamilton flow of the principal symbol of $P_0$ at $\Gamma\subset\Tb^*_XM$, as above; note that $\sigma$ is elliptic on $\Gamma$. This gives that for $\sigma\in\RR$, $\hat P_0(\sigma)-\hat P_0(\sigma)^*$ is order $m-1$ in the large parameter pseudodifferential algebra, so, defining $z=\sigma/|\sigma|$, the semiclassical version gives
$$
\tilde P_0-\tilde P_0^*\in h\Psih^{m-1}(X),\ z\in\RR,
$$
with
$$
\sigma_{\semi,m-1}\Big(\frac{1}{2ih}(\tilde P_0-\tilde P_0^*)\Big)\Big|_{\tilde\Gamma}<\numin/2,\ z\in\RR,
$$
where $\tilde\Gamma$ is the image of $\Gamma$ under the semiclassical identification.
In particular, there is $\gamma_\Gamma>0$ and $\beta_\Gamma<\numin/2$ such that
if $|\Im z|< h\gamma_\Gamma$ then
\begin{equation}\label{EqMaxWeight}
\sigma_{\semi,m-1}\Big(\frac{1}{2ih}(\tilde P_0-\tilde P_0^*)\Big)\Big|_{\tilde\Gamma}<\beta_\Gamma.
\end{equation}
With this background, under our assumptions on the dynamics, propagating estimates from the radial points towards $H_2$, in particular through $\tilde\Gamma$, and using the uniformity in parameters described above Theorem~\ref{ThmNHypSemiMic}, we have:

\begin{thm}
\label{ThmNHypSemiGlobal}
Let $C_0>0$. Suppose $P_0=P_0(v_0)$ satisfies \eqref{EqSkewAdjointPart} at $\Gamma$, $\tilde P_0$ is the semiclassical rescaling of $\hat P_0=\hat N(P_0)$, $s>1/2+\sup(\tilde\beta)\gamma$, $s>1$, $\gamma<\gamma_\Gamma$, $\gamma_\Gamma>0$ as in \eqref{EqMaxWeight}. Then there is $h_0>0$ such that for $h<h_0$, $|\Im z|<h\gamma$,
\begin{equation}
\label{EqNHypSemiGlobal}
  \| u\|_{H^s_h}\lesssim h^{-2}\| \tilde P_0 u\|_{H^{s-m+1}_h},
\end{equation}
with the implied constant and $h_0$ uniform in $v_0$ with $|v_0|\leq C_0$.
\end{thm}

\begin{proof}
This is immediate from piecing together the semiclassical propagation estimates from radial points (which is where $s>1/2+\sup(\tilde\beta)\gamma$ is used, see the corresponding statement in the b-setting given in \cite[Proposition~2.1, Footnote~20]{HintzVasySemilinear}) through $\tilde\Gamma$, using Theorem~\ref{ThmNHypSemiMic}, which is where $\gamma<\gamma_\Gamma$ is used and where $h^{-2}$, rather than $h^{-1}$, is obtained for the right hand side, to $H_2\cap X$, which is where $s>1$ is used.

An alternative proof would be using Dyatlov's setting \cite{DyatlovNormally} directly, together with the gluing of Datchev and Vasy \cite{DatchevVasyGluing}, exactly as described in \cite[Theorem~2.17]{VasyMicroKerrdS}.
\end{proof}

Going back to the operator $P_0(v_0)$ satisfying the conditions stated at the beginning of this section, and under the additional assumption of uniform normal hyperbolicity as explained above, we can now obtain partial expansions of solutions to $Pu=f$ at infinity, i.e.\ at $X$:

\begin{thm}
\label{ThmExpansion}
  (Cf.\ \cite[Theorem~7.9]{HintzQuasilinearDS}.) Let $0<\alpha<\min(1,\gamma_\Gamma)$. Suppose $P$ has a simple rank $1$ resonance at $0$ with resonant state $1$, and that all other resonances have imaginary part less than $-\alpha$. Let $\tilde s>n/2+6$, $s_0>\max(n/2+1/2,1+\sup(\tilde\beta)\alpha)$, and assume $s_0\leq s\leq\tilde s-4$. Let $0\neq r\leq\alpha$. Then any solution $u\in\Hb^{s+4,r_0}(\Omega)^{\bullet,-}$ of $Pu=f$ with $f\in\Hb^{s+3,r}(\Omega)^{\bullet,-}$ satisfies $u\in\cX^{s',r}$ with $s'=s+4$ for $r<0$ and $s'=s$ for $r>0$, and the following tame estimate holds:
  \begin{align*}
    \|u\|_{\cX^{s',r}}&\lesssim \|f\|_{\Hb^{s+3,r}(\Omega)^{\bullet,-}} + \|u\|_{\Hb^{s+4,r_0}(\Omega)^{\bullet,-}} \\
	  &\hspace{1cm} + (\|f\|_{\Hb^{s_0,r}(\Omega)^{\bullet,-}}+\|u\|_{\Hb^{s_0+1,r_0}(\Omega)^{\bullet,-}})\|v\|_{\cX^{\tilde s,\alpha}}.
  \end{align*}
\end{thm}
\begin{proof}
  The proof works in the same way as in the reference by an iterative argument that consists of rewriting $Pu=f$ as $N(P)u=f-(P-N(P))u$ and employing a contour deformation argument, see \cite[Lemma~3.1]{VasyMicroKerrdS} (which uses high-energy estimates for the inverse normal operator family $\widehat P(\sigma)^{-1}$ and the location of resonances, i.e.\ of the poles of this family), to improve on the decay of $u$ by $\alpha$ in each step, but losing an order of differentiability as we are treating $P-N(P)$ as an error term; using tame microlocal regularity for the equation $Pu=f$, Corollary~\ref{CorWavePropSing}, one can regain this loss. We obtain $u\in\Hb^{s+4,r}$ after a finite number of iterations in case $r<0$, and $u\in\Hb^{s+4,r_0}$ for all $r_0<0$ in case $r>0$.

  Assuming we are in the latter case, the next step of the iteration gives a partial expansion $u=c+u'$ with $c\in\C$ (identified, as before, with $c\chi $, where $\chi$ is a smooth cutoff near the boundary) and $u'\in\Hb^{s+2,r'}$ for any $r'$ satisfying $r'\leq r$ and $r'<\alpha$; here, we need $0<\alpha<\gamma_\Gamma$ so that the normally hyperbolic trapping estimate \eqref{EqNHypSemiGlobal} holds with $\gamma>\alpha$, with loss of two derivatives. If $r=\alpha$, we can use this information to deduce
  \[
    N(P)u=f-(P-N(P))u=f-\tilde f,\quad \tilde f\in \Hb^{\tilde s,\alpha}+\Hb^{s,r'+\alpha}\subset\Hb^{s,r},
  \]
  which implies that the expansion $u=c+u'$ in fact holds with the membership $u'\in\Hb^{s,r}$; notice the improvement in the weight. Therefore, $u\in\cX^{s,r}$, finishing the proof.
\end{proof}

Pipelining this result with the existence of solutions, Lemma~\ref{LemmaWaveGlobalSolve}, we therefore obtain:
\begin{thm}
\label{ThmTameSolutionOp}
  Under the assumptions of Theorem~\ref{ThmExpansion} with $r>0$ and $s>n/2+2$, define the space
  \[
    \cY^{s,r}=\{u\in\cX^{s,r}\colon Pu\in\Hb^{s+3,r}(\Omega)^{\bullet,-}\}.
  \]
  Then the operator $P\colon\cY^{s,r}\to\Hb^{s+3,r}(\Omega)^{\bullet,-}$ has a continuous inverse $S$ that satisfies the tame estimate
  \begin{equation}
  \label{EqTameSolOp}
    \|Sf\|_{\cX^{s,r}}\leq C(s,\|v\|_{\cX^{s_0,\alpha}})(\|f\|_{\Hb^{s+3,r}(\Omega)^{\bullet,-}}+\|f\|_{\Hb^{s_0,r}(\Omega)^{\bullet,-}}\|v\|_{\cX^{s+4,\alpha}}).
  \end{equation}
\end{thm}

\subsection{Solving quasilinear wave equations}
\label{SubsecSolve}

We continue to work in the setting of the previous section. With the tame forward solution operator constructed in Theorem~\ref{ThmTameSolutionOp} in our hands, we are now in a position to use a Nash-Moser implicit function theorem to solve quasilinear wave equations. We use the following simple form of Nash-Moser, given in \cite{SaintRaymondNashMoser}:

\begin{thm}
\label{ThmNashMoser}
  Let $(B^s,|\cdot|_s)$ and $(\bB^s,\|\cdot\|_s)$ be Banach spaces for $s\geq 0$ with $B^s\subset B^t$ and indeed $|v|_t\leq|v|_s$ for $s\geq t$, likewise for $\bB^*$ and $\|\cdot\|_*$; put $B^\infty=\bigcap_s B^s$ and similarly $\bB^\infty=\bigcap_s\bB^s$. Assume there are smoothing operators $(S_\theta)_{\theta>1}\colon B^\infty\to B^\infty$ satisfying for every $v\in B^\infty$, $\theta>1$ and $s,t\geq 0$:
  \begin{equation}
  \label{EqSmoothing}
  \begin{split}
    |S_\theta v|_s \leq C_{s,t}\theta^{s-t}|v|_t \tn{ if } s\geq t, \\
	|v-S_\theta v|_s \leq C_{s,t}\theta^{s-t}|v|_t \tn{ if } s\leq t.
  \end{split}
  \end{equation}

  Let $\phi\colon B^\infty\to\bB^\infty$ be a $C^2$ map, and assume that there exist $u_0\in B^\infty$, $d\in\N$, $\delta>0$ and constants $C_1,C_2$ and $(C_s)_{s\geq d}$ such that for any $u,v,w\in B^\infty$,
  \begin{equation}
  \label{EqCondMapCont}
    |u-u_0|_{3d}<\delta \Rightarrow
	  \begin{cases}
	    \forall s\geq d,\quad \|\phi(u)\|_s\leq C_s(1+|u|_{s+d}), \\
		\|\phi'(u)v\|_{2d}\leq C_1|v|_{3d}, \\
		\|\phi''(u)(v,w)\|_{2d}\leq C_2|v|_{3d}|w|_{3d}.
	  \end{cases}
  \end{equation}
  Moreover, assume that for every $u\in B^\infty$ with $|u-u_0|_{3d}<\delta$ there exists an operator $\psi(u)\colon\bB^\infty\to B^\infty$ satisfying
  \[
    \phi'(u)\psi(u)h=h
  \]
  and the tame estimate
  \begin{equation}
  \label{EqCondTameSol}
    |\psi(u)h|_s\leq C_s(\|h\|_{s+d}+|u|_{s+d}\|h\|_{2d}),\quad s\geq d,
  \end{equation}
  for all $h\in\bB^\infty$. Then if $\|\phi(u_0)\|_{2d}$ is sufficiently small depending on $\delta,|u_0|_D$ and $(C_s)_{s\leq D}$, where $D=16d^2+43d+24$, there exists $u\in B^\infty$ such that $\phi(u)=0$.
\end{thm}

To apply this in our setting, we let $B^s=\cX^{s,\alpha}(\Omega)=\C\oplus\Hb^{s,\alpha}(\Omega)^{\bullet,-}$ and $\bB^s=\Hb^{s,\alpha}(\Omega)^{\bullet,-}$ with the corresponding norms; $\phi(u)$ will be the quasilinear equation, with implicit dependence on the forcing term. We now construct the smoothing operators $S_\theta$; we may assume, using a partition of unity, that $\Omega$ is the closure of an open subset of $\Rnhalf$, say $\Omega=\Omega(1)$, where we let $\Omega(x_0)=\{x\leq x_0,|y|\leq 1\}$. Then there are bounded extension and restriction operators
\[
  E\colon\Hb^{s,\alpha}(\Omega)^{\bullet,-}\to\Hb^{s,\alpha}(\Rnhalf),\quad R\colon\Hb^{s,\alpha}(\Rnhalf)\to\Hb^{s,\alpha}(\Omega)^{-,-},
\]
for $s\geq 0$; the operator $E$ can be constructed such that $\supp Ev\subset\{x\leq 1\}$ for $v\in\Hb^{s,\alpha}(\Omega)^{\bullet,-}$. If we then define for $\theta>1$ and $v=(c,u)\in\cX^{s,\alpha}$:
\[
  S^1_\theta v=(c,RS'_\theta Ev),
\]
where $S'_\theta$ is a smoothing operator on $\Rnhalf$ with properties as in \eqref{EqSmoothing}, then $S^1_\theta$ satisfies \eqref{EqSmoothing} in view of $RE$ being the identity on $\Hb^{s,\alpha}(\Omega)^{\bullet,-}$ if the norms on the left hand side are understood to be $\Hb^{s,\alpha}(\Rnhalf)$-norms. However, note that $S^1_\theta$ does not map $\cX^{\infty,\alpha}$ into itself, since smoothing operators such as $S'_\theta$ enlarge supports; we will thus need to modify $S^1_\theta$ below to obtain the operators $S_\theta$. In order to construct $S'_\theta$ on weighted b-Sobolev spaces $\Hb^{s,\alpha}$, it suffices by conjugation by the weight to construct it on the unweighted spaces $\Hb^s$; then, by a logarithmic change of coordinates, we only need to construct the smoothing operator $\tilde S_\theta$ on the standard Sobolev spaces $H^s(\R^n)$, which we will do in Lemma~\ref{LemmaSmoothing} below. In order to deal with the issue of $S^1_\theta$ enlarging supports, we will define $\tilde S_\theta$ such that
\[
  v\in\CIc(\R_{x',y'}^n),\ \supp v\subset\{x'\leq 0\}\ \Rightarrow\ \supp\tilde S_\theta v\subset\{x'\leq\theta^{-1/2}\}.
\]
In particular, when one undoes the logarithmic change of coordinates, this implies
\[
  S^1_\theta\colon\cX^{s,\alpha}(\Omega(1))\to\cX^{s,\alpha}\left(\Omega\bigl(\exp(\theta^{-1/2})\bigr)\right);
\]
more generally, with $D_\lambda$ denoting dilations $D_\lambda(x,y)=(\lambda x,y)$ on $\Rnhalf$, we have
\begin{equation}
\label{EqSmoothingSupp}
  S^\lambda_\theta:=(D_\lambda^{-1})^* S^1_\theta(D_\lambda)^*\colon\cX^{s,\alpha}(\Omega(\lambda))\to\cX^{s,\alpha}\left(\Omega\bigl(\lambda\exp(\theta^{-1/2})\bigr)\right),\quad \lambda>0,
\end{equation}
with the operator norm independent of $\lambda$ near $1$. Now, in our application of Theorem~\ref{ThmNashMoser}, we will have
\[
  \phi\colon\cX^{\infty,\alpha}(\Omega(x_0))\to\Hb^{\infty,\alpha}(\Omega(x_0))^{\bullet,-} \tn{ for all }x_0 \tn{ near }1,
\]
and correspondingly we will have forward solution operators $\psi$ going in the reverse direction, with all relevant constants being uniform in $x_0$. Looking at the proof of Theorem~\ref{ThmNashMoser} in \cite{SaintRaymondNashMoser}, one only uses the smoothing operator $S_{\theta_k}$ with $\theta_k=\theta_0^{(5/4)^k}$ in the $k$-th step of the iteration, with $\theta_0$ chosen sufficiently large; in our situation, where we have \eqref{EqSmoothingSupp}, we can therefore use the smoothing operator
\[
  S_{\theta_k}:=S^{\lambda_k}_{\theta_k},\quad\lambda_k=\exp\left(\sum_{j=0}^{k-1} \theta_j^{-1/2}\right)
\]
in the $k$-th iteration step. Note that, for $\theta_0$ large, we have
\[
  1=\lambda_0\leq\lambda_1\leq\cdots\leq\lambda_\infty=\exp\left(\sum_{j=0}^\infty\theta_j^{-1/2}\right)\leq 1+2\theta_0^{-1/2}.
\]
The solution $u$ to $\phi(u)=0$, obtained as a limit of an iterative scheme (see \cite[Lemma~1]{SaintRaymondNashMoser}), therefore is an element of $\cX^{s,\alpha}(\Omega(\lambda_\infty))$. Taking the hyperbolic nature of the PDE $\phi(u)=0$ into account once more, it will then, in our concrete setting, be easy to conclude that in fact $u\in\cX^{s,\alpha}(\Omega)$.

We now construct the smoothing operators on $\R^n$; the first step of the argument follows the Appendix of \cite{SaintRaymondNashMoser}.

\begin{lemma}
\label{LemmaSmoothing}
  There is a family $(\tilde S_\theta)_{\theta>1}$ of operators on $H^\infty(\R^n)$ satisfying
  \begin{gather}
  \label{EqRnSmoothing1} \|\tilde S_\theta v\|_s \leq C_{s,t}\theta^{s-t}\|v\|_t \tn{ if } s\geq t\geq 0, \\
  \label{EqRnSmoothing2} \|v-\tilde S_\theta v\|_s \leq C_{s,t}\theta^{s-t}\|v\|_t \tn{ if } 0\leq s\leq t, \\
  \label{EqRnSmoothing3} \supp\tilde S_\theta v\subseteq\{x_1\leq\theta^{-1/2}\}
  \end{gather}
  for all $v\in H^\infty(\R^n)$ with $\supp v\subseteq H:=\{x_1\leq 0\}$. Here $\|\cdot\|_s$ denotes the $H^s(\R^n)$-norm, and we write $x=(x_1,x')\in\R^n$.
\end{lemma}
\begin{proof}
  Choose $\chi=\chi_1(x_1)\chi_2(x')\in S(\R^n)$ with $\chi_1\in S(\R),\chi_2\in S(\R^{n-1})$ so that the Fourier transform $\hat\chi$ is identically $1$ near $0$; put $\chi_\theta(z)=\theta^n\chi(\theta z)$ and define the operator $C_\theta v=\chi_\theta*v$. Then $(C_\theta v)\ftrans=\widehat{\chi_\theta}\hat v$ with $\widehat{\chi_\theta}(\xi)=\hat\chi(\xi/\theta)$, therefore \eqref{EqRnSmoothing1} holds for $C_\theta$ in place of $\tilde S_\theta$ with constants $C'_{s,t}$ since $\hat\chi$ decays super-polynomially, and \eqref{EqRnSmoothing2} holds for $C_\theta$ in place of $\tilde S_\theta$ with constants $C'_{s,t}$ since $1-\hat\chi(\xi)$ vanishes at $\xi=0$ with all derivatives.

  Next, let $\psi\in\CI(\R^n)$ be a smooth function depending only on $x_1$, i.e.\ $\psi=\psi(x_1)$, so that $\psi(x_1)\equiv 1$ for $x_1\in(-\infty,1/2]$, $\psi(x_1)\equiv 0$ for $x_1\in[1,\infty)$, and $0\leq\psi\leq 1$. Put $\psi_\theta(x_1,x')=\psi(\theta x_1,x')$, and define
  \[
    \tilde S_\theta v:=\psi_{\theta^{1/2}}C_\theta v.
  \]
  Condition \eqref{EqRnSmoothing3} is satisfied by the support assumption on $\psi$. Let $\varphi=1-\psi$ and $\varphi_\theta=1-\psi_\theta$. To prove the other two conditions, we use the estimate
  \begin{equation}
  \label{EqCutoffOutside}
    \|\varphi_{\theta^{1/2}}C_\theta v\|_s\leq C''_{s,N}\theta^{-N}\|v\|_{L^2},\quad \supp v\subset H,\ s,N\geq 0,
  \end{equation}
  which we will establish below. Taking this for granted, we obtain for $v$ with $\supp v\subset H$:
  \[
    \|\tilde S_\theta v\|_s \leq \|C_\theta v\|_s + \|\varphi_{\theta^{1/2}}C_\theta v\|_s\leq C'_{s,t}\theta^{s-t}\|v\|_t + C''_{s,0}\|v\|_0
  \]
  for $s\geq t\geq 0$, which is the estimate \eqref{EqRnSmoothing1}; and \eqref{EqRnSmoothing2} follows from
  \[
    \|v-\tilde S_\theta v\|_s \leq \|v-C_\theta v\|_s + \|\varphi_{\theta^{1/2}}C_\theta v\|_s \leq C'_{s,t}\theta^{s-t}\|v\|_t+C''_{s,t-s}\theta^{s-t}\|v\|_0
  \]
  for $0\leq s\leq t$.

  We now prove \eqref{EqCutoffOutside} for $s\in\N_0$. For multiindices $\alpha=(\alpha_1,\alpha')$ with $|\alpha|\leq s$, we have for $v$ with $\supp v\subset H$ and for $(x_1,x')\in\supp\varphi_{\theta^{1/2}}C_\theta v$, which in particular implies $x_1\geq 1/(2\theta^{1/2})$:
  \begin{align*}
    \pa^\alpha&(\varphi_{\theta^{1/2}}C_\theta v)(x_1,x')=\sum_{j=0}^{\alpha_1}{\alpha_1\choose j}\theta^{(\alpha_1-j)/2}\varphi^{(\alpha_1-j)}(\theta^{1/2}x_1) \\
	  &\times\iint_{y_1\geq 1/(2\theta^{1/2})} \theta^{n+j+|\alpha'|}\chi_1^{(j)}(\theta y_1)\chi_2^{(\alpha')}(\theta y') v(x_1-y_1,x'-y')\,dy_1\,dy',
  \end{align*}
  thus
  \[
    \|\pa^\alpha(\varphi_{\theta^{1/2}}C_\theta v)\|_{L^2}\leq C_s\theta^{n+s}\|\check\chi_\theta\|_{L^1}\|v\|_{L^2},
  \]
  where
  \[
    \check\chi_\theta(x_1,x')=
	  \begin{cases}
	    0, & x_1<1/(2\theta^{1/2}), \\
	    \sum_{j=0}^{\alpha_1} |\chi_1^{(j)}(\theta x_1)\chi_2^{(\alpha')}(\theta x')| & \tn{otherwise}.
	  \end{cases}
  \]
  But $\|\check\chi_\theta\|_{L^1}\leq C_{N,s}\theta^{-N}$ for all $N$: Indeed, this reduces to the statement that for a fixed $\chi_0\in S(\R)$, one has
  \[
    \int_{1/(2\theta^{1/2})}^\infty |\chi_0(\theta x)|\,dx \leq C_N \int_{\theta^{-1/2}}^\infty (\theta x)^{-2N+1}\,dx=C'_N\theta^{-N}.
  \]
  Hence, we obtain \eqref{EqCutoffOutside}, and the proof is complete.
\end{proof}

We now combine Theorem~\ref{ThmTameSolutionOp}, giving the existence of tame forward solution operators, with Theorem~\ref{ThmNashMoser}, in the extended form described above, to solve quasilinear wave equations.

\begin{thm}
\label{ThmQuasilinearWave}
  Let $N\in\N$ and $c_k\in\CI(\C;\R)$, $g_k\in(\CI+\Hb^\infty)(M;\Sym^2\Tb M)$ for $1\leq k\leq N$; define the map $g\colon\cX^{s,\alpha}\to(\CI+\Hb^{s,\alpha})(M;\Sym^2\Tb M)$ by $g(u)=\sum_{k=1}^N c_k(u)g_k$ and assume that $\Box_{g(0)}$ satisfies the assumptions of \S\ref{SubsecForward} and of Theorem~\ref{ThmTameSolutionOp}; let us moreover assume that for small $u_0$, the metric $g(u_0)$ has uniform normally hyperbolic trapping in the sense described in \S\ref{SubsecTrapping}. Further, let $N'\in\N$ and define
  \begin{equation}
  \label{EqNonlinearQ}
    q(u,\bdiff u)=\sum_{j=1}^{N'} u^{e_j}\prod_{k=1}^{N_j} X_{jk}u,
  \end{equation}
  where
  \[
    e_j,N_j\in\N_0,N_j\geq 1, N_j+e_j\geq 2,X_{jk}\in(\CI+\Hb^\infty)\Vb.
  \]
  Then there exists $C_f>0$ such that for all forcing terms $f\in\Hb^{\infty,\alpha}(\Omega)^{\bullet,-}$ satisfying $\|f\|_{\Hb^{\max(\derivsU,n+5),\alpha}(\Omega)^{\bullet,-}}\leq C_f$, the equation
  \begin{equation}
  \label{EqQuasilinearPDE}
    \Box_{g(u)}u=f+q(u,\bdiff u)
  \end{equation}
  has a unique solution $u\in\cX^{\infty,\alpha}$.

  If more generally $g(u,\bdiff u)=\sum_{k=1}^N c_k(u,X_1 u,\ldots,X_L u)g_k$, where $X_1,\ldots,X_L\in\Vb(M)$ and $c_k\in\CI(\C^{1+L};\R)$, with $g(u_0,0)$ uniformly normally hyperbolic for small $u_0$, then there exists $C_f>0$ such that for all forcing terms $f\in\Hb^{\infty,\alpha}(\Omega)^{\bullet,-}$ satisfying $\|f\|_{\Hb^{\max(\derivsUdU,n+5),\alpha}(\Omega)^{\bullet,-}}\leq C_f$, the equation
  \begin{equation}
  \label{EqQuasilinearPDE2}
    \Box_{g(u,\bdiff u)}u=f+q(u,\bdiff u)
  \end{equation}
  has a unique solution $u\in\cX^{\infty,\alpha}$.

  In both cases, if one instead assumes that the trapping for $g(0,0)$ is $r$-normally hyperbolic for every $r$, then for every fixed $s\in\R$, one has solvability with $u\in\cX^{s,\alpha}$ provided the norm bound $C_f=C_f(s)>0$ on $f$ (on the same spaces as before) is sufficiently small.
\end{thm}

We point out that if all vector fields, coefficients and data are real, then $u$ is real-valued as well.

\begin{proof}[Proof of Theorem~\ref{ThmQuasilinearWave}]
  We write $|\cdot|_s$ for the $\cX^{s,\alpha}$-norm and $\|\cdot\|_s$ for the $\Hb^{s,\alpha}$-norm. (For brevity, we do not specify the underlying set, which, in the notation of \S\ref{SubsecForward}, is $\ft^{-1}_1([-\lambda,\infty))\cap\ft^{-1}_2([0,\infty))$ for varying $\lambda\geq 0$.) We define the map
  \[
    \phi(u;f)=\Box_{g(u)}u-q(u,\bdiff u)-f
  \]
  and check that it satisfies the conditions of Theorem~\ref{ThmNashMoser} with $u_0=0$. From the definition of $\Box_{g(u)}$ and the tame estimates for products, reciprocals and compositions, Corollary~\ref{CorHbModule} and Propositions~\ref{PropHbRec} and \ref{PropCompWithSmooth2}, we obtain
  \[
    \|\phi(u;f)\|_s \leq \|f\|_s + C(|u|_{s_0+2})(1+|u|_{s+2}),\quad s\geq s_0>n/2+1,
  \]
  thus the first estimate of \eqref{EqCondMapCont} for $3d\geq s_0+2$, $d\geq s_0$, $d\geq 2$. Next, we have $\phi'(u;f)v=\bigl(\Box_{g(u)}+L(u,\bdiff u)\bigr)v$, where the first order b-differential operator $L$ is of the form
  \begin{equation}
  \label{EqLinearizationExpr}
    L=\sum_{|\beta|\leq 1}\biggl(\sum_{1\leq|\alpha|\leq 2}a_{\alpha\beta}(u,\bdiff u)\Db^\alpha u\biggr)\Db^\beta + \sum_{|\beta|=1}a_\beta(u,\bdiff u)u\,\Db^\beta,
  \end{equation}
  with the second sum capturing one term of the linearization of terms $u^{e_j}X_{j1}u$ in $q$ (i.e.\ terms for which $N_j=1$). In particular,
  \begin{equation}
  \label{EqPhiPrime}
    \phi'(u;f)=P_0(u_0)+\tilde P(u,\Db u,\Db^2 u),
  \end{equation}
  where $P_0\in\Diffb^2$ and $\tilde P\in\Hb^{s-1,\alpha}\Diffb^2+\Hb^{s-2,\alpha}\Diffb^1$ for $u\in\cX^{s,\alpha}$; in fact, the leading order coefficient has regularity $\Hb^{s,\alpha}$, but we give up one derivative for direct comparison with \eqref{EqP0Form}. Therefore,
  \[
    \|\phi'(u;f)v\|_s \leq C(|u|_{s+2})|v|_{s+2},\quad s>n/2+1,
  \]
  which gives the second estimate of \eqref{EqCondMapCont} for $2d>n/2+1$ and $3d\geq 2d+2$. Next, we observe that $\phi''(u;f)(v,w)$ is bilinear in $v,w$, involves up to two b-derivatives of each $v$ and $w$, and the coefficients depend on up to two b-derivatives of $u$, thus
  \[
    \|\phi''(u;f)(v,w)\|_s \leq C(|u|_{s+2})|v|_{s+2}|w|_{s+2},\quad s>n/2+1,
  \]
  which gives the third estimate of \eqref{EqCondMapCont} for $3d>n/2+3$, $3d\geq 2d+2$. In summary, we obtain \eqref{EqCondMapCont} for integer $d>n/2+1$.

  Finally, we determine $d$ so that we have the tame estimate \eqref{EqCondTameSol}: Given $u\in\cX^{s+5,\alpha}$, we can write $\phi'(u;f)$ as in \eqref{EqPhiPrime}, with $P_0\in\Diffb^2$ and $\tilde P\in\Hb^{s+4,\alpha}\Diffb^2+\Hb^{s+3,\alpha}\Diffb^1$; hence, by Theorem~\ref{ThmTameSolutionOp}, we obtain a solution operator
  \begin{equation}
  \label{EqPsiTameProof}
  \begin{gathered}
    \psi(u;f)\colon\Hb^{s+3,\alpha}\to\cX^{s,\alpha}, \\
    |\psi(u;f)v|_s\leq C(s,|u|_{s_0})(\|v\|_{s+3}+\|v\|_{s_0}|u|_{s+5}),
  \end{gathered}
  \end{equation}
  where $s,s_0>n/2+2$, provided $|u|_{s_0}$ is small enough so that all dynamical and geometric hypotheses hold for $\phi'(u;f)$. Notice that the subprincipal term of $\phi'(u;f)$ can differ from that of $\Box_{g(0)}$ by terms of the form $a(u_0)u_0\Db^\beta$, $a\in\CI$, $|\beta|=1$, see \eqref{EqLinearizationExpr}; however, since such terms eliminate constants, the simple rank $1$ resonance at $0$ with resonant state $1$ does not change; and moreover such terms are \emph{small} because of the factor $u_0$, hence high energy estimates still hold in a (possibly slightly smaller) strip in the analytic continuation, see the remark below \cite[Theorem~1]{DyatlovNormally}. Since $s_0$ is independent of $s$, we have \eqref{EqPsiTameProof} for all $s>n/2+2$, in particular $\psi(u;f)\colon\Hb^{\infty,\alpha}\to\cX^{\infty,\alpha}$. Now, \eqref{EqPsiTameProof} implies that \eqref{EqCondTameSol} holds for $d>n/2+2$, $d\geq 5$, so we need to control $\max(\derivsU,n+5)$ derivatives of $f$.

  Thus, we can apply Nash-Moser iteration, Theorem~\ref{ThmNashMoser}, to obtain a solution $u\in\cX^{s,\alpha}$ of the PDE \eqref{EqQuasilinearPDE}, with the caveat that $u$ is a priori supported on a space slightly larger than $\Omega$, in the sense that $\supp u$ might not be contained in the future of $H_1$. (The support generally extends past $H_2$, but this does not concern us since we are always restricting ourselves to the region bounded by the connected components of $H_2$.) However, local uniqueness for quasilinear hyperbolic equations, see e.g.\ \cite[\S{16.3}]{TaylorPDE}, implies that $u$ in fact vanishes in the past of $H_1$, hence is supported in $\Omega$, and that $u$ is the unique solution of \eqref{EqQuasilinearPDE}, finishing the proof of the first part.

  The proof of the second part proceeds analogously; now the operator $L$ in \eqref{EqLinearizationExpr} involves an additional term
  \[
    \sum_{|\beta|=2}\biggl(\sum_{|\alpha|=1}a'_{\alpha\beta}(u,\bdiff u)\Db^\alpha u\biggr)\Db^\beta
  \]
  since the first order terms of the expression of $\Box_{g(u,\bdiff u)}$ involve derivatives of the metric, hence second derivatives of $u$; notice however that the additional term can be put into $\tilde P$ due to the factor of $du$. Therefore, if now $u\in\cX^{s+6,\alpha}$, then $\tilde P$ in \eqref{EqPhiPrime} lies in the space $\Hb^{s+4,\alpha}\Diffb^2\subset\Hb^{s+4,\alpha}\Diffb^2+\Hb^{s+3,\alpha}\Diffb^1$; the rest of the argument then goes through, the only change being that the $|u|_{s+5}$ norm in \eqref{EqPsiTameProof} needs to be replaced by $|u|_{s+6}$, hence the new requirement $d\geq 6$. Therefore, we need to control $\max(\derivsUdU,n+5)$ derivatives of $f$ for the second part.

  Lastly, we recall that $r$-normal hyperbolicity for any finite $r$ is structurally stable, hence for any fixed $r\in\R$, the metric $g(u_0,0)$ has normally hyperbolic trapping for sufficiently small $u_0$, with stable/unstable manifolds and trapped sets of class $\mc C^r$, depending continuously on $u_0$. Since the estimates at normally hyperbolic trapping on any fixed Sobolev space only require a finite number of regularity (since they then rely on the finiteness of only finitely many norms and seminorms of the manifolds and operators involved) and since Nash-Moser iteration only requires a finite number of derivatives if one is merely interested in obtaining a solution with finite b-regularity (see also Remark~\ref{RmkFiniteRegularity} below), the above arguments imply solvability in $\cX^{s,\alpha}$ provided the norm of $f$ in the stated spaces is small enough, depending on $s$: Indeed, if the norm of $f$ is $\leq C_f$, then the norm of every iterate $u_k$ in the Nash-Moser scheme will be bounded by a constant multiple of $C_f$ in a suitable (fixed) Sobolev space, hence the corresponding metrics $g(u_k,\bdiff u_k)$, restricted to the boundary at future infinity $\Omega\cap\pa M$, will be close to $g(0,0)$ and thus have $r$-normally hyperbolic trapping, where $r$ can be any fixed arbitrarily large real number if $C_f>0$ is sufficiently small, depending on $r$. Therefore, given $s$, we merely need to pick $C_f>0$ so small that the corresponding regularity $\mc C^r$ of the trapping is sufficient for all estimates at the trapping to hold on all (finite regularity!) Sobolev spaces which appear (explicitly and implicitly) in the above proof.
\end{proof}

\begin{rmk}
\label{RmkStrongerThanDSResult}
  In the asymptotically de Sitter setting considered in \cite{HintzQuasilinearDS}, the above Theorem extends \cite[Theorem~8.8]{HintzQuasilinearDS} (at the cost of requiring the control of more derivatives) since we allow the dependence of the metric $g(u,\bdiff u)$ on $\bdiff u$ as well.
\end{rmk}

\begin{rmk}
\label{RmkFiniteRegularity}
  An inspection of the proof of the abstract Nash-Moser theorem \ref{ThmNashMoser} in \cite{SaintRaymondNashMoser} shows that there are constants $C$ and $s_0$, depending only on the `loss of derivatives' $d$, such that the following holds: In order to obtain a solution $u\in\cX^{s,\alpha}$ for some finite $s\geq s_0$, it is sufficient to take $f\in\Hb^{Cs,\alpha}$, still assuming the norm of $f$ in the space indicated in the statement of Theorem~\ref{ThmQuasilinearWave} to be small.
\end{rmk}

Theorem~\ref{ThmQuasilinearWave} immediately implies the following result on Kerr-de Sitter space:

\begin{cor}
\label{CorQuasilinearKdS}
  Under the assumptions of Theorem~\ref{ThmQuasilinearWave}, the quasilinear wave equation \eqref{EqQuasilinearPDE}, resp.\ \eqref{EqQuasilinearPDE2}, on a 4-dimensional asymptotically Kerr-de Sitter space with $|a|\ll\bhm$ has a unique global solution in the space $\cX^{s,\alpha}$ if the $\Hb^{\derivsU,\alpha}(\Omega)^{\bullet,-}$-norm, resp.\ $\Hb^{\derivsUdU,\alpha}(\Omega)^{\bullet,-}$-norm, of the forcing term $f\in\Hb^{\infty,\alpha}(\Omega)^{\bullet,-}$ is sufficiently small.
\end{cor}
\begin{proof}
  For a verification of the dynamical assumptions for asymptotically Kerr-de Sitter spaces, we refer the reader to \cite[\S{6}]{VasyMicroKerrdS}; the resonances on the other hand were computed by Dyatlov \cite{DyatlovQNM}, and also by the authors using a perturbation argument \cite{HintzVasyKdsFormResonances}.
\end{proof}

\subsection{Solving quasilinear Klein-Gordon equations}
\label{SubsecSolveKG}

The only difference between wave and Klein-Gordon equations with mass $m$ (which is to be distinguished from the black hole mass $\bhm$) is that the resonance of the Klein-Gordon operator $\Box-m^2$ with largest imaginary part, which gives the leading order asymptotics, is no longer at $0$ for $m\neq 0$. Thus, if $\sigma_1,\sigma_2\in\C$ are the first two resonances of $\Box-m^2$, i.e.\ there are no further resonances $\sigma$ with $\Im\sigma\geq\Im\sigma_2$, assume
\[
  0<-\Im\sigma_1<r<-\Im\sigma_2,
\]
and assume moreover that the high energy estimates for the normal operator family of $\Box-m^2$ hold in $\Im\sigma\geq -r$, the only change in the statement of Theorem~\ref{ThmExpansion} for Klein-Gordon operators is that the conclusion now is $u\in\cX^{s,r}_{\sigma_1}$, where $\cX^{s,r}_{\sigma_1}=\C\oplus\Hb^{s,r}(\Omega)^{\bullet,-}$, with $(c,u')$ identified with $cx^{i\sigma_1}\chi+u'$ for a smooth cutoff $\chi$ near the boundary.

\begin{rmk}
  There are more cases of potential interest: If $r<-\Im\sigma_1$, we obtain $u\in\Hb^{s,r}(\Omega)^{\bullet,-}$; if $r<0$, the statement of Theorem~\ref{ThmExpansion} is unchanged; and if $\Im\sigma_1$ and $\Im\sigma_2$ are close enough together (including the case that $\sigma_1$ is a double resonance) and there is a spectral gap below $\sigma_2$, one gets two terms in the expansion of $u$. For brevity, we only explain one scenario here. See also the related discussion in \cite[\S{8.4}]{HintzQuasilinearDS}.
\end{rmk}

We thus obtain the following adapted version of Theorem~\ref{ThmTameSolutionOp}:

\begin{thm}
\label{ThmTameSolutionOpKG}
  In the notation of \S\ref{SubsecSolve}, under the above assumptions and for $s>n/2+2$, define the space
  \[
    \cY_{\sigma_1}^{s,r}=\{u\in\cX_{\sigma_1}^{s,r}\colon Pu\in\Hb^{s+3,r}(\Omega)^{\bullet,-}\}.
  \]
  Then the operator $P\colon\cY^{s,r}\to\Hb^{s+3,r}(\Omega)^{\bullet,-}$ has a continuous inverse $S$ that satisfies the tame estimate
  \begin{equation}
  \label{EqTameSolOpKG}
    \|Sf\|_{\cX^{s,r}_{\sigma_1}}\leq C(s,\|v\|_{\cX^{s_0,\alpha}_{\sigma_1}})(\|f\|_{\Hb^{s+3,r}(\Omega)^{\bullet,-}}+\|f\|_{\Hb^{s_0,r}(\Omega)^{\bullet,-}}\|v\|_{\cX^{s+4,\alpha}_{\sigma_1}}).
  \end{equation}
\end{thm}

This immediately gives:

\begin{thm}
\label{ThmQuasilinearKG}
  Under the above assumptions and the assumption $\alpha<-2\Im\sigma_1$, let $N,N'\in\N$ and $c_k\in\CI(\C;\R)$, $g_k\in(\CI+\Hb^\infty)(M;\Sym^2\Tb M)$ for $1\leq k\leq N$; define the map $g\colon\cX_{\sigma_1}^{s,\alpha}\to(\CI+\Hb^{s,\alpha})(M;\Sym^2\Tb M)$ by $g(u)=\sum_{k=1}^N c_k(u)g_k$ and assume that $\Box_{g(0)}$ satisfies the assumptions of \S\ref{SubsecForward} and of Theorem~\ref{ThmTameSolutionOpKG}. Moreover, define
  \[
    q(u,\bdiff u)=\sum_{j=1}^{N'} a_ju^{e_j}\prod_{k=1}^{N_j} X_{jk}u,
  \]
  where
  \[
	e_j,N_j\in\N_0,e_j+N_j\geq 2, a_j\in\CI, X_{jk}\in(\CI+\Hb^\infty)\Vb.
  \]
  Then there exists $C_f>0$ such that for all forcing terms $f\in\Hb^{\infty,\alpha}(\Omega)^{\bullet,-}$ satisfying $\|f\|_{\Hb^{\max(\derivsU,n+5),\alpha}(\Omega)^{\bullet,-}}\leq C_f$, the equation
  \begin{equation}
  \label{EqKGPDE}
    (\Box_{g(u)}-m^2)u=f+q(u,\bdiff u)
  \end{equation}
  has a unique solution $u\in\cX_{\sigma_1}^{\infty,\alpha}$.

  If more generally $g(u,\bdiff u)=\sum_{k=1}^N c_k(u,X_1 u,\ldots,X_L u)$, where $X_1,\ldots,X_L\in\Vb(M)$ and $c_k\in\CI(\C^{1+L};\R)$, then there exists $C_f>0$ such that for all forcing terms $f\in\Hb^{\infty,\alpha}(\Omega)^{\bullet,-}$ satisfying $\|f\|_{\Hb^{\max(\derivsUdU,n+5),\alpha}(\Omega)^{\bullet,-}}\leq C_f$, the equation
  \begin{equation}
  \label{EqKGPDE2}
    (\Box_{g(u,\bdiff u)}-m^2)u=f+q(u,\bdiff u)
  \end{equation}
  has a unique solution $u\in\cX^{\infty,\alpha}$.
\end{thm}

Together with Theorem~\ref{ThmQuasilinearWave}, this proves Theorem~\ref{ThmIntroWave}. 

\begin{proof}[Proof of Theorem~\ref{ThmQuasilinearKG}.]
  The proof proceeds as the proof of Theorem~\ref{ThmQuasilinearWave}. Notice that we allow the nonlinear term $q$ to be more general, the point being that firstly, any at least quadratic expression in $(u,\bdiff u)$ with $u\in\cX_{\sigma_1}^{s,\alpha}$ gives an element of $\Hb^{s,\alpha}$, and secondly, every element in $\cX_{\sigma_1}^{s,\alpha}$ vanishes at the boundary, thus the normal operator family of the linearization of $\Box_{g(u)}-m^2-q(u,\bdiff u)-f$ at any $u\in\cX_{\sigma_1}^{s,\alpha}$ is equal to the normal operator family of $\Box_{g(0)}-m^2$, for which one has high energy estimates by assumption.

  We point out that the trapping trivially is uniformly normally hyperbolic since the normal operator family $\hat N(\Box_{g(u,\bdiff u)})$ is in fact \emph{independent} of $u\in\cX_{\sigma_1}^{s,\alpha}$; hence one indeed obtains a solution in $\cX_{\sigma_1}^{\infty,\alpha}$.
\end{proof}

By \cite[Lemma~3.5]{HintzVasySemilinear}, the assumptions of Theorem~\ref{ThmQuasilinearKG} are satisfied on asymptotically Kerr-de Sitter spaces as long as the mass parameter $m$ is small:

\begin{cor}
\label{CorQuasilinearKGKdS}
  Under the assumptions of Theorem~\ref{ThmQuasilinearKG} and for $a$ and $m>0$ sufficiently small, the quasilinear Klein-Gordon equation \eqref{EqKGPDE}, resp.\ \eqref{EqKGPDE2}, on a 4-dimensional asymptotically Kerr-de Sitter space with angular momentum $a$ has a unique global smooth (i.e.\ conormal, in the space $\cX^{\infty,\alpha}_{\sigma_1}$) solution if the $\Hb^{\derivsU,\alpha}(\Omega)^{\bullet,-}$-norm, resp.\ $\Hb^{\derivsUdU,\alpha}(\Omega)^{\bullet,-}$-norm, of the forcing term $f\in\Hb^{\infty,\alpha}(\Omega)^{\bullet,-}$ is sufficiently small.
\end{cor}

\subsection{Proofs of Theorems~\ref{ThmIntroGeneralKG} and \ref{ThmIntroGeneralWave}}
\label{SubsecGeneralProofs}

Finally, following the same arguments as used in the previous section, we indicate how to prove the general theorems stated in the introduction. We continue to use, but need to generalize the setting considered in \S\ref{SubsecForward}: Namely, generalizing \eqref{EqP0Form}, we now allow $L$ to be any first order b-differential operator, and correspondingly need information on the skew-adjoint part of $P_0$; concretely, we define $\hat\beta$ at the (generalized) radial sets $L_\pm$, using the same notation as in \eqref{EqRadialParams}, by
\begin{equation}
\label{EqHatBeta}
  \sigma_{\bl,1}\Big(\frac{1}{2i}(P_0-P_0^*)\Big)\Big|_{L_\pm}=\pm\hat\beta\beta_0\rho.
\end{equation}
Moreover, at the trapped set $\Gamma=\Gamma^-\cup\Gamma^+$, we assume that
\begin{equation}
\label{EqP0ImagGamma}
  \sfe_1|_\Gamma<\numin/2,\quad \sfe_1=|\sigma|^{-1}\sigma_{\bl,1}\Big(\frac{1}{2i}(P_0-P_0^*)\Big),
\end{equation}
with $\numin$ the minimal normal expansion rate for the Hamilton flow of the principal symbol of $P_0$, and $\sigma$ the Mellin dual variable of $x$ after an identification of a collar neighborhood of $X$ in $M$ with $[0,\ep')_x\times X$; note that $\sigma$ is elliptic on $\Gamma$. Let $r_\thr$ be the threshold weight for the first part of Theorem~\ref{ThmTameHypTr}, i.e.\ $r_\thr=-\sup\sfe_1/c_\pa$ with $c_\pa$ as defined in \eqref{EqVx}.

Then Corollary~\ref{CorWavePropSing} holds in the current, more general setting, provided we assume $r<r_\thr$ and $s'>1+\sup_{L_\pm}(r\tilde\beta-\hat\beta)$. Likewise, we obtain the high energy estimates of Theorem~\ref{ThmNHypSemiMic} under the assumption $s>1/2+\sup_{L_\pm}(\gamma\tilde\beta-\hat\beta)$.

In order to generalize Theorem~\ref{ThmExpansion}, we first choose $0<r_+<1$ such that
\[
  (\sfe_1+r_+c_\pa)|_\Gamma < \numin/2,
\]
which holds for sufficiently small $r_+$ in view of \eqref{EqP0ImagGamma} by the compactness of $\Gamma$ in $\Sb^*M$. We moreover assume that there are no (nonzero) resonances in $\Im\sigma>-r_+$ in the case of Theorem~\ref{ThmIntroGeneralKG} (Theorem~\ref{ThmIntroGeneralWave}), and we assume further that $0<\alpha<r_+$. Then in the proof of Theorem~\ref{ThmExpansion}, ignoring the issue of threshold regularities at radial sets momentarily, we can use the contour shifting argument without loss of derivatives up to, but excluding, the weight $r_\thr$, corresponding to the contour of integration $\Im\sigma=-r_\thr$. Shifting the contour further down, we cannot use the non-smooth real principal type estimate at $\Gamma$ anymore and thus lose $2$ derivatives at each step; the total number of additional steps needed to shift the contour down to $\Im\sigma=-\alpha$ is easily seen to be at most
\[
  N=\max\left(0,\bigg\lceil\frac{\alpha-r_\thr}{\alpha}\bigg\rceil + 1\right),
\]
hence in order to have the final conclusion that $u$ has an expansion with remainder in $\Hb^{s,\alpha}$, we need to assume that $u$ initially is known to have regularity $\Hb^{s+2N,r_0}$ for any $r_0\in\R$, which in turn requires $\tilde s\geq s+2N$ and $f\in\Hb^{s+2N-1,r_0}$ for the first, lossless, part of the argument to work. Taking the regularity requirements at the radial sets into account, we further need to assume $s\geq s_0>\max(n+1/2,1+\sup(r\tilde\beta-\hat\beta))$. Under these assumptions, the proof of Theorem~\ref{ThmExpansion} applies, mutatis mutandis, to our current situation, and we obtain a tame solution operator as in Theorem~\ref{ThmTameSolutionOp}, which now loses $2N-1$ derivatives.

Thus, we can prove Theorems~\ref{ThmIntroGeneralKG} and \ref{ThmIntroGeneralWave} using the same arguments which we used in the proof of Theorem~\ref{ThmQuasilinearWave}; the `loss of derivatives' parameter $d$ now needs to satisfy the conditions
\begin{equation}
\label{EqGeneralDCond}
  d \geq 2N+3,\ d>n/2+6,\ d>1+\sup(r\tilde\beta-\hat\beta),
\end{equation}
with the first condition being the actual loss of derivatives, the second one coming from $s>n/2+6$ certainly being a high enough regularity for $\tilde s=s+2N$ to be $>n/2+6$, which is required for the application of the non-smooth microlocal regularity results, and the last condition being the threshold regularity (for the non-smooth estimates) at the radial sets.

\begin{rmk}
\label{RmkGeneralTheoremsMfds}
  We again point out that these theorems hold under general hypotheses, as explained in \S\ref{sec:general}; one merely needs to use the tame estimates \emph{on manifolds}, which hold by the discussion in Remarks~\ref{RmkRadialPointsMfds} and \ref{RmkNonsmoothTrappingMfds}.
\end{rmk}


\end{document}